\documentclass[leqno, letterpaper]{amsart}
\usepackage{amssymb, mathtools, amsmath, amsthm, amscd, mathrsfs, latexsym} 
\usepackage{tikz, tikz-cd, graphicx}
\usepackage[matrix,arrow,curve]{xy}
\usetikzlibrary{matrix,arrows,decorations.pathmorphing}

\usepackage[pdftex]{hyperref}
\hypersetup{
  colorlinks   = true, 
  urlcolor     = red, 
  linkcolor    = blue, 
  citecolor   =  gray
}

\newcommand{\op}[1]{\!\!\mathop{\rm ~#1}\nolimits}

\newcommand{\h}{\op{Hom}}

\newcommand{\p}{\partial}

\newcommand{\Ci}{C^{\infty}}

\newcommand{\N}{\mathbb{N}}
\newcommand{\Z}{\mathbb{Z}}

\mathchardef\za="710B  
\mathchardef\zb="710C  
\mathchardef\zg="710D  
\mathchardef\zd="710E  
\mathchardef\zve="710F 
\mathchardef\zz="7110  
\mathchardef\zh="7111  
\mathchardef\zy="7112 

\mathchardef\zi="7113  
\mathchardef\zk="7114  
\mathchardef\zl="7115  
\mathchardef\zm="7116  
\mathchardef\zn="7117  
\mathchardef\zx="7118  
\mathchardef\zp="7119  
\mathchardef\zr="711A  
\mathchardef\zs="711B  
\mathchardef\zt="711C  
\mathchardef\zu="711D  
\mathchardef\zf="711E 
\mathchardef\zq="711F  
\mathchardef\zc="7120  
\mathchardef\zw="7121  
\mathchardef\ze="7122  
\mathchardef\zvy="7123  
\mathchardef\zvw="7124  
\mathchardef\zvr="7125 
\mathchardef\zvs="7126 
\mathchardef\zvf="7127  
\mathchardef\zG="7000  
\mathchardef\zD="7001  
\mathchardef\zY="7002  
\mathchardef\zL="7003  
\mathchardef\zX="7004  
\mathchardef\zP="7005  
\mathchardef\zS="7006  
\mathchardef\zU="7007  
\mathchardef\zF="7008  
\mathchardef\zW="700A  

\newcommand{\cyclic}{\mathop{\kern0.9ex{{+}
\kern-2.15ex\raise-.25ex\hbox{\Large\hbox{$\circlearrowright$}}}}\limits}

\newcommand{\on}{\operatorname}
\newcommand{\mc}{\mathcal}

\newcommand{\be}{\begin{equation}}
\newcommand{\ee}{\end{equation}}

\newcommand{\NN}{{\mathbb N}}

\newcommand{\Ac}{{\mathcal A}}
\newcommand{\Bc}{{\mathcal B}}

\newcommand{\Dc}{{\mathcal D}}
\newcommand{\Ec}{{\mathcal E}}
\newcommand{\Fc}{{\mathcal F}}

\newcommand{\Hc}{{\mathcal H}}
\newcommand{\Ic}{{\mathcal I}}
\newcommand{\Jc}{{\mathcal J}}

\newcommand{\Lc}{{\mathcal L}}
\newcommand{\Nc}{{\mathcal N}}
\newcommand{\Oc}{{\mathcal O}}
\newcommand{\Pc}{{\mathcal P}}
\newcommand{\Qc}{{\mathcal Q}}
\newcommand{\Rc}{{\mathcal R}}
\newcommand{\Sc}{{\mathcal S}}

\newcommand{\Vc}{{\mathcal V}}
\newcommand{\Wc}{{\mathcal W}}

\newcommand{\Id}{\on{id}}

\newcommand{\colim}{\on{colim}}
\usepackage{rotating}

\newtheorem{thm}[equation]{Theorem}
\newtheorem{ex}[equation]{Example}

\newtheorem{lem}[equation]{Lemma}
\newtheorem{prop}[equation]{Proposition}

\newtheorem{defi}[equation]{Definition}

   \newtheorem{Defi}[equation]{Definition}
   \newtheorem{rem}[equation]{Remark}

   \numberwithin{equation}{subsection}

\renewenvironment{itemize}{
  \begin{list}{}{
    \setlength{\leftmargin}{4em}
    \setlength{\itemsep}{0.25em}
    \setlength{\parskip}{10pt}
    \setlength{\parsep}{0.25em}
  }
}{
  \end{list}
}

\newcommand{\0}{\otimes}
\newcommand{\ul}{\underline}

\title{Homotopical Algebraic Context over Differential Operators}
\author{Gennaro Di Brino, Damjan Pi\v{s}talo, and Norbert Poncin}
\date{14 June 2017}

\begin{document}

\maketitle

\begin{abstract}
{Building on previous works \cite{DPP, DPP2, PP}, we show that the category of non-negatively graded chain complexes of $\mc{D}_X$-modules -- where $X$ is a smooth affine algebraic variety over an algebraically closed field of characteristic zero -- fits into a homotopical algebraic context in the sense of \cite{TV08}.}
\end{abstract}

\tableofcontents

\section*{Introduction}

The \ul{classical Batalin-Vilkovisky complex} is, roughly, a kind of resolution of $\Ci(\zS)^{\op{gs}}$. The functions $\Ci(\zS)$ of the shell $\zS$ are obtained by identifying those functions of the infinite jet bundle $J^\infty E$ of the field bundle $E\to X$ that coincide on-shell (quotient). We then get the functions $\Ci(\zS)^{\op{gs}}$ by selecting those on-shell functions that are gauge invariant, i.e., constant along the gauge orbits (intersection). When working dually, i.e., with spaces instead function algebras, we first mod out the gauge symmetries, i.e., we consider some space $C:=J^\infty E/\op{GS},$ where $\op{GS}$ refers to integrated gauge symmetry vector fields $\op{gs}$ -- thought of as vector fields prolonged to $J^\infty E$. Since, in the function algebra approach, we determine the shell $\zS$ by solving the algebraic infinite jet bundle equation $\op{Alg}(\op{d}S)=0$ that corresponds to the equation $\op{d}S=0$, where $S$ denotes the functional acting on sections of $E$, it is clear that, in the dual approach, the functional $S$ must be defined on $C$, i.e., $S\in\mathcal{O}(C)$ and $\op{d}S:C\to T^*C$, and that we have then to find those `points $m$' in $C$ that satisfy $\op{d}_mS=0$.\medskip

When switching to the \ul{context of algebraic geometry}, we start with a quasi-coherent module $\Ec\in{\tt qcMod}(\Oc_X)$ over the function sheaf $\Oc_X$ of a {\it scheme} $X$. Let now $\Sc_{\Oc_X}$ be the corresponding symmetric tensor algebra functor. The quasi-coherent commutative $\Oc_X$-algebra $\Sc_{\Oc_X}\Ec\in{\tt qcCAlg}(\Oc_X)$ can be viewed as the pushforward $\Oc_X^{E}$ of the function sheaf of a vector bundle $E\to X$ (we think about $\Ec$ as the module of sections of the dual bundle $E^*$). If $X$ is a {\it smooth scheme}, the infinite jet functor $\mathcal J^\infty$ \cite{BD04} leads to a sheaf ${\mathcal J}^\infty(\Oc_X^{E})\in{\tt qcCAlg}(\Dc_X)$ of commutative algebras over the sheaf $\Dc_X$ of rings of differential operators on $X$, which is quasi-coherent as sheaf of $\Oc_X$-modules. The spectrum of the latter is the infinite jet bundle $J^\infty E\to X$. This bundle is thus an affine $X$-$\Dc_X$-scheme $J^\infty E\in{\tt Aff}(\Dc_X)$.\medskip

Since intersections of two sub-schemes in an affine scheme can be singular at some points (non-transversal intersections), or, algebraically, since tensor products of commutative rings viewed as certain modules can be badly behaved (tensor product functor only right-exact), these tensor products should be left-derived, i.e., commutative rings or commutative algebras should be replaced by simplicial commutative rings or differential non-negatively graded commutative algebras (category {\tt DGCA}). Similarly, since quotients of affine schemes can be non-affine (non-trivial automorphism groups), they should be derived, i.e., replaced by groupoids, or, in the case of higher symmetries, by infinity groupoids or simplicial sets (category {\tt SSet}). For the \ul{functor of points approach} to schemes -- schemes are viewed as, say, locally representable sheaves (for the Zariski topology) $G:\tt CA\to Set$ from commutative algebras to sets -- this means that we pass to functors $F:{\tt DGCA}\to {\tt SSet}$.\medskip

In the following we assume that $X$ is a {\it smooth affine algebraic variety}, so that we can, roughly speaking, replace sheaves by their total sections. In particular, in the above $\Dc$-geometric setting, the differential non-negatively graded commutative algebras of the preceding paragraph, i.e., the objects of $\tt DGCA$, become the objects of ${\tt DG_+qcCAlg}(\Dc_X)$, i.e., the sheaves of differential non-negatively graded $\Oc_X$-quasi-coherent commutative $\Dc_X$-algebras, and, due to the assumption that $X$ is smooth affine, the category ${\tt DG_+qcCAlg}(\Dc_X)$ is equivalent \cite{DPP} to the category $\tt DG\Dc A$ of differential non-negatively graded commutative algebras over the ring $\Dc=\Dc_X(X)$ of total sections of $\Dc_X$. Let us mention that the latter category is of course the category ${\tt CMon(DG\Dc M)}$ of commutative monoids in the symmetric monoidal category $\tt DG\Dc M$ of differential non-negatively graded modules over $\Dc$ (i.e., the category $\tt DG\Dc M$ of non-negatively graded chain complexes of $\Dc$-modules), as well as that, despite the used simplified notation $\tt DG\Dc M$ and $\tt DG\Dc A$, the reader should keep in mind the considered non-negative grading and underlying variety $X$.

It follows that, in $\Dc$-geometry, the above functors $F: \tt DGCA\to SSet$ become functors $$F:\tt DG\Dc A\to SSet\;.$$ As suggested in the second paragraph, the category ${\tt DG\Dc A\simeq DG_+qcCAlg}(\Dc_X)$ is opposite to the category ${\tt D_-Aff}(\Dc_X)$ of  \ul{derived affine} $X$-$\Dc_X$-\ul{schemes}. Those functors or presheaves $F:\tt DG\Dc A\to SSet$ that are actually sheaves are referred to as  \ul{derived} $X$-$\Dc_X$-\ul{stacks} and the model category of presheaves $\tt F:=Fun(DG\Dc A,SSet)$ models the category of derived $X$-$\Dc_X$-stacks \cite{TV05, TV08}. The \ul{sheaf condition} is a natural homotopy version of the standard sheaf condition \cite{KS}. This homotopical variant is correctly encoded in the fibrant object condition of the \ul{local model structure} of $\tt F$. That structure encrypts both, the model structure of the target and the one of the source \cite{DPP, DPP2}. More precisely, one starts with the global model structure on $\tt F$, which is the one implemented `object-wise' by the model structure of the target category $\tt SSet$. The model structure of the source category $\tt DG\Dc A$ is taken into account via the left Bousfield localization with respect to the weak equivalences of $\tt DG\Dc A^{\op{op}}$, what leads to a new model category denoted by $\tt F{\;\hat{}}\,$. If $\zt$ is an appropriate model pre-topology on $\tt DG\Dc A^{\op{op}}$, it should be possible to define homotopy $\zt$-sheaves of groups, as well as a class $H_\zt$ of homotopy $\zt$-hypercovers. The mentioned local model category ${\tt F}{\;\tilde{}}^{\;,\tau}$ arises now as the left Bousfield localization of $\tt F{\;\hat{}}\,$ with respect to $H_\zt$. The local weak equivalences are those natural transformations that induce isomorphisms between all homotopy sheaves. The fibrant object condition in ${\tt F}{\;\tilde{}}^{\;,\tau}\!$, which is roughly the descent condition with respect to the homotopy $\zt$-hypercovers, is the searched sheaf or stack condition for derived $X$-$\Dc_X$-stacks \cite{TV05, TV08}. The notion of derived $X$-$\Dc_X$-stack represented by an object in ${\tt D_-Aff}(\Dc_X)\simeq {\tt DG\Dc A}^{\op{op}}$, i.e., by a derived affine $X$-$\Dc_X$-scheme, can easily be defined.

Notice finally that our two assumptions -- smooth and affine -- on the underlying algebraic variety $X$ are necessary. Exactly the same smoothness condition is indeed used in \cite{BD04}[Remark p. 56], since for an arbitrary singular scheme $X$, the notion of left $\Dc_X$-module is meaningless. On the other hand, the assumption that $X$ is affine is needed to replace the category ${\tt DG_+qcMod}(\Dc_X)$ by the category $\tt DG\Dc M$ and to thus avoid the problem of the non-existence of a projective model structure \cite{Gil06}. However, the confinement to the affine case, does not only allow to use the artefacts of the model categorical environment, but may also allow to extract the fundamental structure of the main actors of the considered problem and to extend these to an arbitrary smooth scheme $X$ \cite{PP}.\smallskip

To implement the preceding ideas, one must prove that the triplet $(\tt DG\Dc M, DG\Dc M, DG\Dc A)$ is a homotopical algebraic context ({\small HA} context) and consider moreover a homotopical algebraic geometric context $({\tt DG\Dc M, DG\Dc M, DG\Dc A},\zt,\mathbf{P})$ ({\small HAG} context). A {\small HA} context is a triplet $(\tt C, C_0,A_0)$ made of a symmetric monoidal model category $\tt C$ and two full subcategories $\tt C_0\subset\tt C$ and $A_0\subset\tt CMon(C)$, which satisfy several quite natural but important assumptions that guarantee that essential tools from linear and commutative algebra are still available. Further, $\mathbf{P}$ is a class of morphisms in $\tt DG\Dc A^{\op{op}}$ that is compatible with $\zt$ (a priori one may think about $\zt$ as being the \'etale topology and about $\mathbf{P}$ as being a class of smooth morphisms). In this framework, a 1-\ul{geometric derived} $X$-$\Dc_X$-\ul{stack} is, roughly, a derived $X$-$\Dc_X$-stack, which is obtained as the quotient by a groupoid action -- in representable derived $X$-$\Dc_X$-stacks -- that belongs to $\mathbf{P}$. Hence, $\mathbf{P}$ determines the type of action we consider (e.g., a smooth action, maybe a not really nice action) and determines the type of geometric stack we get.

Let us now come back to the first two paragraphs of this introduction. Since $J^\infty E\in {\tt Aff}(\Dc_X)\subset {\tt D_-Aff}(\Dc_X)\simeq {\tt DG\Dc A^{\op{op}}}$ is a representable derived $X$-$\Dc_X$-stack, it is natural to view $C:=J^\infty E/\op{GS}$, or, better, $C:=[J^\infty E/\op{GS}]$ as a $1$-geometric derived $X$-$\Dc_X$-stack (or even an $n$-geometric one). Further evidence for this standpoint appears in \cite{CG1, P11, Paugam1, Vino}.

The full implementation of the above $\Dc$-geometric \cite{BD04} extensions of homotopical algebraic geometric ideas \cite{TV05, TV08}, as well as of the program sketched in the first paragraph {\it within} this {\small HAG} setting over differential operators, is being written down in a separate paper \cite{PP1}. In the present text, we prove that $(\tt DG\Dc M, DG\Dc M, DG\Dc A)$ is indeed a {\small HA} context. Let us recall that modules over the non-commutative ring $\Dc$ of differential operators are rather special. For instance, the category $\tt DG\Dc M$ is closed monoidal, with internal Hom and tensor product taken, not over $\Dc$, but over $\Oc$. More precisely, one considers in fact the $\Oc$-modules given, for $M,N\in\tt DG\Dc M$, by $\h_\Oc(M,N)$ and $M\0_\Oc N$, and shows that their $\Oc$-module structures can be extended to $\Dc$-module structures. This and other specificities must be kept in mind throughout the whole of the paper, and related subtleties have to be carefully checked.

It can be shown that the \ul{new homotopical algebraic} $\Dc$-\ul{geometric approach} provides in particular a convenient way to encode total derivatives and allows to recover the classical Batalin-Vilkovisky complex as a specific case of the general constructions \cite{PP1}.

\section{{Monoidal model structure on differential graded $\Dc$-modules.}}\label{S:momo}

In this section, we show that the category ${\tt DG\Dc M}$ is a symmetric monoidal {\it model} category. Such a category is the basic ingredient of a Homotopical Algebraic Context.

\begin{defi} A {\em symmetric monoidal model structure} on a category $\tt C$ is a closed symmetric monoidal structure together with a model structure on $\tt C$, which satisfy the compatibility axioms:
\begin{itemize}
 \item[\bf{MMC1.}] The monoidal structure {$\otimes:{\tt C}\times{\tt C}\rightarrow {\tt C}$} is a Quillen bifunctor.
 \item[\bf{MMC2.}] If $Q\mathbb{I}\xrightarrow{q} \mathbb{I}$ is the cofibrant replacement of the monoidal unit $\mathbb{I}$ (obtained from the functorial `cofibration - trivial fibration' decomposition of $\emptyset\to \mathbb{I}$), then the map $$Q\mathbb{I}\otimes C\xrightarrow{q\otimes \Id} \mathbb{I}\otimes C$$ is a weak equivalence for every cofibrant $C\in\tt C$.
\end{itemize}
\end{defi}

We briefly comment on this definition \cite{Ho99}.\medskip

1. It is known that a morphism of model categories needs not respect the whole model categorical structure -- this would be too strong a requirement. The concept of Quillen functor is the appropriate notion of morphism between model categories. In the preceding definition, we ask that $\otimes$ be a Quillen {\em bi}$\,$functor, i.e., that, for any two cofibrations $f:T\to U$ and $g:V\to W$, the universal morphism or {\em pushout product} $f\Box g$ in the next diagram be a cofibration as well -- which is trivial if one of the inducing maps $f$ or $g$ is trivial.

\begin{equation}\label{e:chpushout}
\begin{tikzcd}[column sep=large]
T\otimes V \arrow{r}{f\otimes \Id} \arrow{d}[swap]{\Id\otimes g} & U\otimes V \arrow{d}{} \arrow[bend left]{ddr}{\Id\otimes g}\\
T\otimes W \arrow[bend right]{drr}[swap]{f\otimes \Id} \arrow{r}{} & U\otimes V\,\coprod_{T\otimes V}\,T\otimes W \arrow[dashed]{dr}[swap]{f\Box g}\\
& & U\otimes W
\end{tikzcd}
\ee

If the model category $\tt C$ is cofibrantly generated, it suffices to check the pushout axiom {\small MMC1} for generating (trivial) cofibrations.

2. Note that the axiom {\small MMC2} is obviously satisfied if $\mathbb{I}$ is cofibrant.\medskip

The category $\tt C = \tt DG\Dc M$ is an Abelian symmetric monoidal and a finitely generated model category \cite{DPP} over any smooth affine variety $X$ over an algebraically closed field of characteristic 0.\medskip

The monoidal unit is $\mathbb{I}=\Oc=\Oc_X(X)$ viewed as concentrated in degree 0 and with zero differential. This complex $(\Oc,0)$ is cofibrant if the unique chain map $(\{0\},0)\to (\Oc,0)$ is a cofibration, i.e., an injective chain map with degree-wise projective cokernel. It is clear that this cokernel is $(\Oc,0)$ itself. It is degree-wise projective if and only if $\Oc$ is a projective $\Dc$-module. Therefore, the axiom {\small MMC2} is not obvious, if $\Oc$ is not a flat $\Dc$-module. The $\Dc$-module $\Oc$ is flat if and only if, for any injective $\Dc$-linear map $M\to N$ between right $\Dc$-modules, the induced $\Z$-linear map $M\0_\Dc\Oc\to N\0_\Dc\Oc$ is injective as well. Let now $\Oc=\mathbb{C}[z]$, consider the complex affine line $X=\text{Spec}\,\Oc$ and denote by $(\p_z)$ the right ideal of the ring $\Dc=\Dc_X(X)$. The right $\Dc$-linear injection $(\p_z)\to \Dc$ induces the morphism $$
(\p_z)\otimes_{\Dc}\Oc\rightarrow \Dc\otimes_{\Dc}\Oc\simeq\Oc\;.$$ Since $\p_z\simeq\p_z\otimes_{\Dc}1$ is sent to $1 \otimes_{\Dc}\p_z 1\simeq 0$, the kernel of the last morphism does not vanish; hence, in the case of the complex affine line, $\Oc$ is not $\Dc$-flat. Eventually, {\small MMC2} is not trivially satisfied.\medskip

Before proving that {\small MMC1} and {\small MMC2} hold, we have still to show that the category ${\tt DG\Dc M}$, which carries a (cofibrantly generated) model structure, is {\em closed} symmetric monoidal. Let us stress that the equivalent category ${\tt DG_+qcMod}(\Dc_X)$ is of course equipped with a model structure, but is {\it a priori} not closed, since the internal Hom of $\Oc_X$-modules does not necessarily preserve $\Oc_X$-quasi-coherence (whereas the tensor product of quasi-coherent $\Oc_X$-modules is quasi-coherent). On the other hand, the category ${\tt DG_+Mod}(\Dc_X)$ is closed symmetric monoidal \cite{Scha, Schn}, but not endowed with a projective model structure \cite{Gil06} (it has an injective model structure, which, however, is not monoidal \cite{Joy}). The problem is actually that the category ${\tt Mod}(\Dc_X)$ has not enough projectives. The issue disappears for ${\tt qcMod}(\Dc_X)$, since this category is equivalent to the category ${\tt \Dc M}$ of modules over the ring $\Dc$.\medskip

Let us start with the following observation. Consider a topological space $X$ -- in particular a smooth variety -- and a sheaf $\Rc_X$ of unital rings over $X$, and let $R=\zG(X,\Rc_X)$ be the ring of global sections of $\Rc_X$. We will also denote the global sections of other sheaves by the Latin letter corresponding to the calligraphic letter used for the considered sheaf. The localization functor $\Rc_X\0_R -: {\tt Mod}(R) \to {\tt Mod}(\Rc_X)$ is left adjoint to the global section functor $\zG(X,-): {\tt Mod}(\Rc_X)\to {\tt Mod}(R)$: \be\label{LeftAdj}\text{Hom}_{\Rc_X} (\Rc_X\0_R V,\Wc) \simeq \text{Hom}_R(V,\text{Hom}_{\Rc_X} (\Rc_X,\Wc))\simeq \text{Hom}_R(V,W)\;,\ee for any $V\in {\tt Mod}(R)$ and $\Wc\in {\tt Mod}(\Rc_X)$  \cite{MilDMod}.\bigskip

As mentioned above, the category $({\tt Mod}(\Dc_X),\0_{\Oc_X},\Oc_X,{\Hc}om_{\Oc_X})$ is Abelian closed symmetric monoidal.
More precisely, for any $\Nc,\Pc,\Qc\in{\tt Mod}(\Dc_X)$, there is an isomorphism \be\label{Isom}{\Hc}om_{\Dc_X}(\Nc\0_{\Oc_X}\Pc,\Qc)\simeq{\Hc}om_{\Dc_X}(\Nc,{\Hc}om_{\Oc_X}(\Pc,\Qc))\;.\ee Let now $\Rc_X$ be $\Oc_X$ or $\Dc_X$. The preceding Hom functor ${\Hc}om_{{\Rc_X}}(-,-)$ is the `internal' Hom of sheaves of ${\Rc_X}$-modules, i.e., the functor defined, for any such sheaves $\Vc,\Wc\in{\tt Mod}(\Rc_X)$ and for any open $U\subset X$, by $${\Hc}om_{{\Rc_X}}(\Vc,\Wc)(U)=\text{Hom}_{{\Rc_X}|_U}(\Vc|_U,\Wc|_U)\;,$$ where the {\small RHS} Hom denotes the morphisms of sheaves of ${\Rc_X}|_U$-modules. This set is an Abelian group and an $\Rc_X(U)$-module, if $\Rc_X$ is commutative. Hence, by definition, we have $$\zG(X,{\Hc}om_{{\Rc_X}}(\Vc,\Wc))=\text{Hom}_{{\Rc_X}}(\Vc,\Wc)\;.\vspace{1.5mm}$$

Recall now that, in the (considered) case of a smooth affine variety $X$, the global section functor $\zG(X,-)$ yields an equivalence $$\zG(X,-):{\tt qcMod}(\Rc_X)\to {\tt Mod}(R):\Rc_X\0_R -\;$$ of Abelian symmetric monoidal categories. The quasi-inverse $\Rc_X\0_R -$ of $\zG(X,-)$ is well-known if $\Rc_X=\Oc_X$; for $\Rc_X=\Dc_X$, we refer the reader to \cite{HTT}; the quasi-inverses are both strongly monoidal. If $\Vc\in{\tt qcMod}(\Rc_X)$ and $\Wc\in{\tt Mod}(\Rc_X)$, we can thus write $\Vc\simeq\Rc_X\0_R V$, where $V=\zG(X,\Vc)$, and, in view of (\ref{LeftAdj}), also \be\label{Sol}\text{Hom}_{\Rc_X} (\Vc,\Wc)=\text{Hom}_{\Rc_X} (\Rc_X\0_R V,\Wc)\simeq \text{Hom}_R(V,W)\;.\ee

When applying the global section functor to (\ref{Isom}), we get $$\text{Hom}_{\Dc_X}(\Nc\0_{\Oc_X}\Pc,\Qc)\simeq\text{Hom}_{\Dc_X}(\Nc,{\Hc}om_{\Oc_X}(\Pc,\Qc))\;,$$ and, when assuming that $\Nc,\Pc,\Qc\in{\tt qcMod}(\Dc_X)$ and using (\ref{Sol}), we obtain $$\text{Hom}_{\Dc}(\zG(X,\Nc\0_{\Oc_X}\Pc),Q)\simeq\text{Hom}_{\Dc}(N,\zG(X,{\Hc}om_{\Oc_X}(\Pc,\Qc)))\;,$$ or, still,
\be\label{Closed}\text{Hom}_{\Dc}(N\0_{\Oc}P,Q)\simeq\text{Hom}_{\Dc}(N,\text{Hom}_{\Oc}(P,Q))\;.\ee Since any $\Dc$-module $L$ can be viewed as $\zG(X,\Lc)$, where $\Lc=\Dc_X\0_\Dc L\in{\tt qcMod}(\Dc_X)\subset{\tt Mod}(\Dc_X)$, the equation (\ref{Closed}) proves that $({\tt \Dc M},\otimes_\Oc,\Oc,\text{Hom}_\Oc)$ is -- just as (${\tt Mod}(\Dc_X),\0_{\Oc_X},\Oc_X,{\Hc}om_{\Oc_X})$ -- an Abelian closed symmetric monoidal category. Observe that the internal Hom of $\tt \Dc M$ is given by: \be\label{IntHomDM}\text{Hom}_\Oc(-,-)=\zG(X,{\Hc}om_{\Oc_X}(\Dc_X\0_\Dc -,\Dc_X\0_\Dc -))\in{\tt \Dc M}\;.\ee

Both categories satisfy the {\small AB3} (Abelian category with direct sums) and {\small AB3*} (Abelian category with direct products) axioms. It thus follows from \cite[Lemma 3.15]{LH} that the corresponding categories of chain complexes are Abelian closed symmetric monoidal as well. The tensor product is the usual tensor product $(-\0_\bullet-,\zd_\bullet)$ of chain complexes and the internal $(\text{Hom}_\bullet(-,-),d_\bullet)$ is defined, for any complexes $(M_\bullet,d_M)$  and $(N_\bullet,d_N)$ and for any $n\in\N$, by \be\label{DGHom}\text{Hom}_n(M_\bullet,N_\bullet)=\begin{cases}\prod_{k\in\N}\text{Hom}_\Oc(M_k,N_{k+n}),\;\,\text{in the case of}\;\, {\tt DG\Dc M}\;,\\ \prod_{k\in\N}{\Hc}om_{\Oc_X}(M_k,N_{k+n}),\;\,\text{in the case of}\;\,{\tt DG_+Mod}(\Dc_X)\;,\end{cases}\ee and, for any $f=(f_k)_{k\in\N}\in\text{Hom}_n(M_\bullet,N_\bullet),$ by $$(d_nf)_k=d_N\circ f_{k}-(-1)^nf_{k-1}\circ d_M\;.\vspace{1.5mm}$$

The closed structure $\text{Hom}_\bullet(-,-)$ of $\tt DG\Dc M$ defines a closed structure ${\Hc}om_\bullet(-,-)$ on the equivalent category ${\tt DG_+qcMod}(\Dc_X)$ via usual transfer $${\Hc}om_\bullet(-,-)=\Dc_X\0_\Dc\left(\text{Hom}_\bullet(\zG(X,-),\zG(X,-))\right)=$$ $$\Dc_X\0_\Dc\left(\prod_{k}\text{Hom}_\Oc(\zG(X,-_{k}),\zG(X,-_{k+\bullet}))\right)=\Dc_X\0_\Dc\left(\prod_{k}\zG(X,{\Hc}om_{\Oc_X}(-_{k},-_{k+\bullet}))\right)\;,$$ where we used (\ref{IntHomDM}). According to what has been said above, we have the adjunction $$\Dc_X\0_\Dc -:{\tt Mod}(\Dc)\rightleftarrows{\tt Mod}(\Dc_X):\zG(X,-)\;,$$ so that $\zG(X,-)$ commutes with limits: $${\Hc}om_\bullet(-,-)=\Dc_X\0_\Dc\zG\left(X,\prod_{k}{\Hc}om_{\Oc_X}(-_{k},-_{k+\bullet})\right)=\Dc_X\0_\Dc\zG\left(X,\text{Hom}_\bullet(-,-)\right)\;,$$ where $\text{Hom}_\bullet(-,-)$ is now the above closed structure of ${\tt DG_+Mod}(\Dc_X)$. However, since $\text{Hom}_\bullet(-,-)$ is in general not quasi-coherent, the {\small RHS} is in the present case not isomorphic to the module $\text{Hom}_\bullet(-,-)$. More precisely, the closed structure on ${\tt DG_+qcMod}(\Dc_X)$ is given by the coherator of the closed structure on ${\tt DG_+Mod}(\Dc_X)$. Note also that, since $\tt DG\Dc M$ and ${\tt DG_+qcMod}(\Dc_X)$ are equivalent symmetric monoidal categories and the internal Hom of the latter is the transfer of the one of the former closed symmetric monoidal category, the second category is closed symmetric monoidal as well (i.e., its monoidal and its closed structures are `adjoint'). \medskip

Hence, the

\begin{prop} The category $({\tt DG\Dc M},\0_\bullet,\Oc,\text{\em Hom}_\bullet)$ $(\,$resp., $({\tt DG_+qcMod}(\Dc_X),\0_\bullet,\Oc_X,{\Hc}om_\bullet)$$\,)$ is Abelian closed symmetric monoidal. The closed structure is obtained by transfer of $(\,$resp., as the coherator of$\;)$ the closed structure of ${\tt DG_+Mod}(\Dc_X)$. In particular, for any $N_\bullet, P_\bullet, Q_\bullet\in {\tt DG\Dc M}$, there is a $\Z$-module isomorphism \be\label{IntHomDG}\text{\em Hom}_{{\tt DG\Dc M}}(N_\bullet\0_\bullet P_\bullet,Q_\bullet)\simeq\text{\em Hom}_{\tt DG\Dc M}(N_\bullet,\text{\em Hom}_\bullet(P_\bullet,Q_\bullet))\;,\ee which is natural in $N_\bullet$ and $Q_\bullet\,$.\end{prop}

To examine the axiom {\small MMC1}, we need the next proposition. As up till now, we write $\Dc$ (resp., $\Oc$) instead of $\zG(X,\Dc_X)$ (resp., $\zG(X,\Oc_X)$).

\begin{prop}\label{Projectivity} If the variety $X$ is smooth affine, the module $\Dc$ is projective as $\Dc$- and as $\Oc$-module.\end{prop}

\newcommand{\te}{\text}

\begin{proof} Projectivity of $\Dc$ as $\Dc$-module is obvious, since $\te{Hom}_\Dc(\Dc,-)\simeq \Id(-)$. Recall now that the sheaf $\Dc_X$ of differential operators is a filtered sheaf $\text{F}\,\Dc_X$ of $\Oc_X$-modules, with filters defined by $$\text{F}_{-1}\Dc_X=\{0\}\quad\text{and}\quad\text{F}_i\Dc_X=\{D\in\Dc_X:[D,\Oc_X]\subset \text{F}_{i-1}\Dc_X\}\;:$$ $\varinjlim_i\text{F}_i\Dc_X=\Dc_X$. The graded sheaf $\text{Gr}\,\Dc_X$ associated to $\text{F}\,\Dc_X$ is the sheaf, whose terms are defined by $$\text{Gr}_i\Dc_X=\text{F}_i\Dc_X/\text{F}_{i-1}\Dc_X\;.$$ Consider now, for $i\in\N$, the short exact sequence of $\Oc_X$-modules $$0\to \te{F}_{i-1}\Dc_X\to \te{F}_i\Dc_X\to \te{Gr}_i\Dc_X\to 0\;.$$ Due to the local freeness of $\Dc_X$, this is also an exact sequence in ${\tt qcMod}(\Oc_X)$. Since $X$ is affine, we thus get the exact sequence \be\label{SplitSeq} 0\to \zG(X,\te{F}_{i-1}\Dc_X)\to \zG(X,\te{F}_i\Dc_X)\to \zG(X,\te{Gr}_i\Dc_X)\to 0\;\ee in ${\tt Mod}(\Oc)$ -- in view of the equivalence of Abelian categories $$\zG(X,-):{\tt qcMod}(\Oc_X)\rightleftarrows {\tt Mod}(\Oc)\;.$$ However, the functor $\zG(X,-)$ transforms a locally free $\Oc_X$-module of finite rank into a projective finitely generated $\Oc$-module. We can therefore conclude that $\zG(X,\te{Gr}_i\Dc_X)$ is $\Oc$-projective, what implies that the sequence (\ref{SplitSeq}) is split, i.e., that $$\zG(X,\te{F}_i\Dc_X)=\zG(X,\te{Gr}_i\Dc_X)\oplus\zG(X,\te{F}_{i-1}\Dc_X)\;.$$ An induction and commutation of the left adjoint $\zG(X,-)$ with colimits allow to conclude that $$\Dc=\zG(X,\varinjlim_i\text{F}_i\Dc_X)=\varinjlim_i\bigoplus_{j=0}^i\zG(X,\te{Gr}_j\Dc_X)=\bigoplus_{j=0}^\infty\zG(X,\te{Gr}_j\Dc_X)\;.$$ Finally, $\Dc$ is $\Oc$-projective as direct sum of $\Oc$-projective modules.
\end{proof}

\begin{thm} The category $\tt DG\Dc M$ is a symmetric monoidal model category. \end{thm}

\proof[Proof of Axiom {\small\em MMC1}] In this proof, we omit the bullets in the notation of complexes. We have to show that the pushout product of two generating cofibrations is a cofibration and that the latter is trivial if one of its factors is a generating trivial cofibration. Recall that the generating cofibrations (resp., generating trivial cofibrations) in $\tt DG\Dc M$ are the canonical maps $$\zi_0:0\to S^0\quad\text{and}\quad\iota_n:S^{n-1}\to D^n\; (n\ge 1)$$ \be\label{GenTriCof}\text{(resp., }\;\zz_n: 0\to D^n\; (n\ge 1)\;)\;.\ee Here $D^n$ is the $n$-disc, i.e., the chain complex \be\label{Disc}D^n: \cdots \to 0\to 0\to \stackrel{(n)}{\Dc} \to \stackrel{(n-1)}{\Dc}\to 0\to \cdots\to \stackrel{(0)}{0}\;(n\ge 1)\;,\ee whereas $S^n$ is the $n$-sphere, i.e., the chain complex \be\label{Sphere}S^n: \cdots \to 0\to 0\to \stackrel{(n)}{\Dc}\to 0\to \cdots\to \stackrel{(0)}{0}\;(n\ge 0)\;.\ee The map $\zi_n$ vanishes, except in degree $n-1$, where it is the identity map $\Id$; the differential in $D^n$ vanishes, except in degree $n$, where it is the desuspension map $s^{-1}$.\medskip

{\it Step 1}. We consider the case of $\zi_m\Box \zi_n$ ($m,n\ge 1$) (the cases $m$ or $n$ is zero and $m=n=0$ are similar but easier), i.e., we prove that the pushout product in the diagram

\begin{equation}\label{e:chpushout}
\begin{tikzcd}[column sep=large]
S^{m-1}\otimes S^{n-1} \arrow{r}{\zi_m\otimes \Id} \arrow{d}[swap]{\Id\otimes \zi_n} & D^m\otimes S^{n-1} \arrow{d}{} \arrow[bend left]{ddr}{\Id\otimes \zi_n}\\
S^{m-1}\otimes D^n \arrow[bend right]{drr}[swap]{\zi_m\otimes \Id} \arrow{r}{} & \coprod:=D^m\0 S^{n-1}\coprod_{S^{m-1}\0 S^{n-1}}S^{m-1}\0 D^n \arrow[dashed]{dr}[swap]{\zi_m\Box \zi_n}\\
& & D^m\otimes D^n
\end{tikzcd}\;
\ee

\noindent is a cofibration.

Remark that $$S^{m-1}\0 S^{n-1}:\cdots \to 0\to \stackrel{(m-1)}{\Dc}\0\stackrel{(n-1)}{\Dc}\to 0\to \cdots\to 0\;,$$ $$D^m\0 S^{n-1}:\cdots \to 0\to \stackrel{(m)}{\Dc}\0\stackrel{(n-1)}{\Dc}\to \stackrel{(m-1)}{\Dc}\0\stackrel{(n-1)}{\Dc}\to 0\to \cdots\to {0}\;,$$ $$S^{m-1}\0 D^{n}:\cdots \to 0\to \stackrel{(m-1)}{\Dc}\0\stackrel{(n)}{\Dc}\to \stackrel{(m-1)}{\Dc}\0\stackrel{(n-1)}{\Dc}\to 0\to \cdots\to {0}\;,$$ and \be\label{Coker1}D^m\0 D^{n}:\cdots \to 0\to \stackrel{(m)}{\Dc}\0\stackrel{(n)}{\Dc}\to \stackrel{(m-1)}{\Dc}\0\stackrel{(n)}{\Dc}\oplus\stackrel{(m)}{\Dc}\0\stackrel{(n-1)}{\Dc}\to \stackrel{(m-1)}{\Dc}\0\stackrel{(n-1)}{\Dc}\to 0\to \cdots\to {0}\;.\ee The non-trivial terms of the differentials are, $s^{-1}\0\Id$ in $D^m\0 S^{n-1}$, $\Id\0 s^{-1}$ in $S^{m-1}\0 D^{n}$, as well as $s^{-1}\0\Id + \Id\0 s^{-1}$ and $\Id\0 s^{-1}\oplus s^{-1}\0\Id$ in $D^m\0 D^{n}$.

In an Abelian category pushouts and pullbacks do exist. For instance, the pushout of two morphisms $f:A\to B$ and $g:A\to C$ is the cokernel $(h,k):B\oplus C\to \text{coker}(f,-g)$ of the morphism $(f,-g):A\to B\oplus C$. In the Abelian category of chain complexes in an Abelian category, and in particular in $\tt DG\Dc M$, cokernels are taken degree-wise. Hence, in degree $p\in\N$, the pushout of the chain maps $\zi_m\0\Id$ and $\Id\0\zi_n$ is the cokernel $$(h_p,k_p):(D^m\0 S^{n-1})_p\oplus (S^{m-1}\0 D^n)_p\to \text{coker}((\zi_m\0\Id)_p,-(\Id\0\zi_n)_p)\;.$$ This cokernel is computed in the category of $\Dc$-modules and is thus obtained as quotient $\Dc$-module of the direct sum $(D^m\0 S^{n-1})_p\oplus (S^{m-1}\0 D^n)_p$ by the $\Dc$-submodule generated by $$\{((\zi_m\0\Id)_p(D\0\zD),-(\Id\0\zi_n)_p(D\0\zD)):D\0\zD\in (S^{m-1}\otimes S^{n-1})_p\}\;.$$ In degree $p\neq m+n-2$, we divide $\{0\}$ out, and, in degree $p=m+n-2$, we divide the module $$\stackrel{(m-1)}{\Dc}\0\stackrel{(n-1)}{\Dc}\oplus\stackrel{(m-1)}{\Dc}\0\stackrel{(n-1)}{\Dc}$$ by the submodule $$\{(D\0\zD,-D\0\zD):D\in \stackrel{(m-1)}{\Dc},\zD\in\stackrel{(n-1)}\Dc\}\;.$$ This shows that the considered pushout is \be\label{Coker2}\coprod: \cdots\to 0\to \stackrel{(m)}{\Dc}\0\stackrel{(n-1)}{\Dc}\oplus \stackrel{(m-1)}{\Dc}\0\stackrel{(n)}{\Dc}\to \stackrel{(m-1)}{\Dc}\0\stackrel{(n-1)}{\Dc}\to 0\to\cdots\to 0\;.\ee The non-trivial term of the pushout differential is direct sum differential $s^{-1}\0\Id\oplus \Id\0 s^{-1}$ viewed as valued in $\Dc\0\Dc$.

It is clear that the unique chain map $\zi_m\Box\zi_n$, which renders the two triangles commutative, vanishes, except in degrees $m+n-1$ and $m+n-2$, where it coincides with the identity. Recall now that the cofibrations of $\tt DG\Dc M$ are the injective chain maps with degree-wise projective cokernel. In view of (\ref{Coker2}) and (\ref{Coker1}), the cokernel of the injective map $\zi_m\Box\zi_n$ vanishes in all degrees, except in degree $m+n$, where it is equal to $\Dc\0_\Oc\Dc$. It thus suffices to show that $\Dc\0_\Oc\Dc$ is $\Dc$-projective, i.e., that $\text{Hom}_\Dc(\Dc\0_\Oc\Dc,-):{\tt Mod}(\Dc)\to {\tt Ab}$ is an exact functor valued in Abelian groups. In view of (\ref{Closed}), we have $$\text{Hom}_\Dc(\Dc\0_\Oc\Dc,-)\simeq\text{Hom}_{\Dc}(\Dc,\text{Hom}_{\Oc}(\Dc,-))\simeq \text{Hom}_{\Oc}(\Dc,-):{\tt Mod}(\Dc)\to {\tt Ab}\;.$$ It follows from Proposition \ref{Projectivity} that $\text{Hom}_{\Oc}(\Dc,-)$ is an exact functor ${\tt Mod}(\Oc)\to {\tt Mod}(\Oc)$, so also an exact functor ${\tt Mod}(\Dc)\to {\tt Ab}$.\medskip

{\it Step 2}. Take now the pushout product $\zz_m\Box\zi_n$ ($m,n\ge 1$, see (\ref{GenTriCof})) (the other cases are analogous). It is straightforwardly seen that the considered chain map is the map $$\zz_m\Box\zi_n:D^m\0 S^{n-1}\to D^m\0 D^n\;,$$ which vanishes in all degrees, except in degrees $m+n-1$ and $m+n-2$, where it coincides with the identity. To see that this cofibration is trivial, i.e., induces an isomorphism in homology, we compute the homologies $\te{H}(D^m\0 S^{n-1})$ and $\te{H}(D^m\0 D^n)$. Since $D^m$ is acyclic and since $\Dc$ is $\Oc$-projective, hence, $\Oc$-flat (this fact has also been proven independently in \cite{DPP}), it follows from K\"unneth's formula \cite[Theorem 3.6.3]{Wei} that $\te{H}(D^m\0 S^{n-1})=\te{H}(D^m\0 D^n)=0\,$. Therefore, the map $\te{H}(\zz_m\Box\zi_n)$ is a $\Dc$-module isomorphism. \endproof

\proof[Proof of Axiom {\small\em MMC2}] Axiom {\small MMC2} holds for $\tt DG\Dc M$, thanks to the following more general result, which will be proven independently in \ref{tocum}.

\begin{lem}\label{WeqCof}
Let $f:A\rightarrow B$ be a weak equivalence in $\tt DG\Dc M$ and let $M$ be a cofibrant object. Then $f\otimes \Id_M: A\otimes M\rightarrow B\otimes M$ is again a weak equivalence.
\end{lem}

\section{Monoidal model structure on modules over differential graded $\Dc$-algebras}

\subsection{Modules over commutative monoids} It turned out that {\bf $\Dc$-geometric Koszul-Tate resolutions} \cite{PP} are specific objects of the category $$\tt CMon(\tt Mod_{\tt DG\Dc M}(\Ac))$$ of commutative monoids in the category $\tt Mod_{\tt DG\Dc M}(\Ac)$ of modules in $\tt DG\Dc M$ (see Definition \ref{mod}) over an object $\Ac$ of the category $\tt CMon(DG\Dc M)=DG\Dc A$. Moreover, it is known that {\bf model categorical Koszul-Tate resolutions} \cite{DPP2} are cofibrant replacements in the coslice category $$\Ac\downarrow \tt DG\Dc A\;.$$ The fact that the latter are special $\Dc$-geometric Koszul-Tate resolutions seems to confirm the natural intuition that there is an isomorphism of categories $$\tt CMon(\tt Mod_{\tt DG\Dc M}(\Ac))\simeq \Ac\downarrow \tt DG\Dc A\;.$$ Despite the apparent evidence, this equivalence will be proven in detail below (note that, since in this proof the unit elements of the commutative monoids and of the differential graded $\Dc$-algebras play a crucial role, a similar equivalence for non-unital monoids and non-unital algebras does not hold). Eventually it is clear that an object of the latter under-category is dual to a relative derived affine $X$-$\Dc_X$-scheme, where $X$ is the fixed underlying smooth affine algebraic variety (see above). Similar spaces appear \cite{DPP2} in the classical Koszul-Tate resolution, where vector bundles are pulled back over a vector bundle with the same base manifold $X$.\medskip

We first recall the definition of $\tt Mod_{\tt C}(\Ac)$ and explain that this category is closed symmetric monoidal.

\begin{Defi}\label{mod}
Let $({\tt C}, \otimes,\te{\em I},\te{\em \underline{Hom}})$ be a closed symmetric monoidal category with all small limits and colimits. Consider an (a commutative) algebra in $\tt C$, i.e., a commutative monoid $(\Ac,\zm,\zh)$. The corresponding algebra morphisms are defined naturally and the category of algebras in $\tt C$ is denoted by $\tt Alg_C$. A (left) {\em $\Ac$-module in $\tt C$} is a $\tt C$-object $M$ together with a ${\tt C}$-morphism $\zn:\Ac\otimes M\rightarrow M$, such that the usual associativity and unitality diagrams commute. {\em Morphisms of $\Ac$-modules in $\tt C$} are also defined in the obvious manner and the category of $\Ac$-modules in $\tt C$ is denoted by $\tt Mod_{\tt C}(\Ac)$.
\end{Defi}

The category of right $\Ac$-modules in $\tt C$ is defined analogously. Since $\Ac$ is commutative, the categories of left and right modules are equivalent (one passes from one type of action to the other by precomposing with the braiding `com').\medskip

The tensor product $\0_\Ac$ of two modules $M',M''\in\tt Mod_{C}(\Ac)$ is defined as usual \cite[VII.4, Exercise 6]{MacL} as the coequalizer in $\tt C$ of the maps $$\psi':=(\zn_{M'}\0\Id_{M''})\circ(\te{com}\0\Id_{M''}),\psi'':=\Id_{M'}\0\zn_{M''}:(M'\otimes \Ac)\otimes M''\simeq M'\0(\Ac\0 M'') \rightrightarrows M'\otimes M''\;.$$ Since $\Ac\in\tt Alg_C$ is commutative, $M'\otimes_\Ac M''$ inherits an $\Ac$-module structure from those of $M'$ and $M''$ \cite{TV08}.\medskip

Even for an abstract $\tt C$, one can further define an internal $\te{\underline{Hom}}_\Ac$ in $\tt Mod_{\tt C}(\Ac)$, see Appendix \ref{IntHomMod}. Moreover, the expected adjointness property holds, $$\te{Hom}_\Ac(M\0_\Ac M',M'')\simeq \te{Hom}_\Ac(M,\te{\underline{Hom}}_\Ac(M',M''))\;,$$ and the category of $\Ac$-modules in $\tt C$ has all small limits and colimits. We thus get the

\begin{prop}\cite{TV08}\label{mod2} Exactly as the original category $(\tt C,\0,\te{\em I},\te{\em\underline{Hom}})$, the category $$(\tt Mod_{\tt C}(\Ac),\0_\Ac,\Ac,\te{\em\underline{Hom}}_\Ac)$$ of modules in $\tt C$ over $\Ac\in\tt Alg_C$ is closed symmetric monoidal and contains all small limits and colimits. \end{prop}

\begin{prop}\label{CMonUnder} For any nonzero $\Ac\in\tt DG\Dc A\,,$ there exists an isomorphism of categories $$\tt CMon(\tt Mod_{\tt DG\Dc M}(\Ac))\simeq \Ac\downarrow \tt DG\Dc A\;,$$ where notation has been introduced above.\end{prop}

\begin{lem}\label{SubAlg} The initial {\small\em DG$\Dc$A} $\Oc$ can be viewed as a sub-{\small\em DG$\Dc$A} of any nonzero {\small\em DG$\Dc$A} $\Ac$.\end{lem}

\begin{proof} It suffices to notice that the (unique) $\tt DG\Dc A$-morphism $\zvf:\Oc\to \Ac$, which is defined by $$\zvf(f) = \zvf(f\cdot 1_\Oc) = f\cdot \zvf(1_\Oc) = f\cdot 1_\Ac\;,$$ is injective, since it is the composition of the injective $\tt DG\Dc A$-morphism $\Oc\ni f\mapsto f\0 1_\Ac\in\Oc\0_\Oc\Ac$ and the bijective $\tt DG\Dc A$-morphism $\Oc\0_\Oc\Ac\ni f\0 a\mapsto f\cdot a\in\Ac\,$.
\end{proof}

\begin{rem} In the sequel, $\Ac$ is assumed to be a nonzero differential graded $\Dc$-algebra, whenever needed.\end{rem}

\begin{proof}[Proof of Proposition \ref{CMonUnder}] As already said, the category $\tt C= DG\Dc M$, or, better, $(\tt DG\Dc M,\0_\bullet,\Oc,\te{Hom}_\bullet)$ satisfies all the requirements of Definition \ref{mod} and the category $(\tt Mod_{\tt DG\Dc M}(\Ac),\0_\Ac,\Ac,\te{\underline{Hom}}_\Ac)$ has thus exactly the same properties, see Proposition \ref{mod2}.\medskip

Note also that in the Abelian category $\tt DG\Dc M$ of chain complexes in $\tt \Dc M$, we get $$M'\0_\Ac M''=\op{coeq}(\psi',\psi'')=\op{coker}(\psi''-\psi')\;,$$ so that $$\left(M'\0_\Ac M''\right)_n=\op{coker}_n(\psi''-\psi')=\op{coker}(\psi''_n-\psi'_n)=\left(M'\0_\bullet M''\right)_n/\op{im}(\psi''_n-\psi'_n)\;,$$ where the $\Dc$-submodule in the {\small RHS} quotient is given by $$\left\{\sum_{\op{fin}}\left(m'\0(a\triangleleft'' m'')-(-1)^{|a||m'|}(a\triangleleft' m')\0 m''\right):|a|+|m'|+|m''|=n\right\}\;,$$ where all sums are finite and where $\triangleleft'$ (resp., $\triangleleft''$) denotes the $\Ac$-action $\zn_{M'}$ (resp., $\zn_{M''}$). Hence, in all degrees, the tensors $M'\0_\Ac M''$ are the tensors $M'\0_\bullet M''$ where we identify the tensors $(a\triangleleft' m')\0 m''$ with the tensors $(-1)^{|a||m'|}\;m'\0(a\triangleleft'' m'')\,.$ It is straightforwardly checked that the differential of $M'\0_\bullet M''$ stabilizes the submodules, so that the quotient $M'\0_\Ac M''$ is again in $\tt DG\Dc M$. Moreover, a $\tt DG\Dc M$-morphism $M'\0_\bullet M''\to M$, which vanishes on the submodules, defines a $\tt DG\Dc M$-morphism $M'\0_\Ac M''\to M$. \medskip

Now, an object in $\Ac\downarrow \tt DG\Dc A$ is a $\tt DG\Dc A$-morphism $\zf:\Ac\to M$, i.e., a $\tt DG\Dc M$-morphism that respects the multiplications and units. The target is an element $M\in\tt DG\Dc M$ and is endowed with two $\tt DG\Dc M$-morphisms $\zm_M:M\0_\bullet M\to M$ and $\zh_M:\Oc\to M$, which render commutative the usual associativity, unitality and commutativity diagrams.\medskip

On the other hand, an object $N\in\tt CMon(\tt Mod_{\tt DG\Dc M}(\Ac))$ is an $N\in\tt DG\Dc M$ equipped with a $\tt DG\Dc M$-morphism $\zn:\Ac\0_\bullet N\to N$, for which the associativity and unitality diagrams commute. Moreover, it carries a commutative monoid structure, i.e., there exist $\Ac$-linear $\tt DG\Dc M$-morphisms $\zm_N:N\0_\Ac N\to N$ and $\zh_N:\Ac\to N$, such that the associativity, unitality and commutativity requirements are fulfilled.\medskip

Start from $(\zf:\Ac\to M)\in\Ac\downarrow \tt DG\Dc A$ and set $N=M$ and $\zm_M=-\star -$. Remember that $-\star -$ is $\Oc$-bilinear associative unital and graded-commutative, and define an $\Ac$-action on $M$ by \be\label{Action}a\triangleleft m:=\zn(a\0 m):=\zf(a)\star m\;.\ee In view of \cite[Proposition 6]{DPP}, the well-defined map $\zn$ is a $\tt DG\Dc M$-morphism and it can immediately be seen that $$a'\triangleleft(a''\triangleleft m)=(a' \ast a'')\triangleleft m\quad\te{and}\quad 1_\Ac\triangleleft m = m\;,$$ where $\ast$ denotes the multiplication in $\Ac$. Since, we have \be\label{ABilin}(a\triangleleft m')\star m''=\zf(a)\star m'\star m''=(-1)^{|a||m'|} m'\star\zf(a)\star m''=\ee $$(-1)^{|a||m'|}\;m'\star(a\triangleleft m'')=a\triangleleft(m'\star m'')\;,$$ the $\tt DG\Dc M$-morphism $\zm_M$ is a well-defined $\tt DG\Dc M$ and $\Ac$-linear morphism $\zm_N$ on $M\0_\Ac M$. As for $\zh_N$, note that $\zh_M$ is completely defined by $\zh_M(1_\Ac)=\zh_M(1_\Oc)=1_M$, see Lemma \ref{SubAlg}. Define now an $\Ac$-linear morphism $\zh_N:\Ac\to M$ by setting $\zh_N(1_\Ac)=1_M$. It follows that $$\zh_N(a)=a\triangleleft \zh_N(1_\Ac)=a\triangleleft 1_M=\zf(a)\star 1_M=\zf(a)\;,$$ so that $\zh_N$ is a $\tt DG\Dc M$-morphism, which coincides with $\zh_M$ on $\Oc\subset \Ac$: $$\zh_N(f)=f\cdot 1_M=f\cdot\zh_M(1_\Oc)=\zh_M(f)\;.$$

Conversely, if $N\in\tt CMon(\tt Mod_{\tt DG\Dc M}(\Ac))$ is given, set $M=N$. The composition of the $\tt DG\Dc M$-morphism $\zp:M\0_\bullet M\to M\0_\Ac M$ with the $\tt DG\Dc M$-morphism $\zm_N$ is a $\tt DG\Dc M$-morphism $\zm_M:M\0_\bullet M\to M$. The restriction of the $\tt DG\Dc M$-morphism $\zh_N:\Ac\to M$ to the subcomplex $\Oc\subset\Ac$ in $\tt\Dc M$ is a $\tt DG\Dc M$-morphism $\zh_M:\Oc\to M$. We thus obtain a differential graded $\Dc$-algebra structure on $M$ with unit $1_M=\zh_M(1_\Oc)=\zh_N(1_\Ac)$. Define now a $\tt DG\Dc M$-morphism $\zf:\Ac\to M$ by \be\label{AlgMorph}\zf(a)=\zn(a\0 1_M)=a\triangleleft 1_M\;.\ee This map visibly respects the units and, since $\zm_M=-\star-$ is $\Ac$-bilinear in the sense of (\ref{ABilin}), it respects also the multiplications.\medskip

When starting from a $\tt DG\Dc A$-morphism $\zf_1$ and applying the maps (\ref{Action}) and (\ref{AlgMorph}), we get a $\tt DG\Dc A$-morphism $$\zf_2(a)=a\triangleleft 1_M=\zf_1(a)\star 1_M=\zf_1(a)\;.$$ Conversely, we obtain $$a\triangleleft_2 m=\zf(a)\star m=(a\triangleleft_1 1_M)\star m=a\triangleleft_1\zm_N[1_M\0 m]=a\triangleleft_1 m\;,$$ with self-explaining notation.\medskip

In fact, the two maps we just defined between the objects of the categories $\Ac\downarrow {\tt DG\Dc A}$ and $\tt CMon(\tt Mod_{\tt DG\Dc M}(\Ac))$, say $F$ and $G$, are functors and even an isomorphism of categories.\medskip

Indeed, if \[
\begin{tikzcd}[column sep=1.5em]
 & M' \arrow{dd}{\zy} \\
 \Ac \arrow[swap]{dr}{\psi''}\arrow{ur}{\psi'}\\
 & M''\;,
\end{tikzcd}
\] is a morphism $\Theta$ in $\Ac\downarrow {\tt DG\Dc A}$, then $F(\Theta)=\zy$ is a morphism in $\tt CMon(\tt Mod_{\tt DG\Dc M}(\Ac))$ between the modules $F(\psi')=M'$ and $F(\psi'')=M''$, with $\Ac$-action given by $$a\triangleleft' m'=\psi'(a)\star' m'$$ and similarly for $M''$. To prove this claim, it suffices to check that $\zy$ is $\Ac$-linear: $$\zy(a\triangleleft' m')=\zy(\psi'(a)\star' m')=\psi''(a)\star'' \zy(m')=a\triangleleft''\zy(m')\;.$$ Conversely, if $\zp:N'\to N''$ is a morphism in $\tt CMon(\tt Mod_{\tt DG\Dc M}(\Ac))$ and if $\zf':\Ac\to N'$ is the morphism (\ref{AlgMorph}) in $\tt DG\Dc A$ defined by $a\mapsto a\triangleleft'1_{N'}$ and similarly for $\zf''$, then $G(\zp)$, given by the commutative triangle
\[
\begin{tikzcd}[column sep=1.5em]
 & N' \arrow{dd}{\zp} \\
 \Ac \arrow[swap]{dr}{\zf''}\arrow{ur}{\zf'}\\
 & N''\;,
\end{tikzcd}
\]
is a morphism $\zP$ in $\Ac\downarrow {\tt DG\Dc A}$ between $G(N')=\zf'$ and $G(N'')=\zf''\,.$ Eventually, the maps $F$ and $G$ are actually functors, and, as verified above, the composites $FG$ and $GF$ coincide with the corresponding identity functors on objects. It is easily seen that the same holds on morphisms. This is clear for $FG$, whereas for $GF$ one has to notice that $\zf'(a)=a\triangleleft'1_{N'}=\psi'(a)\star'1_{N'}=\psi'(a)$ and analogously for $\zf''$.

\end{proof}

\subsection{Differential graded $\Dc$-algebras and modules over them as algebras over a monad}

For the purpose of further studying the category $\tt Mod_{DG\Dc M}(\Ac)$ of modules in $\tt DG\Dc M$ over an algebra $\Ac\in\tt DG\Dc A$, as well as the category $\tt DG\Dc A$ itself, we rely on results of \cite{SS98}. To be able to apply the latter, we must view the two preceding categories as categories of algebras over monads. For details on adjunctions, monads and algebras over them, we refer to Appendix \ref{B:monads}. Information on locally presentable categories can be found in Appendix \ref{LPC}.

\subsubsection{Differential graded $\Dc$-algebras}\label{DGDAs}

Consider the adjunction $$\Sc:\tt DG\Dc M\rightleftarrows DG\Dc A:\op{F}\;,$$ where $\Sc$ is the graded symmetric tensor product functor and $\op{F}$ the forgetful functor \cite{DPP}. This Hom-set adjunction can be viewed as a unit-counit adjunction $\langle \Sc,\op{F},\zh,\ze\rangle$. It implements a monad $\langle T,\zm,\zh\rangle=\langle \op{F}\Sc,\op{F}\ze\Sc,\zh\rangle$ in $\tt DG\Dc M$.

\begin{prop}\label{L:Talg}
The category $\tt DG\Dc A$ of differential graded $\Dc$-algebras and the Eilenberg-Moore category ${\tt DG\Dc M}\,^T$ of $T$-algebras in $\tt DG\Dc M$ are equivalent.
\end{prop}

\proof The statement is true if the forgetful functor $\op{F}$ is monadic. This can be checked using the crude monadicity theorem, see Appendix \ref{OMA}. However, there is a quicker proof. It is known \cite{Porst} that, if $\tt C$ is a symmetric monoidal, locally presentable category, and such that, for any $c\in\tt C$, the functor $c\0 \bullet$ respects directed colimits, then the forgetful functor $\op{For}: {\tt CMon(C)\to C}$ is monadic. Note first that the category $\tt C = DG\Dc M$ {\it is} locally presentable. The result can be proven directly, but follows also from \cite{Rosicky}. Moreover, this category is Abelian closed symmetric monoidal. In view of closedness, the functor $c\0\bullet$ is a left adjoint functor and respects therefore all colimits. Hence, the functor $\op{F}:\tt DG\Dc A\to DG\Dc M$ is monadic.
\endproof

\subsubsection{Modules over a differential graded $\Dc$-algebra}

Let $\Ac\in\tt DG\Dc A$ and consider the adjunction $$\zS:\tt DG\Dc M\rightleftarrows {\tt Mod_{DG\Dc M}}(\Ac):\op{\zF}\;,$$ where $\zS$ is the functor $\Ac\0_\bullet -$ and $\zF$ the forgetful functor. Checking that these functors really define an adjunction, so that, for any $M\in\tt DG\Dc M$, the product $\zS(M)=\Ac\0_\bullet M$ is the free $\Ac$-module in $\tt DG\Dc M$, is straightforward. When interpreting this Hom-set adjunction as a unit-counit adjunction $\langle \zS,\zF,\zh,\ze\rangle$, we get an induced monad $\label{e:Umon} \langle U,\zm,\zh\rangle =\langle \zF\,\zS,\zF\, \ze\, \zS,\zh \rangle$ in $\tt DG\Dc M$.

\begin{prop}\label{L:Ualg}
The category $\tt Mod_{DG\Dc M}(\Ac)$ of $\Ac$-modules in $\tt DG\Dc M$ and the Eilenberg-Moore category ${\tt DG\Dc M}^{\,U}$ of $U$-algebras in $\tt DG\Dc M$ are equivalent.
\end{prop}

We address the proof of this proposition later on. In view of the requirements of a Homotopical Algebra Context, we will show that the model structure of $\tt DG\Dc M$ can be lifted to $\tt Mod_{DG\Dc M}(\Ac)$:

\begin{thm}\label{P:AmodCGMMC} The category $\tt Mod_{DG\Dc M}(\Ac)$, $\Ac\in\tt DG\Dc A$, is a cofibrantly generated symmetric monoidal model category that satisfies the monoid axiom (see below). For its monoidal structure we refer to Proposition \ref{mod2}. The weak equivalences and fibrations are those $\Ac$-module morphisms $\psi$ whose underlying $\tt DG\Dc M$-morphisms $\zF(\psi)$ are weak equivalences or fibrations, respectively. The cofibrations a defined as the morphisms that have the {\small LLP} with respect to the trivial fibrations. The set of generating cofibrations (resp., generating trivial cofibrations) is made of the image $\zS(I)=\left\{\Id_\Ac\otimes_\bullet\, \iota_n: \iota_n\in I \right\}$ (resp., $\zS(J)=\left\{\Id_\Ac\otimes_\bullet\,\zeta_n:\zeta_n\in J \right\}$) of the set $I$ of generating cofibrations (resp., the set $J$ of generating trivial cofibrations) of $\tt DG\Dc M$.
\end{thm}

The proof will turn out to be a consequence of \cite[Theorem 4.1(2)]{SS98}. For convenience, we recall that this theorem states that, if $\tt C$, here $\tt DG\Dc M$, is a cofibrantly generated symmetric monoidal model category, which satisfies the monoid axiom and whose objects are small relative to the entire category, then, for any $\Ac\in\tt CMon(DG\Dc M)=DG\Dc A$, the category ${\tt DG\Dc M}^{\,U}$ is a cofibrantly generated symmetric monoidal model category satisfying the monoid axiom.  The monoidal and model structures are defined as detailed in Theorem \ref{P:AmodCGMMC}. The model part of this result \cite[Proofs of Theorems 4.1(1) and 4.1(2)]{SS98} is a direct consequence of \cite[Lemma 2.3]{SS98}. This allows in fact to conclude also that the generating sets of cofibrations and trivial cofibrations are the sets $\zS(I)$ and $\zS(J)$ described in \ref{P:AmodCGMMC}.\medskip

Since any chain complex of $\Dc$-modules is small relative to all chain maps, any object in $\tt DG\Dc M$ is small relative to all $\tt DG\Dc M$-morphisms. Hence, to finish the proof of Theorem \ref{P:AmodCGMMC}, it suffices to check that $\tt DG\Dc M$ satisfies the monoid axiom:

\begin{defi} A monoidal model category $\tt C$ satisfies the {\em monoid axiom} \cite[Definition 3.3]{SS98}, if any $\op{TrivCof}\0\,\tt C$-cell $(\,$a concise definition of cells can be found, for instance, in \cite[Appendix 6]{DPP}$\,)$, i.e., any cell with respect to the class of the tensor products $\zf\0\Id_C:C'\0 C\to C''\0 C$ of a trivial cofibration $\zf:C'\to C''$ and the identity of an object $C\in\tt C$, is a weak equivalence.\end{defi}

If $\tt C$ is cofibrantly generated and closed symmetric monoidal, the monoid axiom holds, if any $J\0\tt C$-cell, where $J$ is a set of generating trivial cofibrations of $\tt C$, is a weak equivalence \cite[Lemma 3.5]{SS98}.\medskip

Hence, to prove that $\tt DG\Dc M$ satisfies the monoid axiom, it suffices to show that a $J\0_\bullet\tt DG\Dc M$-cell, i.e., a transfinite composition of pushouts of morphisms in $J\0_\bullet \tt DG\Dc M$, is a weak equivalence. Since $\tt DG\Dc M$ is a finitely generated model category \cite{DPP} and the domains and codomains of its generating cofibrations $I$ are finite, i.e., $n$-small ($n\in\N$), relative to the whole category \cite[Lemma 2.3.2]{Ho99}, weak equivalences are closed under transfinite compositions \cite[Corollary 7.4.2]{Ho99}. Therefore, it is enough to make sure that a pushout of a morphism $\zeta_n\0_\bullet\Id_M\in J\0_\bullet\tt DG\Dc M$ ($n\ge 1,\,M\in\tt DG\Dc M$) is a weak equivalence. Here $\zeta_n:0\to D^n$ and $$ D^n: \cdots \to 0\to 0\to \stackrel{(n)}{\Dc} \to \stackrel{(n-1)}{\Dc}\to 0\to \cdots\to \stackrel{(0)}{0}\;.$$ Using standard arguments that have already been detailed above, one easily checks that the pushout of $\zeta_n\0_\bullet\Id_M:0\to D^n\0_\bullet M$ along a morphism $\zf:0\to N$ is given by

\begin{figure}[h]
\begin{center}
\begin{tikzpicture}
  \matrix (m) [matrix of math nodes, row sep=3em, column sep=3em]
    {  0 & N  \\
       D^n\0_\bullet M & (D^n\0_\bullet M)\oplus N \\ };
 \path[->]
 (m-1-2) edge  node[right] {$\scriptstyle{i_2}$} (m-2-2);
 \path[->]
 (m-1-1) edge  node[above] {$\scriptstyle{\zf}$} (m-1-2);
  \path[->]
 (m-1-1) edge  node[left] {$\scriptstyle{\zeta_n\0_\bullet\Id_M}$} (m-2-1);
  \path[->]
 (m-2-1) edge  node[above] {$\scriptstyle{i_1}$} (m-2-2);
\end{tikzpicture}
\end{center}
\end{figure}

\noindent Applying K\"unneth's Theorem to the complexes $D^n$ and $M$ -- noticing that both, $D^n$ and $d(D^n)$ (which vanishes, except in degree $n-1$, where it coincides with $\Dc$), are termwise flat $\mathcal{O}$-modules (see Proposition \ref{Projectivity}; for a direct proof, see \cite{DPP})  --\nolinebreak, we get, for any $m$, a short exact sequence

\begin{align}\nonumber
0\rightarrow \bigoplus_{p+q=m}H_p(D^n)\otimes H_q(M)&\rightarrow H_m(D^n\0_\bullet M)\rightarrow\\ \nonumber
&\bigoplus_{p+q=m-1} \op{Tor}_1(H_p(D^n), H_q(M))\rightarrow 0\;.
\end{align}

\noindent Since $D^n$ is acyclic, the central term of this exact sequence vanishes, as the first and the third do. Eventually, the pushout $i_2$ of $\zeta_n\0_\bullet\Id_M$ is a weak equivalence, since $$H(i_2):H(N)\to H(D^n\0_\bullet M)\oplus H(N)\simeq H(N)$$ is obviously an isomorphism.\medskip

The category $\tt DG\Dc M$ thus satisfies all the conditions of \cite[Theorem 4.1(2)]{SS98}. It now follows from \cite[Proofs of Theorems 4.1(1) and 4.1(2)]{SS98} that the category $\tt Mod_{DG\Dc M}(\Ac)$ is equivalent to the category ${\tt DG\Dc M}^{\,U}$ (the result can also be obtained via the crude monadicity theorem). This completes the proofs of Proposition \ref{L:Ualg} and Theorem \ref{P:AmodCGMMC}.

\begin{rem} For any $\Ac\in\tt CMon(DG\Dc M)= DG\Dc A$, an $\Ac$-algebra $A\in\tt Alg_{DG\Dc M}(\Ac)$ is defined in \cite{SS98} as a monoid $A\in\tt Mon(Mod_{DG\Dc M}(\Ac))$. Theorem 4.1(3) in \cite{SS98} states that $\tt Alg_{DG\Dc M}(\Ac)$ is a cofibrantly generated model category. When choosing $\Ac=\Oc$, we find that $\tt Alg_{DG\Dc M}(\Oc)= Mon(Mod_{DG\Dc M}(\Oc))=\tt Mon(DG\Dc M)$ is cofibrantly generated. However, the theorem does not treat the case of commutative $\Oc$-algebras, of commutative monoids in $\tt DG\Dc M$, or, still, of differential graded $\Dc$-algebras $\tt DG\Dc A$. The fact that $\tt DG\Dc A$ is a cofibrantly generated model category has been proven independently in \cite{DPP}. \end{rem}

\begin{rem} In the sequel, we write $\tt Mod(\Ac)$ instead of $\tt Mod_{DG\Dc M}(\Ac)$, whenever no confusion arises.\end{rem}

\subsection{Cofibrant objects in $\tt Mod(\Ac)$}

In this last Subsection, we describe cofibrant $\Ac$-modules. We need a similar lemma as \cite[Lemma 1]{DPP2} that we used to characterize cofibrations in $\tt DG\Dc A$.

\begin{lem}\label{LemAmod} Let $(\Ac,d_\Ac)\in\tt DG\Dc A$, $(T,d_T)\in \tt Mod(\Ac)$, let $(g_j)_{j\in J}$ be a family of symbols of degree $n_j\in \NN$, and let $V=\bigoplus_{j\in J}\Dc\cdot g_j$ be the free non-negatively graded $\Dc$-module with homogeneous basis $(g_j)_{j\in J}$.\smallskip

(i) To endow the graded $\mathcal{D}$-module $T\oplus \Ac\otimes_\bullet V$ -- equipped with the naturel $\Ac$-module structure induced by the $\Ac$-actions of $T$ and $\Ac\0_\bullet V$ -- with a differential $d$ that makes it an $\Ac$-module, it suffices to define \be\label{CondAmodDiff}d (g_j)\in T_{n_j-1}\cap d_T^{-1}\{0\}\;,\ee to extend $d$ as $\Dc$-linear map to $V$, and to finally define $d$ on $T\oplus \Ac\otimes_\bullet V$, for any $t\in T_p, a\in \Ac_{k}, v\in V_{p-k}\,$, by
   \be\label{DefAmodDiff}d(t\oplus a\otimes v)=d_T(t)+d_\Ac(a)\otimes v+(-1)^k a\triangleleft d(v)\;,\ee
where $\triangleleft$ is the $\Ac$-action on $T$. The inclusion
    $$(T,d_T)\hookrightarrow (T\oplus \Ac\otimes_\bullet V,d)$$
is a morphism of $\Ac$-modules. Moreover, the differential (\ref{DefAmodDiff}) is the unique differential that restricts to $d_T$ on $T$, maps $V$ into $T$ and provides an $\Ac$-module structure on the graded $\Dc$-module $T\oplus \Ac\otimes_\bullet V$ equipped with its natural $\Ac$-action.\smallskip

(ii) If $(B,d_B)\in{\tt Mod(\Ac)}$ and $p\in\op{Hom}_\Ac(T,B)$, it suffices -- to define a morphism $q\in{\op{Hom}_\Ac}(T\oplus \Ac\otimes_\bullet V,B)$ (where the $\Ac$-module $(T\oplus \Ac\otimes_\bullet V,d)$ is constructed as described in (i)) -- to define
 \be\label{CondAmodMorph}q(g_j)\in B_{n_j}\cap d_B^{-1}\{p\,d(g_j)\}\;,\ee to extend $q$ as $\Dc$-linear map to $V$, and to eventually define $q$ on $T\oplus \Ac\0_\bullet V$ by
    \be\label{DefAmodMorph}q(t\oplus a\otimes v)=p(t)+a\triangleleft q(v)\;,\ee
where $\triangleleft$ is the $\Ac$-action on $B$. Moreover, \ref{DefAmodMorph} is the unique $\Ac$-module morphism $(T\oplus \Ac\otimes_\bullet V,d)\to (B,d_B)$ that restricts to $p$ on $T$.
\end{lem}

Note that Condition (\ref{CondAmodDiff}) corresponds to the characterizing lowering condition in relative Sullivan $\Dc$-algebras \cite{DPP}.

\begin{proof}
(i) It is straightforward to see that $d$ is a well-defined, degree $-1$ and $\Dc$-linear map on $T\oplus\Ac\0_\bullet V$. It squares to zero, since the $\Ac$-action $-\triangleleft-=\zn$ on $T$ commutes with the differentials on $\Ac\0_\bullet T$ and $T$, $$d_T(a\triangleleft d(v))=d_T(\zn(a\0 d(v)))=\zn(d_\Ac(a)\0 d(v)+(-1)^ka\0 d_T(d(v)))=d_\Ac(a)\triangleleft d(v)\;,$$ and thus compensates the other non-vanishing term in $d^2(t+a\0 v)$. Hence, $T\oplus\Ac\0_\bullet V\in\tt DG\Dc M$. Its naturel $\Ac$-action -- also denoted by $-\triangleleft-$ -- endows it with an $\Ac$-module structure, if it commutes with the differentials $d_\Ac\0\Id+\Id\0\, d$ of $\Ac\0_\bullet(T\oplus\Ac\0_\bullet V)$ and $d$ of $T\oplus\Ac\0_\bullet V$. This condition is easily checked, so that $T\oplus\Ac\0_\bullet V$ is actually an $\Ac$-module for the differential $d$ and the $\Ac$-action $$a'\triangleleft(t+a''\0 v)=a'\triangleleft t +(a'a'')\0 v\;.$$ It is clear that $T$ is an $\Ac$-submodule of $T\oplus\Ac\0_\bullet V$. Concerning uniqueness, let $\p$ be any differential that has the required properties. Then, $$\p(t+a\0 v)=d_T(t)+\p(a\triangleleft (1_\Ac\0 v))=$$ $$d_T(t)+d_\Ac(a)\triangleleft(1_\Ac\0 v)+(-1)^ka\triangleleft\p(v)=d_T(t)+d_\Ac(a)\0 v+(-1)^ka\triangleleft\p(v)\;,$$ with $$\p(g_j)\in T_{n_j-1}\cap d_T^{-1}\{0\}\;.$$

(ii) Similar proof.
\end{proof}

We are now prepared to study cofibrant $\Ac$-modules. This description will be needed later on. Let us recall that the cofibrations in $\tt DG\Dc A$, or, equivalently, in ${\tt DG\Dc M}^T$ -- where $T$ is the composite of the free differential graded $\Dc$-algebra functor $\Sc$ (symmetric tensor product functor) and the forgetful functor -- , are the retracts of relative Sullivan $\Dc$-algebras $(\Bc\0\Sc V,d)$ \cite{DPP}. We will prove that, similarly, cofibrant objects in $\tt Mod(\Ac)$, or, equivalently, in ${\tt DG\Dc M}^U$ -- where $U$ is the composite of the free $\Ac$-module functor $\Ac\0_\bullet - $ and the forgetful functor -- , are retracts of `Sullivan $\Ac$-modules'. If one remembers that the binary coproduct in $\tt Mod(\Ac)$ (resp., $\tt DG\Dc A$) is the direct sum (resp., tensor product), and that the initial object in $\tt Mod(\Ac)$ (resp., $\tt DG\Dc A$) is $(\{0\},0)$ (resp., $(\Oc,0)$), the definition of relative Sullivan $\Ac$-modules is completely analogous to that of relative Sullivan $\Dc$-algebras \cite{DPP}:

\begin{defi} Let $\Ac\in\tt DG\Dc A$. A {\em relative Sullivan $\Ac$-module} ({\small RS$\Ac$M}) is a $\tt Mod(\Ac)$-morphism $$(B,d_B)\to (B\oplus\Ac\0_\bullet V,d)\;$$that sends $b\in B$ to $b+0\in B\oplus\Ac\0_\bullet V$. Here $V$ is a free non-negatively graded $\Dc$-module, which admits a homogeneous basis $(m_\za)_{\za<\zl}$ that is indexed by a well-ordered set, or, equivalently, by an ordinal $\zl$, and is such that \be\label{Lowering}d m_\za\in B\oplus \Ac\0_\bullet V_{<\za}\;,\ee for all $\za<\zl$. In the last requirement, we set $V_{<\za}:= \bigoplus_{\zb<\za}\Dc\cdot m_\zb$. We translate (\ref{Lowering}) by saying that the differential $d$ is {\em lowering}. A {\small RS$\Ac$M} over $(B,d_B)=(\{0\},0)$ is called a {\em Sullivan $\Ac$-module} ({\small S$\Ac$M}) $(\Ac\0_\bullet V, d)$.\end{defi}

In principle the free $\Ac$-module functor is applied to $(M,d_M)\in \tt DG\Dc M$ and leads to $(\Ac\0_\bullet M,d_{\Ac\0_\bullet M})\in\tt Mod(\Ac)$. In the preceding definition, this functor is taken on $V\in\tt G\Dc M$ and provides a graded $\Dc$-module with an $\Ac$-action. The latter is endowed with a lowering differential $d$ such that $(\Ac\0_\bullet V,d)\in\tt Mod(\Ac)$.

\begin{thm}\label{L:cofobjsinAMod} Let $\Ac\in\tt DG\Dc A$. Any cofibrant object in $\tt Mod(\Ac)$ is a retract of a Sullivan $\Ac$-module and vice versa.
\end{thm}

Since we do not use the fact that any retract of a Sullivan module is cofibrant, we will not prove this statement.

\proof By Proposition \ref{P:AmodCGMMC}, the model category $\tt Mod(\Ac)$ is cofibrantly generated. Cofibrations are therefore retracts of morphisms in $\zS(I)$-cell \cite[Proposition 2.1.18 (b)]{Ho99}, i.e., they are retracts of transfinite compositions of pushouts of generating cofibrations $\zS(I)$.\medskip

We start studying the pushout of a generating cofibration
$$
\zS(\iota_n):=\Id_\Ac\otimes_\bullet \iota_n: \Ac\otimes_\bullet S^{n-1}\rightarrow \Ac\otimes_\bullet D^{n}
$$
along a $\tt Mod(\Ac)$-morphism $f:\Ac\otimes_\bullet S^{n-1}\rightarrow B$, where $n>0$ (the case $n=0$ is simpler). This pushout is given by the square

\begin{equation}\label{e:pushchar}
\begin{tikzcd}[column sep=small]
(\Ac\otimes_\bullet S^{n-1}, d_{\Ac\otimes_\bullet S^{n-1}})\arrow{r}{f} \arrow[swap]{d}{\zS(\iota_n)} & (B, d_B) \arrow{d}{h}\\
(\Ac\otimes_\bullet D^{n}, d_{\Ac\otimes_\bullet D^{n}})\arrow{r}{g} & (B\oplus \Ac\otimes_\bullet S^n, d)\quad ,\\
\end{tikzcd}
\end{equation}
where the differential $d$ and the $\tt Mod(\Ac)$-maps $g$ and $h$ are defined as follows.\medskip

Observe that $(\Ac\0_\bullet S^{n-1},d_{\Ac\0_\bullet S^{n-1}})$ meets the requirements of point (i) of Lemma \ref{LemAmod}. Indeed, if $1_{n-1}$ is the basis of $S^{n-1}$, the differential $\zd$ constructed in Lemma \ref{LemAmod} satisfies $\zd(1_{n-1})=0$ and $$\zd(a\0 (\zD\cdot 1_{n-1}))=d_\Ac(a)\0 (\zD\cdot 1_{n-1})=d_{\Ac\0_\bullet S^{n-1}}(a\0 (\zD\cdot 1_{n-1}))\;,$$ where $\zD\cdot 1_{n-1}$ denotes the action of $\zD\in\Dc$ on $1_{n-1}$. It now follows from point (ii) of Lemma \ref{LemAmod} that $f$ is completely determined by its value $f(1_{n-1})\in B_{n-1}\cap d_B^{-1}\{0\}$ (we identify $\zD\cdot 1_{n-1}$ with $1_\Ac\0 (\zD\cdot 1_{n-1}))$.\medskip

Using again Lemma \ref{LemAmod}.(i), we define $d$ as the unique differential on $B\oplus \Ac\otimes_\bullet S^n$ satisfying
$$
d|_B=d_B \qquad \text{and}\qquad d(1_{n})=f(1_{n-1})\;.
$$

The morphism $h$ is defined as the inclusion of $(B,d_B)$ into $(B\oplus \Ac\otimes_\bullet S^n,d)$.\smallskip

As for $g$, we define it as $h\circ f$ on $\Ac\0_\bullet S^{n-1}$. Then we set $T=\Ac\otimes_\bullet S^{n-1}$ and $V=S^n$, and observe that the differential $\p$ on $$T\oplus \Ac\0_\bullet V = \Ac\0_\bullet D^n\;,$$ given by $$\p(1_n)=1_{n-1}\in T_{n-1}\cap\, d_T^{-1}\{0\}\;,$$ coincides with the differential $d_{\Ac\0_\bullet D^n}$. In view of this observation, the $\tt Mod(\Ac)$-map $g$ can be defined as the extension of $h\circ f$ to $\Ac\0_\bullet D^n$. Part (ii) of Lemma \ref{LemAmod} allows to see that $g$ is now fully defined by $$g(1_n)=1_n\in (B\oplus \Ac\0_\bullet S^n)_n\cap d^{-1}\{h(f(\p (1_n)))\}\;.\vspace{1mm}$$

Next, we prove that the diagram \eqref{e:pushchar} commutes and is universal among all such diagrams.\smallskip

A concerns commutativity, note that any element of $\Ac\0_\bullet S^{n-1}$ is a finite sum of elements $a\0 (D\cdot 1_{n-1})=a\triangleleft(D\cdot 1_{n-1})$, so that the two $\tt Mod(\Ac)$-maps $h\circ f$ and $g\circ \zS(\iota_n)$ coincide if they do on $1_{n-1}$ -- what is a direct consequence of the preceding definitions.\smallskip

To prove universality, consider any $\Ac$-module $(C,d_C)$, together with two $\tt Mod(\Ac)$-morphisms $$p: (\Ac\otimes_\bullet D^{n}, d_{\Ac\otimes_\bullet D^{n}})\rightarrow (C,d_C)$$ and $q: (B,d_B)\rightarrow (C,d_C)$, such that $q\circ f = p\circ \zS(\iota_n)\,,$ and show that there is a unique $\tt Mod(\Ac)$-map $$u:(B\oplus \Ac\otimes_\bullet S^n, d)\rightarrow (C,d_C)$$ that renders commutative the `two triangles'.\smallskip

When extending $q$ by means of Lemma \ref{LemAmod} to a $\tt Mod(\Ac)$-map on $B\oplus \Ac\0_\bullet S^n$, we just have to define $u(1_n)\in C_n\cap d_C^{-1}(p(1_{n-1}))$. Observe that, if $u$ exists, we have necessarily $u|_B=u\circ h=q$ and $u(1_n)=u(g(1_n))=p(1_n)$. It is easily seen that the latter choice satisfies the preceding conditions and that $u$ is unique. Notice that, obviously, $u\circ h=q$ and that $u\circ g=p$, since this equality holds on $1_{n}$ and $1_{n-1}$: for $1_{n-1}$, we have
\begin{align*}
u(g(1_{n-1}))=&\,u(f(1_{n-1}))=\,q(f(1_{n-1}))=\,p(\zS(\iota_n)(1_{n-1}))=\,p(1_{n-1})\;.
\end{align*}
Finally (\ref{e:pushchar}) is indeed the pushout diagram of $\zS(\iota_n)$ along $f$.\medskip

The maps in $\zS(I)$-cell are the transfinite compositions of such pushout diagrams. A transfinite composition of pushouts is the colimit of a colimit respecting functor $X:\zl\to \tt Mod(\Ac)$ (where $\zl$ is an ordinal), such that the maps $X_\zb\to X_{\zb+1}$ ($\zb+1<\zl$) are pushouts of generating cofibrations $\zS(I)$. Let therefore $$X_0\to X_1\to \ldots \to X_\zb\to X_{\zb+1}\to \ldots$$ be such a functor. The successive maps are pushouts in the category $\tt Mod(\Ac)$: $$X_0=B, X_1=B\oplus\Ac\0_\bullet S^{n(1)},\ldots,X_\zb=B\oplus\Ac\0_\bullet\bigoplus_{\za\le\zb}S^{n(\za)},\ldots,$$ $$X_\zw=B\oplus\Ac\0_\bullet\bigoplus_{\za<\zw}S^{n(\za)}, X_{\zw+1}=X_\zw\oplus\Ac\0_\bullet S^{n(\zw+1)},\ldots\;,$$ where any $n(\za)\in\N$. It follows that the transfinite composition or colimit is $$\op{colim}_{\za<\zl}X_\za=B\oplus\Ac\0_\bullet\bigoplus_{\za<\zl,\za\in\mathbf{O}_s}S^{n(\za)}\;,$$ where $\mathbf{O}_s$ denotes the successor ordinals, or, better, the composition is the $\tt Mod(\Ac)$-map \be\label{TransCompPushGen}(B,d_B)\to \left(B\oplus\Ac\0_\bullet\bigoplus_{\za<\zl,\za\in\mathbf{O}_s}S^{n(\za)},d\right)\;,\ee where $d$ is defined by $d\,|_{X_\za}=d_{X_\za}$ ($\za\in\mathbf{O}_s$) and $d_{X_\za}$ is defined inductively by $d_{X_\za}|_{X_{\za-1}}=d_{X_{\za-1}}$ and by $d_{X_\za}(1_{n(\za)})=f_\za(1_{n(\za)-1})\,,$ with self-explaining notation (if $\za=\zw+1$, then $d_{X_{\za-1}}=d_{X_\zw}$ is defined by its restrictions to the $X_\zb$, $\zb<\zw$). Eventually, any $\zS(I)$-cell is a relative Sullivan $\Ac$-module and any cofibration is a retract of a relative Sullivan $\Ac$-module.\medskip

Let now $C$ be a cofibrant $\Ac$-module and let $QC$ be its cofibrant replacement, given by the small object argument \cite[Theorem 2.1.14]{Ho99}: the $\tt Mod(\Ac)$-map $z':0\rightarrow QC$ is in $\zS(I)$-cell $\subset\op{Cof}$, hence it is a relative Sullivan $\Ac$-module. Moreover, in the commutative diagram

\begin{equation}\label{e:trick}
\begin{tikzcd}[column sep=large]
0\arrow{r}{z'} \arrow{d}{z} & QC \arrow{d}{z''}\\
C\arrow[swap]{r}{\Id_C} \arrow[dashed]{ur}{\ell} & C\\
\end{tikzcd},
\end{equation}
the right-down arrow $z''$ is in $\op{TrivFib}$ and the left-down $z$ in $\op{Cof}=\op{LLP(TrivFib)}$, so that the dashed $\tt Mod(\Ac)$-arrow $\ell$ does exist. The diagram encodes the information $z''\circ\ell=\op{id_C}$, i.e., the information that the cofibrant $C\in\tt Mod(\Ac)$ is a retract of the Sullivan $\Ac$-module $QC$.
\endproof

\section{Homotopical Algebraic Context for $\tt DG\Dc M$}

A {\em Homotopical Algebraic Context} ({\small HAC}) is a context that satisfies several minimal requirements for the development of Homotopical Algebra within this setting. Such a context is a triplet $(\tt C, C_0,A_0)$ made of a symmetric monoidal model category $\tt C$ and two full subcategories $\tt C_0\subset\tt C$ and $A_0\subset\tt CMon(C)$, which satisfy assumptions that will be recalled and commented below.\medskip

We will show that the triplet \[
(\tt DG\Dc M, DG\Dc M, DG\Dc A)
\] is a {\small HAC}. Therefore, some preparation is needed.

\subsection{Transfinite filtrations and graduations}

We start with a useful lemma. If $\Ac\in\tt DG\Dc A$ and $M\in\tt G\Dc M$, the tensor product $\Ac\0_\bullet M$ can be an object in $\tt Mod(\Ac)$, in essentially two ways. If $M$ comes with its own differential, i.e., if $M\in\tt DG\Dc M$, the natural choice for the differential on $\Ac\0_\bullet M$ is the standard differential on a tensor product of complexes. If, on the contrary, $M$ has no own differential, the tensor product can be a Sullivan $\Ac$-module. We will tacitely use the following

\begin{lem}\label{SimplTensA} Let $\Ac\in\tt DG\Dc A$, $B\in\tt Mod(\Ac)$, and $M\in\tt G\Dc M$ such that $\Ac\0_\bullet M\in \tt Mod(\Ac)$. Then, the $\Ac$-module $B\0_\Ac(\Ac\0_\bullet M)$ and the $\Ac$-module $B\0_\bullet M$ -- with canonical $\Ac$-action and transferred differential coming from the differential of $B\0_\Ac(\Ac\0_\bullet M)$ -- are isomorphic as $\Ac$-modules. If $M\in\tt DG\Dc M$ and the differential of $\Ac\0_\bullet M$ is the standard differential, the transferred differential on $B\0_\bullet M$ is also the standard differential, so that the isomorphism of $\Ac$-modules holds with standard differentials. \end{lem}

\begin{proof} We consider first the general case. The $\Ac$-action on the graded $\Dc$-module $B\0_\bullet M$ is the natural action implemented by the action of $B$. Set now \be\label{i}\imath: B\0_\Ac(\Ac\0_\bullet M)\ni b\0(a\0 m)\mapsto (-1)^{|a||b|}a\triangleleft (b\0 m)\in B\0_\bullet M\;.\ee It can straightforwardly be checked that $i$ is a well-defined isomorphism of $\Ac$-modules in $\tt G\Dc M$. Let now $d_B$ (resp., $d$) be the differential of $B$ (resp., $\Ac\0_\bullet M$). The differential \be\label{diff}\p:=\imath\circ(d_B\0\Id_{\0_\bullet}+\Id_B\0\, d)\circ \imath^{-1}\;\ee makes $B\0_\bullet M$ an $\Ac$-module and $\imath$ an isomorphism of $\Ac$-modules. The particular case mentioned in the lemma is obvious.\end{proof}

We also need in the following some results related to $\zl$-filtrations, where $\zl\in\mathbf{O}$ is an ordinal. Recall first that, if $\tt C$ is a cocomplete category, the colimit is a functor $\op{colim}:{\tt Fun}(\zl,{\tt C})\to \tt C$, whose source is the category ${\tt Fun}(\zl,{\tt C})$ of diagrams of type $\zl$ in $\tt C$.

\begin{Defi}
Let $\zl\in\mathbf{O}$ be an ordinal and let $\tt C\in\tt Cat$ be a category, which is closed under small colimits. An object $C\in\tt C$ is \emph{$\zl$-filtered}, if it is the colimit $C=\op{colim}_{\zb<\zl}\op{F}_\zb C$ of a $\zl$-sequence of $\tt C$-monomorphisms, i.e., of a colimit respecting functor $\op{F}C:\zl\to \tt C$, such that all maps $\op{F}_{\zb,\zb+1}C:\op{F}_\zb C\rightarrow \op{F}_{\zb+1}C$, $\zb+1<\zl$, are $\tt C$-monomorphisms:
\[
\op{F}_0 C\rightarrow \op{F}_1C\rightarrow\ldots\rightarrow \op{F}_{\zg}C\rightarrow\ldots\;
\]
The family $(\op{F}_\zb C)_{\zb\in\zl}$ is called a \emph{$\zl$-filtration} of $C$.\smallskip

Let $(\op{F}_\zb C)_{\zb\in \lambda}$ and $(\op{F}_\zb D)_{\zb\in \lambda}$ be $\lambda$-filtrations of $C\in\tt C$ and $D\in\tt C$, respectively. A $\tt C$-morphism $f:C\to D$ is \emph{compatible with the $\lambda$-filtrations}, if it is the colimit $f=\op{\colim}_{\zb<\zl}\zvf_\zb$ of a natural transformation $\zvf:\op{F}C\to \op{F}D$:
 \begin{equation}\nonumber
 \begin{tikzcd}
\op{F}_0C \ar[r] \ar{d}{\zvf_{0}} &\op{F}_1C \ar{r} \ar{d}{\zvf_{1}} &\ldots \ar[r]&\op{F}_{\zg}C \ar{r} \ar{d}{\zvf_{\zg}} & \ldots \\
\op{F}_0D \ar[r]    &\op{F}_1D \ar{r} &\ldots \ar[r]&\op{F}_{\zg}D \ar{r}   & \ldots
\end{tikzcd}
\end{equation}
\end{Defi}

In the first two lemmas below, we replace our standard category $\tt DG\Dc M$ by the more general category ${\tt DG}\text{\it\small R}\,{\tt M}$, where $R$ is, as usual, an arbitrary unital ring. For the model structure on ${\tt DG}\text{\it\small R}\,{\tt M}$, we refer to \cite{DPP}, as well as to references therein.

\begin{lem}\label{L:limitord}
Consider a nonzero ordinal $\zl\in\mathbf{O}\setminus\{0\}$, two $\zl$-filtered chain complexes $C,D\in{\tt DG}\text{\it\small R}\,{\tt M}$, with $\zl$-filtrations $(\op{F}_\zb C)_{\zb\in \lambda}$ and $(\op{F}_\zb D)_{\zb\in \lambda}$, and let $f:C\rightarrow D$ be a ${\tt DG}\text{\it\small R}\,{\tt M}$-morphism, which is compatible with the filtrations and whose corresponding natural transformation is denoted by $\zvf:\op{F}C\to \op{F}D$. If, for any $\zb<\zl$, the map $\zvf_\zb:\op{F}_\zb C\rightarrow \op{F}_\zb D$ is a weak equivalence in ${\tt DG}\text{\it\small R}\,{\tt M}$, then the same holds for $f$.
\end{lem}

\proof In the following, {\it we assume temporarily that $F_{\zb\zg}C$ and $F_{\zb\zg}D$ are injective, for all $\zb<\zg<\zl$}. Note first that, any ${\tt DG}\text{\it\small R}\,{\tt M}$-map $g:C'\to C''$ induces a ${\tt DG}\text{\it\small R}\,{\tt M}$-isomorphism $C'/\ker g\simeq \op{im}g$. Hence, for any $\zb\le\zg<\zl$, we get $$\op{F}_\zb C\simeq \op{im}(\op{F}_{\zb\zg}C)\subset\op{F}_\zg C\;.$$ This identification implies that $\op{F}_{\zb\zg}C$ is the canonical injection $$\op{F}_{\zb\zg}C:\op{F}_\zb C\hookrightarrow\op{F}_\zb C\subset\op{F}_\zg C$$ and that the differentials $\p_\zb,\p_\zg$ of $\op{F}_\zb C,\op{F}_\zg C$ satisfy $$\p_\zg|_{\op{F}_\zb C}=\p_\zb\;.$$ The same observation holds for $D$. For the natural transformation $\zvf$, we get $$\zvf_\zg|_{\op{F}_\zb C}=\zvf_\zb\;.$$

Recall now that a colimit in ${\tt DG}\text{\it\small R}\,{\tt M}$, say $C=\op{colim}_{\zb<\zl}\op{F}_\zb C$, is constructed degree-wise in ${\tt Mod}(R)$: $$C_n:=\coprod_{\zb<\zl}\op{F}_{\zb,n}C/\sim\;,$$ where $c_{\zb,n}\sim c_{\zg,n}$, if there is $\zd\ge \op{sup}(\zb,\zg),\zd<\zl$ such that $\op{F}_{\zb\zd}C\,(c_{\zb,n})=\op{F}_{\zg\zd}C\,(c_{\zg,n})$, i.e., $c_{\zb,n}=c_{\zg,n}.$ It follows that \be\label{LimCup}C_n=\bigcup_{\zb<\zl}\op{F}_{\zb,n} C\;.\ee The set $C_n$ can be made an object $C_n\in{\tt Mod}(R)$ in a way such that the maps $\zp_{\zb,n}:\op{F}_{\zb,n}C\to C_n$ become ${\tt Mod}(R)$-morphisms and $C_n$ becomes the colimit in ${\tt Mod}(R)$ of $\op{F}_nC:\zl\to {\tt Mod}(R)$. Due to (\ref{LimCup}), the maps $\zp_{\zb,n}$ are the canonical injections $$\zp_{\zb,n}:\op{F}_{\zb,n}C\hookrightarrow C_n\;.$$ Universality of the colimit allows to conclude that there is a ${\tt Mod}(R)$-morphism $\p_{n}:C_n\to C_{n-1}$ such that \be\label{LimDiff}\p_n|_{\op{F}_{\zb,n}C}=\p_{\zb}\;.\ee We thus get a complex $(C_\bullet,\p_\bullet)\in {\tt DG}\text{\it\small R}\,{\tt M}$, together with ${\tt DG}\text{\it\small R}\,{\tt M}$-morphisms \be\label{injcolim}\zp_{\zb,\bullet}:\op{F}_{{\zb,\bullet}}C\hookrightarrow C_\bullet\;,\ee and this complex is the colimit $C$ in ${\tt DG}\text{\it\small R}\,{\tt M}$ of $\op{F}C$ \cite{DPP}.\medskip

We have still to remove the temporary assumption. Note first the following:

\begin{rem} If $\zl\in\mathbf{O}$ and $X\in\op{Fun}(\zl,{\tt C})$ is a $\zl$-diagram in a cocomplete category $\tt C$, then, for any $\zb<\zg\le\zl$, the map $X_{\zb\ast}$, which assigns to any $\zg\setminus\zb$-object $\za$ the $\tt C$-morphism $X_{\zb\za}:X_\zb\to X_\za$, is a natural transformation between the constant functor $X_\zb$ and the functor $X$, both restricted to $\zg\setminus\zb$. The application of the colimit functor $\op{colim}:{\tt Fun}(\zg\setminus\zb,{\tt C})\to{\tt C}$ to this natural transformation leads to \be\label{Interpret1}\op{colim}_{\zb\le\za<\zg}X_{\zb\za}:X_\zb\to\op{colim}_{\zb\le\za<\zg}X_\za\;.\ee Further, a functor $G:{\tt D}'\to{\tt D}''$ preserves colimits, if, in case $(C,\psi)$ is the colimit of a diagram $F$ in ${\tt D}'$, then $(G(C),G(\psi))$ is the colimit of the diagram $G\,F$ in ${\tt D}''$. Hence, the functor $X:\zl\to \tt C$ preserves colimits means that, for a limit ordinal $\zg=\op{colim}_{\za<\zg}\za$ in $\zl$, i.e., for the colimit $(\zg,\zb<\zg)$ of the diagram $0\to 1\to\ldots\to\za\to\za+1\to\ldots(\zg)$ in $\zl$, the colimit of the diagram $X_0\to X_1\to\ldots\to X_\za\to X_{\za+1}\to\ldots(\zg)$ in $\tt C$ is $(X_\zg,X_{\zb\zg})$. In other words, \be\label{Interpret2}X_\zg=\op{colim}_{\zb\le\za<\zg}X_\za\quad\text{and}\quad X_{\zb\zg}=\op{colim}_{\zb\le\za<\zg}X_{\zb\za}\;.\ee\end{rem}

We are now prepared to show, by transfinite induction on $\zg$, that the {\it temporary} hypothesis assuming that $F_{\zb\zg}C$ (the case of $F_{\zb\zg}D$ is similar) is injective for $\zb<\zg<\zl$, is in fact a consequence of the {\it actual} assumptions of Lemma \ref{L:limitord}. The induction starts, since $F_{\zb,\zb+1}C$ is injective for $\zb+1<\zl$. The induction assumption is that $F_{\zb\za}C$ is injective for $\zb<\za<\zg$. In the case $\zg\in\mathbf{O}_s$, we have $F_{\zb\zg}C=F_{\zg-1,\zg}C\,F_{\zb,\zg-1}C$, which is injective, because the first acting map is injective in view of the induction assumption (or the fact that it is identity) and the second map is injective since $(F_\zb C)_{\zb<\zl}$ is a $\zl$-filtration. If $\zg\in\mathbf{O}_\ell$, $\zg=\op{colim}_{\za<\zg}\za$, it follows from (\ref{Interpret2}), applied to the colimit respecting functor $X=\op{F}C$, that $$F_{\zb\zg}C:F_\zb C\to F_\zg C\;$$ is the map $$\op{colim}_{\zb\le\za<\zg}F_{\zb\za}C:F_\zb C\to \op{colim}_{\zb\le\za<\zg}F_\za C\;.$$ Moreover, the equation (\ref{Interpret1}) shows that the map (\ref{injcolim}) is nothing but $\op{colim}_{\zb\le\za<\zl}F_{\zb\za}C$. When, at the beginning of the proof of Lemma \ref{L:limitord}, the role of $\zl$ is played by $\zg$, the temporary assumption is exactly the induction assumption, so that $F_{\zb\zg}=\op{colim}_{\zb\le\za<\zg}F_{\zb\za}C$ is the natural injection (\ref{injcolim}) for the considered case $\zl=\zg$, what eventually removes the temporary assumption.\medskip

In the sequel, we omit the subscript $\bullet\,$, as well as the index $n$ of chain maps and differentials.\medskip

When considering both colimits, $C$ and $D$, we use the above notation, adding a superscript $C$ or $D$, if confusion has to be avoided. Further, the colimit map $f=\op{colim}_{\zb<\zl}\zvf_\zb$ is obtained using the universality of the colimit $C=\op{colim}_{\zb<\zl}\op{F}_\zb C$. More precisely, the ${\tt DG}\text{\it\small R}\,{\tt M}$-morphisms $\zvf_\zb:\op{F}_\zb C\to D$ factor through $C$, i.e., \be\label{Univ}f|_{\op{F}_{\zb}C}=\zvf_\zb\;.\ee

We are now prepared to show that the ${\tt DG}\text{\it\small R}\,{\tt M}$-morphism $f$ induces an isomorphism of graded $R$-modules in homology.\medskip

If the induced degree zero ${\tt Mod}(R)$-morphism $H(f)$ is not injective, one of its components $H(f):H_n(C)\to H_n(D)$, has a non-trivial kernel, i.e., there is a $\p^C$-cycle $c_{n}$ that is not a $\p^C$-boundary, such that \be\label{NonTrivKer}f\,c_n=\p^D\,d_{n+1}\;.\ee We have \be\label{Classes}c_n=c_{\zb,n}\quad\text{and}\quad d_{n+1}=d_{\zg,n+1}\;,\ee for some $\zb,\zg<\zl.$ It is clear that $c_{\zb,n}$ is a $\p^C_{\zb}$-cycle, but not a $\p^C_{\zb}$-boundary. Moreover, $$\zvf_{\zb}\, c_{\zb,n}=f\,c_n=\p^D_{\zg}d_{\zg,n+1}\;.$$ Depending on whether $\zb\ge\zg$ or $\zb<\zg$, this contradicts the fact that $H(\zvf_{\zb})$ or that $H(\zvf_\zg)$ is an isomorphism. Therefore, we finally conclude that $H(f)$ is indeed injective. \medskip

As for the subjectivity of $H(f):H_n(C)\to H_n(D)$, let $v_n\in D_n\cap\ker\p^D$: $v_n=v_{\zb,n}$ and $v_{\zb,n}\in\op{F}_{\zb,n}D\cap \ker\p^D_\zb$. Since $H(\zvf_{\zb}):H(\op{F}_\zb C)\to H(\op{F}_\zb D)$ is surjective, the homology class $[v_{\zb,n}]_{\op{im}\p^D_\zb}$ is the image by $H(\zvf_{\zb})$ of the homology class of some $u_{\zb,n}\in\op{F}_{\zb,n}C\cap\ker\p^C_{\zb}$. Thus $$\zvf_\zb u_{\zb,n}=v_{\zb,n}+\p^D_\zb v_{\zb,n+1}\quad\text{and}\quad f\,u_{\zb,n}= v_n+\p^D\,v_{\zb,n+1}\;.$$ Since $u_{\zb,n}\in C_n\cap\ker \p^C$, it follows that $[u_{\zb,n}]_{\op{im}\p^C}\in H_n(C)$ is sent by $H(f)$ to $[v_n]_{\op{im}\p^D}\in H_n(D)$.
\endproof

To state and prove the next lemma, we need some preparation.\medskip

Consider the setting of Lemma \ref{L:limitord}. The cokernel of any ${\tt DG}\text{\it\small R}\,{\tt M}$-map $g:C'\to C''$ is computed degree-wise, so that $\op{coker}g:C''\to C''/\op{im}g$, where the {\small RHS} differential is induced by the differential of $C''$.\medskip

In our context, we thus get that, for any $\zb+1<\zl$, the cokernel of $\op{F}_{\zb,\zb+1}C$ is the ${\tt DG}\text{\it\small R}\,{\tt M}$-morphism $$h^C_{\zb+1}:\op{F}_{\zb+1}C\ni c_{\zb+1}\mapsto [c_{\zb+1}]_{\op{F}_\zb C}\in\op{F}_{\zb+1}C/\op{F}_\zb C\;.$$ The target complex is denoted by $\op{Gr}_{\zb+1}C\in {\tt DG}\text{\it\small R}\,{\tt M}$ and its differential is the differential $\p^C_{\zb+1,\sharp}$ induced by $\p^C_{\zb+1}$. It follows that $\zvf_{\zb+1}$ induces a ${\tt DG}\text{\it\small R}\,{\tt M}$-map $$\zvf_{\zb+1,\sharp}:\op{Gr}_{\zb+1}C\to \op{Gr}_{\zb+1}D\;.$$ It is possible to extend $\op{Gr}C$, defined so far on successor ordinals, to a colimit respecting functor $\op{Gr}C:\zl\to {\tt DG}\text{\it\small R}\,{\tt M}$, which we call {\em $\zl$-graduation associated to the $\zl$-filtration} $\op{F}C:\zl\to {\tt DG}\text{\it\small R}\,{\tt M}$. \medskip

Although we will not need this extension, we will use the precise definition of a colimit respecting functor $F:\tt C\to \tt D$. Recall first that, if $J:\tt I\to \tt C$ is a $\tt C$-diagram, its $\tt C$-colimit, if it exists, is an object $c\in\tt C$, together with $\tt C$-morphisms $\zh_i:J_i\to c$, such that $\zh_jJ_{ij}=\zh_i$ (i.e., together with a natural transformation $\zh$ between $J$ and the constant functor $c$). The functor $F$ is said to be colimit preserving, if the $\tt D$-colimit of $FJ$ exists and is given by the object $F(c)$, together with the $\tt D$-morphisms $F(\zh_i):F(J_i)\to F(c)$ (i.e., the natural transformation is the whiskering of $\zh$ and $F$).\medskip

Observe now that, since $\op{F}C:\zl\to{\tt DG}\text{\it\small R}\,{\tt M}$ is colimit respecting by assumption, we have, for $\za<\zl, \za\in\mathbf{O}_\ell, \za=\op{colim}_{\zb<\za}\zb$, $$\op{F}_\za C=\op{colim}_{\zb<\za}\op{F}_\zb C\,,\quad\text{together with the canonical injections}\quad \op{F}_{\zb\za}C:\op{F}_\zb C\hookrightarrow\op{F}_\za C\;.$$ The same holds for $C$ replaced by $D$. Since the colimit $\op{colim}_{\zb<\za}\zvf_\zb$ is obtained using the universality of $(\op{F}_\za C, \op{F}_{\bullet\za}C)$ with respect to the cocone $(F_\za D, \op{F}_{\bullet\za}\!D\,\zvf_\bullet)$, this colimit map is the unique ${\tt DG}\text{\it\small R}\,{\tt M}$-map $m_\za:\op{F}_\za C\to\op{F}_\za D$, such that $m_\za\op{F}_{\zb\za}C=\op{F}_{\zb\za} D\,\zvf_\zb$, i.e., such that $m_\za|_{\op{F}_\zb C}=\zvf_\zb$. Hence, for any limit ordinal $\za<\zl$, we have \be\label{PhiLim}\zvf_\za=\op{colim}_{\zb<\za}\zvf_\zb\;.\ee

\begin{lem}\label{L:filtrweq}
Let $\zl\in\mathbf{O}\setminus\{0\}$, let $C,D\in{\tt DG}\text{\it\small R}\,{\tt M}$, with $\zl$-filtrations $(\op{F}_\zb C)_{\zb\in \lambda}$ and $(\op{F}_\zb D)_{\zb\in \lambda}$, and let $f:C\rightarrow D$ be a ${\tt DG}\text{\it\small R}\,{\tt M}$-morphism, which is compatible with the filtrations and whose corresponding natural transformation is denoted by $\zvf:\op{F}C\to \op{F}D$. Assume that $\zvf_0:\op{F}_0C\rightarrow \op{F}_0D$ is a weak equivalence and that, for any $\zg+1<\zl$, the induced ${\tt DG}\text{\it\small R}\,{\tt M}$-map $\zvf_{\zg+1,\sharp}:\op{Gr}_{\zg+1}C\rightarrow \op{Gr}_{\zg+1}D$ is a weak equivalence. Then $\zvf_\zb:\op{F}_\zb C\rightarrow \op{F}_\zb D$ is a weak equivalence, for all $\zb<\zl$, and therefore $f$ is also a weak equivalence.
\end{lem}

\proof
We proceed by transfinite induction. The induction starts, since $\zvf_0$ is a weak equivalence by assumption. Let now $\zb<\zl$ and assume that $\zvf_\za$ is a weak equivalence for all $\za<\zb$.\medskip

If $\zb\in\mathbf{O}_s$, say $\zb=\zg+1$, we consider the commutative diagram
\begin{center}
\begin{tikzcd}[column sep=0.8em]
0 \ar[r]&\op{F}_\zg C \ar[hookrightarrow]{r} \ar{d}{\sim}&\op{F}_{\zg+1} C \ar[r] \ar[d]&\op{Gr}_{\zg+1}C \ar[r] \ar{d}{\sim} & 0 \\
0 \ar[r]&\op{F}_\zg D \ar[hookrightarrow]{r}        &\op{F}_{\zg+1}D \ar[r]        &\op{Gr}_{\zg+1}D \ar[r]         & 0
\end{tikzcd}\quad,
\end{center}
whose rows are exact and whose left (resp., right) vertical arrow $\zvf_\zg$ (resp., $\zvf_{\zg+1,\sharp}$) is a weak equivalence. The connecting homomorphism theorem now induces in ${\tt Mod}(R)$ the diagram
{\footnotesize
\begin{equation}\nonumber
\begin{tikzcd}[column sep=0.7em]
\ldots \ar[r]&H_{n+1}(\op{Gr}_{\gamma+1}C) \ar[r] \ar{d}{\cong} &H_n(\op{F}_\zg C) \ar[r] \ar{d}{\cong}&H_n(\op{F}_{\gamma+1}C) \ar[r] \ar[d] &H_n(\op{Gr}_{\gamma+1}C) \ar[r] \ar{d}{\cong} &H_{n-1}(\op{F}_\zg C) \ar[r] \ar{d}{\cong}& \ldots\\
\ldots \ar[r]&H_{n+1}(\op{Gr}_{\gamma+1}D) \ar[r] &H_n(\op{F}_\zg D)\ar[r] &H_n(\op{F}_{\zg+1} D) \ar[r] &H_n(\op{Gr}_{\gamma+1}D) \ar[r] &H_{n-1}(\op{F}_\zg D)\ar[r] & \ldots
\end{tikzcd}
\end{equation}}with exact rows and isomorphisms as non-central vertical arrows. Further the diagram commutes. Indeed, the connecting homomorphism $\zD^C:H_{n+1}(\op{Gr}_{\zg+1}C)\to H_{n}(\op{F}_{\zg}C)$ is defined by $$\zD^C[[c_{\zg+1}]_{\op{F}_\zg C}]_{\op{im}\p^C_{\zg+1,\sharp}}=[c_{\zg}]_{\op{im}\p^C_\zg}$$ if and only if there exists $c'_{\zg+1}\in\op{F}_{\zg+1}C$, such that $$[c_{\zg+1}]_{\op{F}_\zg C}=[c'_{\zg+1}]_{\op{F}_{\zg}C}\quad\text{and}\quad \p^C_{\zg+1}c'_{\zg+1}=c_\zg\;,$$i.e., if and only if $$\p^C_{\zg+1}c_{\zg+1}=c_\zg\;.$$ Hence, in the {\small LHS} square of the preceding diagram, the top-right composition leads to $$H(\zvf_\zg)\,\zD^C[[c_{\zg+1}]_{\op{F}_\zg C}]_{\op{im}\p^C_{\zg+1,\sharp}}=[\zvf_\zg c_{\zg}]_{\op{im}\p^D_\zg}\;.$$ On the other hand, the left-bottom composition $$\zD^D\,H(\zvf_{\zg+1,\sharp})[[c_{\zg+1}]_{\op{F}_\zg C}]_{\op{im}\p^C_{\zg+1,\sharp}}=\zD^D[\zvf_{\zg+1,\sharp}[c_{\zg+1}]_{\op{F}_\zg C}]_{\op{im}\p^D_{\zg+1,\sharp}}=\zD^D[[\zvf_{\zg+1}c_{\zg+1}]_{\op{F}_\zg D}]_{\op{im}\p^D_{\zg+1,\sharp}}$$ coincides with the value $[\zvf_\zg c_{\zg}]_{\op{im}\p^D_\zg}$, if and only if $$\p^D_{\zg+1}\zvf_{\zg+1}c_{\zg+1}=\zvf_\zg c_\zg\:,$$ what is obviously the case.\medskip

It now follows from the Five Lemma that $H_n(\zvf_{\zg+1})$ is an isomorphism, for all $n\in \mathbb{N}$, i.e., that $\zvf_{\zg+1}=\zvf_\zb$ is a weak equivalence.\medskip

If $\zb\in\mathbf{O}_\ell$, it follows from Lemma \ref{L:limitord} that $\zvf_\zb$ is a weak equivalence.
\endproof

The last lemma may be advantageously used to prove Lemma \ref{WeqCof}.\medskip

\noindent {\bf Proof of Lemma \ref{WeqCof}\label{tocum}}. Recall that our aim is to prove that, if, in $\tt DG\Dc M$, $f:A\to B$ is a weak equivalence and $M$ is a cofibrant object, then $f\0\Id_M:A\0 M\to B\0 M$ is a weak equivalence as well (we omit the subscript $\bullet$ in the tensor product). It follows from the description of the model structure on $\tt DG\Dc M$ \cite{DPP}, that cofibrant objects are exactly those differential graded $\Dc$-modules that are degree-wise $\Dc$-projective. In particular, each term of $M$ is $\Dc$-flat. On the other hand, $\Dc$ is $\Oc$-projective and thus $\Oc$-flat. Therefore, if $0\to N\to P\to Q\to 0$ is a short exact sequence ({\small SES}) in $\tt Mod(\Oc)$, the free $\Dc$-module functor $\Dc\0_\Oc \bullet$ on $\tt Mod(\Oc)$ transforms the considered {\small SES} into a new {\small SES} in $\tt Mod(\Oc)$ and even in $\tt Mod(\Dc)$. Further, left-tensoring the latter sequence over $\Dc$ by any term $M_k$, leads to a {\small SES} in Abelian groups $\tt Ab$ and even in $\tt Mod(\Oc)$. Since $M_k\0_\Dc\Dc\0_\Oc \bullet\simeq M_k\0_\Oc\bullet\,$, one deduces that any term $M_k$ of $M$ is also $\Oc$-flat.\medskip

Let now $(M_{\leq k},d_M)\in\tt DG\Dc M$ be the chain complex $(M,d_M)$ truncated at degree $k\in\N$. Then, in the diagram (\ref{e:moj}) below, the top and the bottom rows are $\zw$-filtrations of $A\otimes M$ and $B\otimes M$, respectively. In addition, the product $f\0\Id_M$ is compatible with these $\zw$-filtrations and is the colimit of the natural transformation $\zvf_\bullet:=f\0\Id_{M\le\bullet}\,$.
{\begin{center}
 \begin{equation}\label{e:moj}\begin{tikzcd}
A\otimes M_{\leq 0} \ar[r,hook] \ar{d}{f\otimes \Id_{M_{\leq 0}}} &A\otimes M_{\leq 1} \ar[hook]{r} \ar{d}{f\otimes \Id_{M_{\leq 1}}} &\ldots \ar[r,hook]&A\otimes M_{\leq n} \ar[hook]{r} \ar{d}{f\otimes \Id_{M_{\leq n}}} & \ldots \\
B\otimes M_{\leq 0} \ar[r,hook]    &B\otimes M_{\leq 1} \ar[r,hook] &\ldots \ar[r,hook]&B\otimes M_{\leq n} \ar[hook]{r}   & \ldots
\end{tikzcd}\end{equation}\end{center}
}
\noindent The morphism $f\0\Id_M$ is a weak equivalence, if the assumptions of Lemma \ref{L:filtrweq} are satisfied. For any $1\le k+1<\zw$, the induced map $$\zvf_{k+1,\sharp}:\op{Gr}_{k+1}(A\0 M)\to \op{Gr}_{k+1}(B\0 M)\quad\text{is}\quad f\otimes \Id_{M_{k+1}}:A\otimes M_{k+1}\to B\otimes M_{k+1}\;.$$ Moreover, the map $$\zvf_0:A\0 M_{\le 0}\to B\0 M_{\le 0}\quad\text{is}\quad f\0\Id_{M_0}:A\0 M_0\to B\0 M_0\;.$$ To show that $f\0\Id_{M_k}$, $k\in\N$, is a weak equivalence, we prove the equivalent statement that its mapping cone $\op{Mc}(f\0\Id_{M_k})$ is acyclic. Notice that \be\label{MapCone}\op{Mc}(f\otimes \Id_{M_k})\simeq (\op{Mc}(f))[-k]\otimes M_k\;,\ee as $\tt DG\Dc M$, since $M_k$ has zero differential. To find that the {\small RHS} is acyclic, it suffices to consider the involved complexes in $\tt DG\Oc M$, to recall that $M_k$ is $\Oc$-flat and that, since $f$ is weak equivalence, $H((\op{Mc}(f))[-k])=0.$ The looked for acyclicity then follows from K\"unneth's formula.\endproof

\subsection{HAC condition 1: properness}\label{SS:properness}

The first of the {\small HAC} assumptions mentioned at the beginning of this section is the condition HAC1 \cite[Assumption 1.1.0.1]{TV08}.\medskip

\noindent{\bf HAC1}. The underlying model category $\tt C$ is proper, pointed, and, for any $c',c''\in\tt C$, the morphisms

\be\label{HAC1} Qc'\coprod Qc''\to c'\coprod c''\to Rc'\prod Rc''\;,\ee \vspace{0.1mm}

\noindent where $Q$ (resp., $R$) denotes the cofibrant (resp., the fibrant) replacement functor, are weak equivalences. Moreover, the homotopy category $\tt Ho(C)$ of $\tt C$ is additive.\medskip

Assumption {\bf HAC1} implies that $\op{Hom}_{\tt C}(c',c'')$ is an Abelian group. This fact and the homotopy part of the assumption allow to understand that the idea is to require that $\tt C$ be a kind of `weak' additive or Abelian category. \medskip

Let us briefly explain the different parts of condition {\small HAC1}. Properness is defined as follows \cite[Def. 13.1.1]{Hir}:

\begin{Defi}
A model category $\tt C$ is said to be:
\begin{enumerate}
\item \emph{left proper}, if every pushout of a weak equivalence along a cofibration is a weak equivalence,
\item \emph{right proper}, if every pullback of a weak equivalence along a fibration is a weak equivalence,
\item \emph{proper}, if it is both, left proper and right proper.
\end{enumerate}
\end{Defi}

Pointed means that the category has a zero object $0$. The first morphism in (\ref{HAC1}) comes from the composition of the weak equivalences $Qc'\to c'$ and $Qc''\to c''$ with the canonical maps $c'\to c'\coprod c''$ and $c''\to c'\coprod c''$, respectively. As for the second, note that, in addition to the weak equivalence $c'\to Rc'$, we have also the map $c'\to 0\to Rc''$, hence, finally, a map $c'\to Rc'\prod Rc''$. Similarly, there is a map $c''\to Rc'\prod Rc''$, so, due to universality, there exists a map $c'\coprod c''\to Rc'\prod Rc''$.\medskip

We now check HAC1 for the basic model category $\tt DG\Dc M$ of the present paper.\medskip

Properness of $\tt DG\Dc M$ will be dealt with in Theorem \ref{T:leftproper} below. Since $\tt DG\Dc M$ is Abelian, hence, additive, it has a zero object -- in the present situation $(\{0\},0)$ -- . As for the arrows in (\ref{HAC1}), note that the coproduct and the product of chain complexes of modules are computed degree-wise and that finite coproducts and products of modules coincide and are just direct sums. Since the direct sum of two quasi-isomorphisms is a quasi-isomorphism, the first canonical arrow is a quasi-isomorphism. Recall moreover, that, in $\tt DG\Dc M$, every object is fibrant, so that we can choose the identity as fibrant replacement functor $R$. Therefore, the second canonical arrow is the identity map and is thus a quasi-isomorphism. Further, the homotopy category $\tt Ho(DG\Dc M)$ is equivalent to the derived category $\mathbb{D}^+({\tt Mod}(\Dc))$ (see Appendix \ref{HDCatDFun}), which is triangulated and thus additive. This additive structure on the derived category can be transferred to the homotopy category (so that the functor that implements the equivalence becomes additive).\medskip

We are now left with verifying properness of $\tt DG\Dc M$. By \cite[Corollary 13.1.3]{Hir}, a model category all of whose objects are fibrant is right proper: it is easily seen that this is the case, not only for $\tt DG\Dc M$, but also for $\tt DG\Dc A$ and ${\tt Mod}(\Ac)$. We will check that these three categories are left proper as well, so that

\begin{thm}\label{T:leftproper}
The model categories $\tt DG\Dc M,$ $\tt DG\Dc A$ and ${\tt Mod}(\Ac)$ are proper.
\end{thm}

\proof Since $\Oc\in\tt DG\Dc A$ and ${\tt Mod}(\Oc)=\tt DG\Dc M$, it suffices to prove the statement for $\tt DG\Dc A$ and ${\tt Mod}(\Ac)$, where $\Ac$ is any object of $\tt DG\Dc A$. Below, the letter $\tt C$ denotes systematically any of the latter categories, $\tt DG\Dc A$ or ${\tt Mod}(\Ac)$. \medskip

We already mentioned that $\tt C$ is cofibrantly generated, so that any cofibration is a retract of a map in $I$-cell, where $I$ is the set of generating cofibrations \cite[Proposition 2.1.18(b)]{Ho99}. More precisely, the small object argument allows to factor any $\tt C$-morphism $s:X\to W$ as $s=p\,i$, with $i\in I$-cell\,$\subset\op{Cof}$ and $p\in I$-inj\,$=\op{TrivFib}$. If $s\in\op{Cof}$, it has the {\small LLP} with respect to $p$. Hence, the commutative diagram

\begin{equation}\label{e:sretri}
\begin{tikzcd}
X \ar[equal]{r} \ar[rightarrowtail]{d}{s}&X \ar[equal]{r} \ar[rightarrowtail]{d}{i}&X\ar[rightarrowtail]{d}{s}\\
W\ar{r}[swap]{l}                   &U \ar[twoheadrightarrow]{r}{\sim}[swap]{p}        &W
\end{tikzcd}\quad,
\end{equation}
where $l$ is the lift and where $p\,l=1$.\medskip

We must show that the pushout $g:W\to V$ of a weak equivalence $f:X\to Y$ along the cofibration $s:X\to W$ is a weak equivalence.\medskip

Note first that, if $h:U \to Z$ is the pushout of $f$ along $i$, then $g$ is a retract of $h$. To see this, consider the following commutative diagram, where the dashed arrows come from the universality of a pushout:
\begin{equation}
\begin{tikzcd}
X \arrow{r}{f} \arrow{d}{s} \ar[bend right]{dd}[swap]{i}\ar[bend right=43]{ddd}[swap]{s} & Y \arrow{d} \arrow[bend left]{ddr}\arrow[bend left]{dddrr}\\
W\arrow{d}{l} \arrow{r}{g} & V \arrow[dashed]{dr}\\
U\ar{d}{p}\ar{rr}{h}& & Z\arrow[dashed]{dr}\\
W\ar{rrr}{g}&&& V
\end{tikzcd}.
\ee
\noindent Due to the uniqueness property encrypted in any universal construction, the composite of the two dashed arrows is the identity of $V$. Hence, $g$ is indeed a retract of $h$. As weak equivalences are closed under retracts, it thus suffices to show that the pushout $h$ is a weak equivalence.\medskip

We will actually prove that the pushout of a weak equivalence along any map in $I$-cell, i.e., along any transfinite composition of pushouts of maps in $I$, is a weak equivalence.\medskip

{\bf Step 1}. In this step, we explain -- separately in each of the two categories $\tt DG\Dc A$ and ${\tt Mod}(\Ac)$ -- why the pushout of a weak equivalence along a pushout of a map in $I$, i.e., along a pushout of a generating cofibration, is again a weak equivalence.\medskip

In $\tt DG\Dc A$, see \cite[Example 1]{DPP2}, any pushout of a generating cofibration is a (minimal) relative Sullivan $\Dc$-algebra $(T,d_T)\hookrightarrow(T\otimes\,\Sc S^n, d)$, where $d$ is defined as described in \cite[Lemma 1]{DPP2}. Similarly, in ${\tt Mod}(\Ac)$, see Proof of Theorem \ref{L:cofobjsinAMod}, any pushout of a generating cofibration is a relative Sullivan $\Ac$-module $(T,d_T)\hookrightarrow (T\oplus\Ac\0 S^n,d)$, where $d$ is defined as detailed in Lemma \ref{LemAmod}.\medskip

We first examine the $\tt DG\Dc A$-case. Here the pushout

\begin{equation}\nonumber
\begin{tikzcd}
X \ar{r}{i_X} \ar{d}{f}&X\0\Sc S^n \ar{d}{f\otimes \Id_{\Sc S^n}}\\
Y \ar{r}{i_Y}        &Y\0\Sc S^n
\end{tikzcd}\quad .
\end{equation}
of a weak equivalence $f:(X,\p)\xrightarrow{\sim} (Y,\zd)$ along a relative Sullivan $\Dc$-algebra $(X,\p)\hookrightarrow (X\otimes\, \Sc S^n,\p^{(1)})$ is made of
\begin{itemize}\item[-] the relative Sullivan $\Dc$-algebra $(Y,\zd)\hookrightarrow(Y\otimes\, \Sc S^n,\zd^{(1)})$, whose differential $\zd^{(1)}$ is given by
\be\label{e:d2c}
\zd^{(1)}(1_n):=f(\p^{(1)}(1_n))\in Y_{n-1}\cap \zd^{-1}\{0\}\;,
\ee
where $1_n$ is the basis of $S^n$, and
\item[-] the $\tt DG\Dc A$-morphism $f\otimes \Id_{\Sc S^n}:(X\otimes\, \Sc S^n,\p^{(1)})\to (Y\otimes\, \Sc S^n,\zd^{(1)})\,$.\vspace{3mm}
\end{itemize}
Reference \cite[Lemma 1(i)]{DPP2} allows to see that (\ref{e:d2c}) defines a relative Sullivan $\Dc$-algebra. Since the {\small RHS} arrow in the above diagram is necessarily an extension $\ze$ of $i_Y\,f:X\to Y\otimes\, \Sc S^n$, we apply \cite[Lemma 1(ii)]{DPP2} to the morphism $i_Y\,f$. Therefore we note that the relative Sullivan $\Dc$-algebra $X\otimes\, \Sc S^n$ is actually constructed according to \cite[Lemma 1(i)]{DPP2}. Indeed, in view of the first paragraph below \cite[Lemma 1]{DPP2}, since the differential $\p^{(1)}$ restricts to $\p$ on $X$ and satisfies $\p^{(1)}(1_n)\in X_{n-1}\cap \p^{-1}\{0\}$, it is necessarily given by Equation (9) in \cite[Lemma 1]{DPP2}. Hence, the reference \cite[Lemma 1(ii)]{DPP2} can be used and the extension $\ze$ is fully defined by $$\ze(1_n):=1_Y\0 1_n\in (Y\0\,\Sc S^n)_n\cap (\zd^{(1)})^{-1}\{f\,\p^{(1)}(1_n)\}\;.$$ The extending $\tt DG\Dc A$-morphism $\ze$ is then given, for any $x\in X$ and any $\zs\in\Sc S^n$, by $\ze(x\0 \zs)=f(x)\0 \zs$, so that $\ze=f\0\Id_{\Sc S^n}$. As concerns universality, let $h:Y\to E$ and $k:X\0\,\Sc S^n\to E$ be $\tt DG\Dc A$-maps, such that $k\,i_X=h\,f$, and define the `universality map' $\zm:Y\0\,\Sc S^n\to E$ as extension of $h$ (using the same method as for $\ze$), by setting $$\zm(1_n):=k(1_X\0 1_n)\in E_n\cap d_E^{-1}\{h\,\zd^{(1)}(1_n)\}\;.$$ To check the latter condition on $d_E$, it suffices to note that, on $1_n$, we have $$d_E\,k=k\,\p^{(1)}=k\,i_X\,\p^{(1)}=h\,f\,\p^{(1)}=h\,\zd^{(1)}\;,$$ due to (\ref{e:d2c}). Further, the condition $\zm\,i_Y=h$ is satisfied by construction, and to see that $\zm\,\ze=k$, we observe that $$\zm(\ze(x\0\zs))=h(f(x))\star_E\zm(\zs)\quad\text{and}\quad k(x\0 \zs)=k(x\0 1_\Oc)\star_E k(1_X\0\zs)=h(f(x))\star_E \zm(1_Y\0\zs)\;,$$ where $k(1_X\0\zs)$ coincides with $\zm(1_Y\0\zs)$, since both maps are $\tt DG\Dc A$-maps and $k(1_X\0 1_n)$ coincides with $\zm(1_Y\0 1_n)$, by definition. Eventually, uniqueness of the `universality map' is easily checked. \medskip

As $f\otimes \Id_{\Sc S^n}$ is a weak equivalence in $\tt DG\Dc A$ if it is a weak equivalence in $\tt DG\Dc M$, we continue working in the latter category. Notice first that, if $Z$ denotes $X$ or $Y$ and if $d^{(1)}$ denotes $\p^{(1)}$ or $\zd^{(1)}$, the differential $d^{(1)}$ stabilizes the graded $\Dc$-submodule $Z_k=Z\0\,\Sc^{\le k} S^n$ ($k\in\N=\zw$) of $Z\0\,\Sc S^n$ \cite[Lemma 1(i)]{DPP2}. Hence, the restriction $$\zvf_k:=f\0\Id_{\Sc^{\le k} S^n}:X_k\to Y_k$$ of the $\tt DG\Dc M$-map $f\otimes \Id_{\Sc S^n}$ is itself a $\tt DG\Dc M$-map. Moreover, the injections $Z_{k\ell}:Z_k\to Z_\ell$ ($k\le\ell$) are canonical $\tt DG\Dc M$-maps, so that we have a functor $Z_\ast:\zw\to\tt DG\Dc M$, with obvious colimit $Z\0\,\Sc S^n$. Since $\zvf_\ast:X_\ast\to Y_\ast$ is a natural transformation between the $\zw$-filtrations of $X\0\,\Sc S^n$ and $Y\0\,\Sc S^n$, with colimit $f\0\Id_{\Sc S^n}$, it remains to prove that the diagram
\begin{equation}\label{e:filtrarrcat}
\begin{tikzcd}
X \ar[hook,r] \ar{d}{f}[swap]{\sim}&X_1\ar{d}{\zvf_1}\ar[hook,r]&\ldots\ar[hook,r]&X_k\ar[hook,r]\ar{d}{\zvf_k}&\ldots\\
Y \ar[hook,r]&Y_1\ar[hook,r]&\ldots\ar[hook,r]&Y_k\ar[hook,r]&\ldots
\end{tikzcd}
\end{equation}
satisfies the hypotheses of Lemma \ref{L:filtrweq}, i.e., that for any $k$, the $\tt DG\Dc M$-map $\zvf_{k,\sharp}$ induced by $\zvf_k$ between the $k$-terms of the $\zw$-graduations associated to the two $\zw$-filtrations, is a weak equivalence. This will be done independently in Lemma \ref{L:filtrtech}, what then concludes our argument in the $\tt DG\Dc A$-case.\medskip

In the category ${\tt Mod}(\Ac)$, the pushout
\begin{equation}\nonumber
\begin{tikzcd}
X \ar{r}{i_X} \ar{d}{f}[swap]{\sim}&X\oplus \Ac\otimes S^n \ar{d}{f\otimes \Id_{\Ac\otimes S^n}}\\
Y \ar{r}{i_Y}        &Y\oplus \Ac\otimes S^n
\end{tikzcd}
\end{equation}
of a weak equivalence $f:(X,\p)\xrightarrow{\sim} (Y,\zd)$ along a relative Sullivan $\Ac$-module $(X,\p)\hookrightarrow(X\oplus \Ac\otimes S^n,\p^{(1)})$ is made of
\begin{itemize}\item[-] the relative Sullivan $\Ac$-module $(Y,\zd)\hookrightarrow(Y\oplus \Ac\otimes S^n,\zd^{(1)})$, whose differential $\zd^{(1)}$ is determined by
\be\nonumber \label{e:d3c}
\zd^{(1)}(1_n)=f(\p^{(1)}(1_n))\in Y_{n-1}\cap \zd^{-1}\{0\}\;,
\ee
and
\item[-] the ${\tt Mod}(\Ac)$-morphism $f\oplus \Id_{\Ac\otimes S^n}:(X\oplus \Ac\otimes S^n,\p^{(1)})\to (Y\oplus \Ac\otimes S^n,\zd^{(1)})\,$.\vspace{3mm}
\end{itemize}
This statement can be understood similarly to (but more easily than) its counterpart in the $\tt DG\Dc A$-case (replace Sullivan algebras and \cite[Lemma 1]{DPP2} by Sullivan modules and Lemma \ref{LemAmod}).\medskip

As above, since $f\oplus\Id_{\Ac\0 S^n}$ is a weak equivalence in ${\tt Mod}(\Ac)$ if it is a weak equivalence in $\tt DG\Dc M$, we continue working in the latter category. Notice also that the rows of the preceding diagram are $2$-filtrations of $X\oplus \Ac\otimes S^n$ and $Y\oplus \Ac\otimes S^n$, respectively, that $f\oplus \Id_{\Ac\otimes S^n}$ is compatible with these filtrations and that the corresponding natural transformation $\zvf_\ast$ is defined by the vertical arrows of the diagram. Since $f$ is a weak equivalence, and the induced $\tt DG\Dc M$-map $\zvf_{1,\sharp}$ between the 1-terms of the associated $2$-graduations is, as $\tt DG\Dc M$-map, the identity $\Id_{\Ac\otimes S^n}$, the map $f\oplus \Id_{\Ac\0 S^n}$ is a weak equivalence, thanks to Lemma \ref{L:filtrweq}.\medskip

From here to the end of this proof, we consider the two cases, $\tt DG\Dc A$ and ${\tt Mod}(\Ac)$, again simultaneously and denote both categories by $\tt C$. We have just shown that the pushout of any weak equivalence along the pushout of any generating cofibration is itself a weak equivalence. In the sequel, we denote the pushout of a generating cofibration, or, better, the corresponding relative Sullivan $\Dc$-algebra $(X,\p)\hookrightarrow(X\0\,\Sc S^n,\p^{(1)})$ or relative Sullivan $\Ac$-module $(X,\p)\hookrightarrow(X\oplus\Ac\0 S^n,\p^{(1)})$, by \be\label{Push}X^{(0,1)}:(X^{(0)},\p^{(0)})\hookrightarrow (X^{(1)},\p^{(1)})\quad\text{or even}\quad X^{(\zb,\zb+1)}:(X^{(\zb)},\p^{(\zb)})\hookrightarrow (X^{(\zb+1)},\p^{(\zb+1)})\;,\ee where $\zb$ is an ordinal.\medskip

{\bf Step 2}. In this second step, we finally show that the pushout of a weak equivalence $\zf^{(0)}:X^{(0)}\to Y^{(0)}$ in $\tt C$ along a $\tt C$-map in $I$-cell, i.e., along a transfinite composition of pushouts of maps in $I$, is again a weak equivalence. More precisely, such a composition is the colimit $$\op{colim}_{\zb<\zl}X^{(0,\zb)}:X^{(0)}\to \op{colim}_{\zb<\zl}X^{(\zb)}$$ of a colimit respecting functor $X^{(\ast)}:\zl\to\tt C$ ($\zl\in\mathbf{O}$), such that any map $X^{(\zb,\zb+1)}:X^{(\zb)}\hookrightarrow X^{(\zb+1)}$ ($\zb+1<\zl$) is the pushout of a map in $I$, i.e., is a Sullivan `object' of the type (\ref{Push}).\medskip

It might be helpful to notice that the considered transfinite composition is given, in the ${\tt Mod}(\Ac)$-case, by (\ref{TransCompPushGen}), and in the $\tt DG\Dc A$-case, by $$X^{(0)}\to X^{(0)}\0\,\Sc(\bigoplus_{\tiny\zb<\zl,\zb\in\mathbf{O}_s}S^{n{(\zb)}})\;,$$ see \cite[Proof of Theorem 4(i)]{DPP2}.

{\bf Step 2.a}. The idea is to first construct the following commutative diagram:
\begin{equation}
\begin{tikzcd}
X^{(0)} \ar{r}{X^{(0,1)}} \ar{d}{\zf^{(0)}}[swap]{\sim} & X^{(1)}\ar{d}{\zf^{(1)}}\ar{r}{X^{(1,2)}}&\ldots\ar[r]& X^{(\zb)}\ar{d}{\zf^{(\zb)}}\ar{r}{X^{(\zb,\zb+1)}}& X^{(\zb+1)}\ar{d}{\zf^{(\zb+1)}}\ar[r]&\ldots\; X^{(\zg)}\ar{d}{\zf^{(\zg)}}\ldots\\
Y^{(0)} \ar{r}{Y^{(0,1)}}& Y^{(1)}\ar{r}{Y^{(1,2)}} & \ldots \ar[r]&Y^{(\zb)}\ar{r}{Y^{(\zb,\zb+1)}}& Y^{(\zb+1)}\ar[r]&\ldots\; Y^{(\zg)}\ldots
\end{tikzcd}\label{fig PushoutFinal}\quad
\end{equation}
\begin{center}\text{Figure: Pushout along an $I$-cell}\end{center}\medskip
\noindent More precisely, for $\zg<\zl\,$, we will build, by transfinite induction,

\begin{itemize}
\item[-] a colimit respecting functor $Y^{(\ast)}_\zg:\zg+1\to\tt C$ with injective elementary maps $Y^{(\zb,\zb+1)}$ ($\zb+1<\zg+1$) and
\item[-] a natural transformation $\zf^{(\ast)}_\zg$ between $X^{(\ast)}_\zg$ and $Y^{(\ast)}_\zg$,
\end{itemize}
such that $\zf^{(\zg)}$ is a weak equivalence and $$Y^{(0)}\stackrel{Y^{(0,\zg)}}{\longrightarrow} Y^{(\zg)}\stackrel{\zf^{(\zg)}}{\longleftarrow} X^{(\zg)}$$ is the pushout of $$Y^{(0)}\stackrel{\zf^{(0)}}{\longleftarrow} X^{(0)}\stackrel{X^{(0,\zg)}}{\longrightarrow} X^{(\zg)}\;.$$ This construction is based on the assumption that $Y^{(\ast)}_\za$ and $\zf^{(\ast)}_\za$ have been constructed with the mentioned properties, for any $\za<\zg$.\medskip

The induction starts since the requirements concerning $Y^{(\ast)}_0$ and $\zf^{(\ast)}_0$ are obviously fulfilled and $\zf^{(0)}$ is a weak equivalence.\medskip

We first examine the case $\zg\in\mathbf{O}_s$. We can begin with the functor $Y^{(\ast)}_{\zg-1}$, the natural transformation $\zf^{(\ast)}_{\zg-1}$ and the square of the pushout $\zf^{(\zg-1)}$. Then we build the pushout $$Y^{(\zg-1)}\stackrel{Y^{(\zg-1,\zg)}}{\longrightarrow} Y^{(\zg)}\stackrel{\zf^{(\zg)}}{\longleftarrow} X^{(\zg)}\;$$ of $$Y^{(\zg-1)}\stackrel{\zf^{(\zg-1)}}{\longleftarrow}X^{(\zg-1)}\stackrel{X^{(\zg-1,\zg)}}{\longrightarrow} X^{(\zg)}$$ as in Step 1. It follows from the induction assumption and the description in Step 1 that there is a canonical functor $Y^{(\ast)}_\zg$ that has the required properties, as well as a canonical natural transformation $\zf^{(\ast)}_\zg$. Moreover, the map $\zf^{(\zg)}$ is a weak equivalence, and, since the outer square of two pushout squares is a pushout square, the map $\zf^{(\zg)}$ has the requested pushout property.\medskip

If $\zg=\op{colim}_{\zb<\zg}\zb\in\mathbf{O}_\ell$, note that, since colimits commute, the searched pushout $$\op{colim}(Y^{(0)}\longleftarrow X^{(0)}\longrightarrow X^{(\zg)})$$ of $\zf^{(0)}$ along $X^{(0,\zg)}:X^{(0)}\to X^{(\zg)}$, i.e., along $\op{colim}_{\zb<\zg}X^{(0,\zb)}:X^{(0)}\to\op{colim}_{\zb<\zg}X^{(\zb)}$, is equal to $$\op{colim}_{\zb<\zg}\op{colim}(Y^{(0)}\stackrel{\zf^{(0)}}{\longleftarrow} X^{(0)}\stackrel{X^{(0,\zb)}}{\longrightarrow} X^{(\zb)})=$$ \be\label{ColimPush}\op{colim}_{\zb<\zg}(Y^{(0)}\stackrel{Y^{(0,\zb)}}{\longrightarrow} Y^{(\zb)}\stackrel{\zf^{(\zb)}}{\longleftarrow} X^{(\zb)})\;.\ee Of course, the functors $Y^{(\ast)}_\za$ (resp., the natural transformations $\zf^{(\ast)}_\za$), $\za<\zg$, define a functor $Y^{(\ast)}:\zg\to\tt C$ with the same properties (resp., a natural transformation $\zf^{(\ast)}:X^{(\ast)}\to Y^{(\ast)}$). The functor $Y^{(\ast)}$ can be extended by $$Y^{(\zg)}:=\op{colim}_{\zb<\zg}Y^{(\zb)}\quad\text{and}\quad Y^{(\za,\zg)}:=\op{colim}_{\za\le\zb<\zg}Y^{(\za,\zb)}\;,$$ as colimit respecting functor $Y^{(\ast)}_\zg$ with injective elementary maps. Similarly, the natural transformation $\zf^{(\ast)}$ can be extended, via the application of the colimit functor, $$\zf^{(\zg)}:=\op{colim}_{\zb<\zg}\zf^{(\zb)}:X^{(\zg)}\to Y^{(\zg)}\;,$$ to a natural transformation $\zf^{(\ast)}_\zg$. Hence, the colimit (\ref{ColimPush}) is given by $$Y^{(0)}\stackrel{Y^{(0,\zg)}}{\longrightarrow} Y^{(\zg)}\stackrel{\zf^{(\zg)}}{\longleftarrow} X^{(\zg)}\;.$$ It now suffices to check that $\zf^{(\zg)}$ is a weak equivalence in $\tt DG\Dc M$. Since, as easily seen, $X^{(\ast)}$ (resp., $Y^{(\ast)}$) is a $\zg$-filtration of $X^{(\zg)}$ (resp., $Y^{(\zg)}$), since $\zf^{(\zg)}$ is filtration-compatible with associated natural transformation $\zf^{(\ast)}$, and since $\zf^{(\za)}$, $\za<\zg$, is a weak equivalence, it follows from Lemma \ref{L:limitord} that $\zf^{(\zg)}$ is a weak equivalence as well.\medskip

{\bf Step 2.b}. The pushout of $\zf^{(0)}$ along $$\op{colim}_{\zg<\zl}X^{(0,\zg)}:X^{(0)}\to \op{colim}_{\zg<\zl}X^{(\zg)}$$ is given by Equation (\ref{ColimPush}) with $\zg$ replaced by $\zl$ (and $\zb$ by $\zg$). It is straightforwardly checked that $Y^{(\ast)}$ (resp., $\zf^{(\ast)}$) is a functor defined on $\zl$ (resp., a natural transformation between such functors). Hence, the colimit (\ref{ColimPush}) is given by $$Y^{(0)}\stackrel{\op{colim}_{\zg<\zl}Y^{(0,\zg)}}{\longrightarrow}\op{colim}_{\zg<\zl}Y^{(\zg)}\stackrel{\op{colim}_{\zg<\zl}\zf^{(\zg)}}{\longrightarrow}\op{colim}_{\zg<\zl}X^{(\zg)}\;.$$ Since $X^{(\ast)},Y^{(\ast)}$ are $\zl$-filtrations of $\op{colim}_{\zg<\zl}X^{(\zg)}$ and $\op{colim}_{\zg<\zl}Y^{(\zg)}$, respectively, and the considered pushout $\op{colim}_{\zg<\zl}\zf^{(\zg)}$ of $\zf^{(0)}$ is filtration-compatible, it follows from Lemma \ref{L:limitord} that this pushout is a weak equivalence.\medskip

To complete the proof of Theorem \ref{T:leftproper}, it remains to show that the following lemma, which we state separately for future reference, holds.

\begin{lem}\label{L:filtrtech}
Diagram \eqref{e:filtrarrcat} satisfies the assumptions of Lemma \ref{L:filtrweq}.
\end{lem}
\proof For $X\0\,\Sc S^n$, the term $X^{k+1}:=X_{k+1}/X_k$ ($k+1<\zw$) of the $\zw$-graduation, associated to the $\zw$-filtration with filters $X_\ell=X\0\,\Sc^{\le \ell} S^n$ ($\ell<\zw$), is isomorphic as $\tt DG\Dc M$-object to $$X^{k+1}\simeq X\0\,\Sc^{k+1}S^n\;,$$ where the {\small RHS} is endowed with the usual tensor product differential. A similar statement holds for $Y\0\,\Sc S^n$. Moreover, when read through the preceding isomorphisms, say $\Ic_X$ and $\Ic_Y$, the $\tt DG\Dc M$-map $\zvf_{k+1,\sharp}:X^{k+1}\to Y^{k+1}$ induced by $\zvf_{k+1}=f\0\Id_{\Sc^{\le k+1}S^n}$, is the $\tt DG\Dc M$-map $$\Ic_Y\,\zvf_{k+1,\sharp}\,\Ic_X^{-1}=f\otimes \Id_{\Sc^{k+1} S^n}\;.$$ Since $\zvf_0=f$ is a weak equivalence, it remains to show that $f\otimes \Id_{\Sc^{k+1} S^n}$ is a weak equivalence, for all $k+1<\zw$, or, still, that $f\otimes \Id_{\Sc^{k} S^n}$ is a weak equivalence, for all $1\le k<\zw$.\medskip

Just as in Equation (\ref{MapCone}), we have here $$\op{Mc}(f\otimes \Id_{\Sc^kS^n})\simeq \op{Mc}(f)[-kn]\otimes \Sc^kS^n\;,$$
as $\tt DG\Dc M$-object, since $\Sc^kS^n$ has zero differential. We now proceed as in \cite[Sections 7.5 and 8.7]{DPP}: The symmetrisation map $\zs$ induces a short exact sequence
$$
0\rightarrow\ker^k\zs\xrightarrow{i} \bigotimes^k{S^n}\xrightarrow{\zs} {\Sc^{k}S^n}\rightarrow 0
$$
in the Abelian category $\tt DG\Oc M$. Since this sequence canonically splits, we get the $\tt G\Oc M$-isomorphism $$H(\op{Mc}(f)[-kn]\otimes\bigotimes^k{S^n})\simeq H(\op{Mc}(f)[-kn]\otimes\ker^k\zs)\;\oplus\;H(\op{Mc}(f)[-kn]\otimes{\Sc^{k}S^n})\;.$$ To prove the weak equivalence condition, it suffices to show that the {\small LHS}-homology vanishes. Assume that the claim is proven for $0\le k-1<\zw$. The induction starts since $f$ is a weak equivalence, i.e., a quasi-isomorphism. The fact that $$H(\op{Mc}(f)[-kn]\otimes\bigotimes^{k-1}{S^n}\otimes S^n)=0$$ is then a consequence of the K\"unneth formula for complexes and the previously mentioned fact that $\Dc$ is $\Oc$-flat. \endproof

\subsection{HAC condition 2: combinatoriality}\label{SS:combinatoriality}

All the requirements of the second axiom HAC2 \cite[Assumption 1.1.0.2]{TV08} of a Homotopical Algebraic Context have been established above, except the combinatoriality condition for the model structure of ${\tt Mod}(\Ac)$. For future reference, we will also prove the combinatoriality of the model structures of $\tt DG\Dc M$ and of $\tt DG\Dc A$. A reader, who is interested in set-theoretical size issues and {\bf universes}, finds all relevant information in Appendix \ref{AppendixUniverses}. \medskip

Roughly, a combinatorial model category is a well manageable type of model category, in the sense that it is generated from small ingredients: it is a category
\begin{itemize}
\item[-] in which any object is the colimit of small objects from a given set of generators, and

\item[-] which carries a cofibrantly generated model structure, i.e., a model structure whose cofibrations (resp., trivial cofibrations) are generated by sets $I$ (resp., $J$) of generating morphisms whose sources are small.
\end{itemize}\medskip

More precisely,

\begin{Defi}\label{D:cmc}
A {\bf combinatorial model category}, is a locally presentable category that is endowed with a cofibrantly generated model structure.
\end{Defi}

For {\bf locally presentable categories}, i.e., categories that are locally $\zk$-presentable for some regular cardinal $\zk$, we refer to Appendix \ref{LPC}. Aspects of the foundational background of and further details on combinatorial model categories are available in \cite{Du01, AR}. Eventually, a category that satisfies all the conditions of a locally presentable category, except that it is not necessarily cocomplete, is referred to as an {\bf accessible category}.\medskip

Our categories of interest, $\tt DG\Dc M$, $\tt DG\Dc A$, and ${\tt Mod}(\Ac)$, are cofibrantly generated model categories. In particular, they are (complete and) cocomplete, so that, to prove their combinatoriality, it suffices to prove their accessibility. As for $\tt DG\Dc M$, we mentioned in the proof of Proposition \ref{L:Talg} that it is locally presentable, hence, accessible. Regarding the accessibility of ${\tt DG\Dc A}\simeq {\tt DG\Dc M}^T$ (see Proposition \ref{L:Talg}) and ${\tt Mod}(\Ac)\simeq {\tt DG\Dc M}^U$ (see Proposition \ref{L:Ualg}), we recall \cite[2.78]{AR} that a category of algebras over a monad is accessible, if the monad is. Furthermore \cite[2.16]{AR}, a monad $(V,\mu,\eta)$ over a category ${\tt C}$ is accessible, if its endofunctor $V:\tt C\to\tt C$ is accessible. Finally, a functor $G:{\tt C}'\to {\tt C}''$ is called accessible, if it is accessible for some regular cardinal $\zk$, i.e., if ${\tt C}'$ and ${\tt C}''$ are $\zk$-accessible categories, and if $G$ preserves $\zk$-directed colimits. Summarizing, to prove that $\tt DG\Dc A$ and ${\tt Mod}(\Ac)$ are accessible, we only need to show that both, $T=F\Sc$ and $U=\zF\,\zS$, preserve $\zk$-directed colimits. In fact, since the left adjoints $\Sc$ and $\zS$ respect all colimits, it suffices to reassess the right adjoints $F:{\tt DG\Dc A}\to {\tt DG\Dc M}$ and $\zF:{\tt Mod}(\Ac)\to{\tt DG\Dc M}$. However, in \cite{DPP}, we showed that $F$ commutes with directed colimits (and $\zk$-directed ones), and the proof for $\zF$ is similar. Hence,

\begin{prop}\label{ProperCombinatorial}
The (proper) model categories $\tt DG\Dc M$, $\tt DG\Dc A$, and ${\tt Mod}(\Ac)$ are combinatorial model categories.
\end{prop}

\subsection{HAC condition 3: cofibrancy and equivalence-invariance}\label{SS:cofflat}

As above, we choose $\Ac\in\tt DG\Dc A$. The condition HAC3 \cite[Assumption 1.1.0.3]{TV08} asks that, for any cofibrant $M\in{\tt Mod}(\Ac)$, the functor $$-\0_\Ac M:{\tt Mod}(\Ac)\to{\tt Mod}(\Ac)$$ respect weak equivalences. The requirement is not really surprising. Indeed, to avoid `equivalence-invariance breaking' in the model category ${\tt Mod}(\Ac)$ via tensoring by $M$, this operation should preserve weak equivalences -- at least for `good' objects $M$, i.e., for cofibrant ones. This is similar to tensoring, in the category ${\tt Mod}(R)$ of modules over a ring $R$, by an $R$-module $M$, what is an operation that respects injections for `good' objects $M$, i.e., for flat $R$-modules.

\begin{prop}\label{P:flatness}
Let $\Ac\in\tt DG\Dc A$. For every cofibrant $M\in \tt Mod(\Ac)$, the functor
$$
-\otimes_\Ac M:{\tt Mod(\Ac)}\rightarrow{\tt Mod(\Ac)}
$$
preserves weak equivalences.
\end{prop}
\proof
We have to prove that, if $f:P\to Q$ is a weak equivalence in $\tt Mod(\Ac)$, then $f\0_\Ac\Id_M$ is a weak equivalence as well.\medskip

By Lemma \ref{L:cofobjsinAMod}, the module $M$ is a retract of a Sullivan $\Ac$-module, i.e., there exist a Sullivan $\Ac$-module $\Ac\otimes V$ and $\tt Mod(\Ac)$-morphisms $i:M\to \Ac\otimes V$ and $j:\Ac\otimes V\to M$ such that
$j\circ i=\Id_M$. Since the diagram

\begin{equation}\label{e:sretri}
\begin{tikzcd}[column sep=large]
P\otimes_\Ac M \ar{r}{\Id_P\otimes_\Ac i} \ar{d}{f\otimes_\Ac \Id_M}&P\otimes_\Ac(\Ac\otimes V) \ar{r}{\Id_P\otimes_\Ac j} \ar{d}{f\otimes_\Ac \Id_{\Ac\otimes V}}&P\otimes_\Ac M\ar{d}{f\otimes_\Ac \Id_M}\\
Q\otimes_\Ac M\ar{r}{\Id_Q\otimes_\Ac i}  &Q\otimes_\Ac(\Ac\otimes V) \ar{r}{\Id_Q\otimes_\Ac j}        &Q\otimes_\Ac M
\end{tikzcd}
\end{equation}
is a retract diagram in $\tt Mod(\Ac)$, it suffices to show that $f\otimes_\Ac \Id_{\Ac\otimes V}$ is a weak equivalence.\medskip

The proof of Lemma \ref{L:cofobjsinAMod} shows that $\Ac\0 V$ is the colimit of the $\zl$-sequence $X:\zl\to \tt Mod(\Ac)$ defined by $$X_\zb=\Ac\0\bigoplus_{\za\le\zb,\za\in\mathbf{O}_s}S^{n(\za)}\;,$$ where $S^{n(\za)}$ is the sphere $\Dc\cdot 1_{n(\za)}$, $n(\za)\in\N$, see Equation {\ref{TransCompPushGen}. If we shift the index $\za$ of the generators $1_{n(\za)}$ by $-1$, we get $$X_\zb=\Ac\0 V_{<\zb}\,,\quad\text{ where }V_{<\zb}=\bigoplus_{\za-1<\zb}\Dc\cdot 1_{n(\za-1)}\;.$$ As the morphisms $X_{\zb,\zb+1}$ are the canonical injections, the $\Ac$-module $\Ac\otimes V$ is equipped with the $\zl$-filtration $\op{F}_\zb(\Ac\0 V)=X_\zb$, i.e., with the $\zl$-filtration
$$
0\hookrightarrow \Ac\otimes V_{<1}\hookrightarrow\ldots\hookrightarrow \Ac\otimes V_{<\zb}\hookrightarrow\ldots
$$

Since $\tt Mod(\Ac)$ is a closed monoidal category, the tensor product $P\otimes_\Ac-$ is a left adjoint functor and thus preserves colimits. Hence,
$$
0\hookrightarrow P\otimes_\Ac(\Ac\otimes V_{<1})\hookrightarrow\ldots\hookrightarrow P\otimes_\Ac (\Ac\otimes V_{<\zb})\hookrightarrow\ldots
$$
is a $\zl$-filtration of the $\Ac$-module $P\otimes_\Ac(\Ac\otimes V)$, or, in view of Lemma \ref{SimplTensA},
$$
0\hookrightarrow P\otimes V_{<1}\hookrightarrow\ldots\hookrightarrow P\otimes V_{<\zb}\hookrightarrow\ldots
$$
is a $\zl$-filtration of the $\Ac$-module $P\0 V$. If we tensor by $Q$ instead of $P$, we get an analogous $\zl$-filtration for $Q\otimes V$. Moreover, one easily checks (see Equation \ref{i}) that $$\imath\circ(f\0_\Ac\Id_{\Ac\0 V})\circ\imath^{-1}=f\0\Id_V\;,$$ so that the used identifications imply that the $\tt Mod(\Ac)$-morphism $f\0_\Ac\Id_{\Ac\0 V}$ coincides with the $\tt Mod(\Ac)$-morphism $f\0\Id_V\,.$ Since the weak equivalences in $\tt Mod(\Ac)$ are those $\tt Mod(\Ac)$-morphisms that are weak equivalences in $\tt DG\Dc M$, it actually suffices to show that $f\otimes\Id_V:P\otimes V\to Q\otimes V$ is a weak equivalence in $\tt DG\Dc M$.\medskip

We already mentioned (see paragraph above Proposition \ref{ProperCombinatorial}) that the forgetful functor $\zF:\tt Mod(\Ac)\to DG\Dc M$ respects directed colimits (alternatively we may argue that $\zF$ preserves filtered colimits as right adjoint between two accessible categories). Thus, the $\zl$-filtrations of $P\0 V$ and $Q\0 V$ in $\tt Mod(\Ac)$ are also $\zl$-filtrations in $\tt DG\Dc M$. Let now $\zvf$ be the natural transformation between the $\tt DG\Dc M$-filtration functors $\op{F}_\zb(P\0 V)=P\0 V_{<\zb}$ and $\op{F}_\zb(Q\0 V)=Q\0 V_{<\zb}$, defined by $\zvf_\zb=f\0\Id_{V_{<\zb}}$:
\begin{center}
\begin{equation}\label{dia}
\begin{tikzcd}
0 \ar[hook,r] \ar{d}{\zvf_0=0}&P\otimes V_{<1}\ar{d}{\zvf_1= f\otimes\Id_{ V_{<1}}}\ar[hook,r]&\ldots\ar[hook,r]&P\otimes  V_{<\zb},\ar[hook,r]\ar{d} {\zvf_\zb=f\otimes\Id_{ V_{<\zb}}}&\ldots\\
0 \ar[hook,r]&Q\otimes V_{<1}\ar[hook,r]&\ldots\ar[hook,r]&Q\otimes  V_{<\zb}\ar[hook,r]&\ldots\\
\end{tikzcd}\quad.
\end{equation}
\end{center}
The colimit $\op{colim}_{\zb <\lambda}\zvf_\zb$ is given by $f\0\Id_V$, so that the $\tt DG\Dc M$-morphism $f\0\Id_V$ is compatible with the considered $\tt DG\Dc M$-filtrations.\medskip

In order to apply Lemma \ref{L:filtrweq}, note that $\zvf_0$ is a weak equivalence and look at $$\zvf_{\zb+1,\sharp}:\op{Gr}_{\zb+1}(P\0 V)\to\op{Gr}_{\zb+1}(Q\0 V)\;,$$ where $$\op{Gr}_{\zb+1}(P\0 V)=P\0 V_{<\zb+1}/P\0 V_{<\zb}\;.$$ Observe first that, we have a $\tt G\Dc M$-isomorphism $\jmath:\op{Gr}_{\zb+1}(P\0 V)\to P\0 S^{n(\zb+1)}$ and denote by $\zd$ the pushforward $$\zd=\jmath\circ \p_{\zb+1,\sharp}^{P\0 V}\circ \jmath^{-1}$$ of the differential of $\op{Gr}_{\zb+1}(P\0 V)\,$. As, in view of Equation (\ref{diff}), the differential $\p_{\zb+1}^{P\0 V}$ of $\op{F}_{\zb+1}(P\0 V)=P\0 V_{<\zb+1}$ is the pushforward $$\imath\circ (d_P\0\Id_{\0}+\Id_P\0\,d)\circ \imath^{-1}\;$$ of the differential of $P\0_\Ac(\Ac\0 V_{<\zb+1})$, and as $d$ is the lowering differential of the Sullivan $\Ac$-module $\Ac\0 V$ (i.e., $d\,1_{n(\zb)}\in\Ac\0 V_{<\zb}$), we get, for any argument in $P\0 S^{n(\zb+1)}$, $$\zd(p\0 \zD\cdot 1_{n(\zb)})=\jmath[\p_{\zb+1}^{P\0 V}(p\0 \zD\cdot 1_{n(\zb)})]=\jmath[d_P\,p\0\zD\cdot 1_{n(\zb)}+(-1)^{|p|}\imath(p\0 \zD\cdot d\,1_{n(\zb)})]=$$ $$d_P\,p\0\zD\cdot 1_{n(\zb)}=(d_P\0\Id_{S^{n(\zb+1)}})(p\0 \zD\cdot 1_{n(\zb)})=:d_\0(p\0 \zD\cdot 1_{n(\zb)})\;,$$ where $d_\0$ the natural differential on $P\0 S^{n(\zb+1)}$. It follows that $$d_\0=\jmath\circ \p_{\zb+1,\sharp}^{P\0 V}\circ \jmath^{-1}\;,$$ so that $\jmath$ is a $\tt DG\Dc M$-isomorphism and that we can identify $(\op{Gr}_{\zb+1}(P\0 V),\p_{\zb+1,\sharp}^{P\0 V})$ and $(P\0 S^{n(\zb+1)},d_\0)$ as differential graded $\Dc$-modules (and similarly for $P$ replaced by $Q$). It is now easily checked that, when read through these isomorphisms, the morphism $\zvf_{\zb+1,\sharp}$ is the $\tt DG\Dc M$-morphism $$f\0\Id_{S^{n(\zb+1)}}:P\0 S^{n(\zb+1)}\to Q\0 S^{n(\zb+1)}\;.$$ In view of Lemma \ref{L:filtrweq}, it finally suffices to prove that $f\0\Id_{S^{n(\zb+1)}}$ is a weak equivalence.\medskip

Via the by now standard argument, we get $$\op{Mc}(f\otimes \Id_{S^{n(\zb+1)}})\simeq (\op{Mc} f)[-n(\zb+1)] \otimes {S^{n(\zb+1)}}\;.$$ Since $f$ is a weak equivalence by assumption, K\"unneth's formula gives
$$
H_\bullet(\op{Mc}(f\otimes \Id_{S^{n(\zb+1)}}))\simeq H_\bullet((\op{Mc}f)[-n(\zb+1)]\otimes S^{n(\zb+1)})\simeq H_{\bullet-n(\zb+1)}(\op{Mc}f)\otimes S^{n(\zb+1)}=0\;,
$$
so that $f\otimes \Id_{S^{n(\zb+1)}}$ is a weak equivalence. This completes the proof.
\endproof

\subsection{HAC condition 4: base change and equivalence-invariance}\label{SS:algandbasech}

In this section, we investigate the condition HAC4 \cite[Assumption 1.1.0.4]{TV08}.\medskip

It actually deals with the categories $\tt CMon(\tt Mod_{\tt DG\Dc M}(\Ac))$ and $\tt NuCMon(\tt Mod_{\tt DG\Dc M}(\Ac))$ of unital and non-unital commutative monoids in ${\tt Mod}(\Ac)$. In \cite{TV08}, the HAC conditions are formulated over an underlying category, which is not necessarily $\tt DG\Dc M$, but any combinatorial symmetric monoidal model category $\tt C$. The category of non-unital monoids appears in Assumption 1.1.0.4, only since $\tt C$ is not necessarily additive \cite[Remark 1.1.0.5]{TV08}. In our present case, the category $\tt C=DG\Dc M$ is Abelian and thus additive, so that the condition on non-unital algebras is redundant here.\medskip

Just as there is an adjunction $\Sc:{\tt DG\Dc M}\rightleftarrows {\tt CMon(DG\Dc M)}:F$, see Subsection \ref{DGDAs}, we have an adjunction $$\Sc_\Ac:{\tt Mod}(\Ac)\rightleftarrows{\tt CMon(\tt Mod(\Ac))}:F_\Ac\;,$$ which is defined exactly as $\Sc\dashv F$, except that the tensor product is not over $\Oc$ but over $\Ac$. Hence, it is natural to define weak equivalences (resp., fibrations) in $\tt CMon(\tt Mod(\Ac))$, as those morphisms that are weak equivalences (resp., fibrations) in ${\tt Mod}(\Ac)$, or, equivalently, in $\tt DG\Dc M$. Assumption 1.1.0.4 now reads\medskip

\noindent{\bf HAC 4}.\label{HAC4}\begin{enumerate}\item[-] The preceding classes of weak equivalences and fibrations endow $\tt CMon(\tt Mod(\Ac))$ with a combinatorial proper model structure.
\item[-] For any cofibrant $\Bc\in\tt CMon(\tt Mod(\Ac))$, the functor $$\Bc\0_\Ac -:{\tt Mod}(\Ac)\to{\tt Mod}(\Bc)$$ respects weak equivalences. \end{enumerate}\medskip

The axiom is easily understood. Recall first that the category $\tt CMon(\tt Mod(\Ac))$ is isomorphic to the category $\Ac\downarrow \tt DG\Dc A$, see Proposition \ref{CMonUnder}. Moreover, in \cite{DPP2} and \cite{PP}, we emphasized the importance of a base change, i.e., of the replacement of $\Ac\downarrow \tt DG\Dc A$ by $\Bc\downarrow \tt DG\Dc A$ (we actually passed from $\Ac=\Oc$ to $\Bc=\Jc$, where $\Jc$ was interpreted as the function algebra of an infinite jet bundle). This suggests to reflect upon a functor from $\tt CMon(\tt Mod(\Ac))$ to $\tt CMon(\tt Mod(\Bc))$, or, simply, from ${\tt Mod}(\Ac)$ to ${\tt Mod}(\Bc)$. The {\it natural$\,$} transition functor is $\Bc\0_\Ac-\,$, provided $\Bc$ is not only an object $\Bc\in\tt DG\Dc A$ but an object $\Bc\in\tt CMon(\tt Mod(\Ac))$. Just as the functor $-\0_\Ac M$ in HAC3, the functor $\Bc\0_\Ac-$ is required to preserve weak equivalences, at least for cofibrant objects $\Bc\in\tt CMon(\tt Mod(\Ac))$. HAC4 further asks that the above-defined weak equivalences and fibrations implement a model structure on $\tt CMon(\tt Mod(\Ac))$ and that cofibrancy be with respect to this structure. Finally, exactly as the so far considered model categories $\tt DG\Dc M$, $\tt DG\Dc A$, and ${\tt Mod}(\Ac)$, the model category $\tt CMon(\tt Mod(\Ac))$ must be combinatorial and proper.\medskip

Note that there are important examples that do not satisfy this axiom. For instance, it does not hold if the underlying category is the category ${\tt C=DG}R\,{\tt M}$ of non-negatively graded chain complexes of modules over a commutative ring $R$ with nonzero characteristic. Our task is to show that it {\it is} valid, if $R$ is replaced by the non-commutative ring $\Dc$ of characteristic 0. On the other hand, the assumption HAC4 is essential in proving, for instance, the existence of an analog of the module $\Omega_{B/A}$ of relative differential 1-forms. The existence of this cotangent complex `$\,\zW_{B/A}$' is on its part the main ingredient in the definition of smooth and \'etale morphisms.

\begin{prop}\label{CPMC}
The category $\tt CMon(Mod(\Ac))$, whose morphisms are weak equivalences (resp., fibrations) if they are weak equivalences (resp., fibrations) in $\tt DG\Dc M$, and whose morphisms are cofibrations if they have in $\tt CMon(Mod(\Ac))$ the left lifting property with respect to trivial fibrations, is a combinatorial proper model category.
\end{prop}

\begin{proof} The categorical isomorphism $$\tt CMon(Mod(\Ac))\simeq \Ac\downarrow\tt DG\Dc A\;$$ of Proposition \ref{CMonUnder} allows endowing $\tt CMon(Mod(\Ac))$ with the model structure of $\Ac\downarrow\tt DG\Dc A$: a $\tt CMon(Mod(\Ac))$-morphism is a weak equivalence (resp., a fibration, a cofibration), if it is a weak equivalence (resp., a fibration, a cofibration) in $\tt DG\Dc A$ \cite{HirOU}. Hence, a $\tt CMon(Mod(\Ac))$-morphism is a weak equivalence (resp., a fibration), if it is a weak equivalence (resp., a fibration) in $\tt DG\Dc M$ \cite{DPP}. We know that these definitions provide $\tt CMon(Mod(\Ac))$ with a model structure, so that a $\tt CMon(Mod(\Ac))$-morphism is a cofibration if and only if it has in $\tt CMon(Mod(\Ac))$ the left lifting property with respect to trivial fibrations. It follows that the distinguished classes of Proposition \ref{CPMC} equip $\tt CMon(Mod(\Ac))$ with a model structure. In addition, in view of \cite[Theorem 2.8]{HirOU}, this model category is proper and cofibrantly generated. Its generating cofibrations (resp., trivial cofibrations) are obtained from the cofibrations $I$ (resp., trivial cofibrations $J$) of $\tt DG\Dc A$ by means of the left adjoint functor \be\label{e:monadic?}L_\0:{\tt DG\Dc A}\ni \Bc\mapsto (\Ac\to \Ac\0\Bc)\in \tt CMon(Mod(\Ac)):\op{For}\;.\ee

It remains to show that the category $\tt CMon(Mod(\Ac))$ is accessible. Remark that the monad $\op{For}L_\0$ coincides with the coproduct functor $\Ac\0-:{\tt DG\Dc A}\ni \Bc\mapsto \Ac\0\Bc\in {\tt DG\Dc A}$, and that, if $\op{For}$ is monadic, we have the equivalence of categories $\tt DG\Dc A^{\Ac\0-}\simeq\tt CMon(Mod(\Ac))$. To prove the accessibility of $\tt CMon(Mod(\Ac))$, it thus suffices to show that $\Ac\0-$ respects directed colimits {\it and} that $\op{For}$ satisfies the requirements of the monadicity theorem \ref{MonaTheo}. However, the coproduct functor $\Ac\0-$ of $\tt DG\Dc A$ commutes with colimits and in particular with directed ones. The first condition of Theorem \ref{MonaTheo} asks that $\op{For}$ reflect isomorphisms, what is easily checked. The second condition asks that $\tt CMon(Mod(\Ac))$ admit coequalizers of reflexive pairs and that $\op{For}$ preserve them. Since $\tt CMon(Mod(\Ac))$ is a model category, it has {\it all} coequalizers. Finally, when applying $\op{For}$ to the coequalizer of two parallel morphisms in $\tt CMon(Mod(\Ac))$, we get the coequalizer in $\tt DG\Dc A$ of the images by $\op{For}$ of the considered parallel morphisms. Indeed, as for the universality of this coequalizer-candidate in $\tt DG\Dc A$, any second coequalizer-candidate can be canonically lifted to $\Ac\downarrow\tt DG\Dc A$, and universality in $\Ac\downarrow\tt DG\Dc A$ provides a unique factorization-morphism in $\Ac\downarrow\tt DG\Dc A$, whose projection via $\op{For}$ is a factorization-morphism in $\tt DG\Dc A$. It can further be seen that the latter is unique, what completes the proof of the accessibility of $\tt CMon(Mod(\Ac))$.\end{proof}

The next proposition ensures that also Part 2 of HAC4 is satisfied.

\begin{prop}\label{P:basechange}
Let $\Ac$ be an object in ${\tt DG\Dc A}$ and let $\Bc$ be a cofibrant object in $\tt CMon(Mod(\Ac))$. The functor
$$
\Bc\otimes_\Ac -:\tt Mod(\Ac)\rightarrow \tt Mod(\Bc)
$$
preserves weak equivalences.
\end{prop}

\proof By assumption the morphism $\zf_\Bc:\Ac\to\Bc$ is a cofibration in $\tt CMon(Mod(\Ac))\simeq\Ac\downarrow {\tt DG\Dc A}$, i.e., a cofibration in ${\tt DG\Dc A}$. Consider now, in ${\tt DG\Dc A}$, the cofibration - trivial fibration factorization of $\zf_\Bc$ constructed in \cite[Theorem 5]{DPP2}:

\begin{equation}\label{e:trick2}
\begin{tikzcd}[column sep=large]
(\Ac,d_\Ac)\arrow[rightarrowtail]{r} \arrow[rightarrowtail]{d} & (\Ac\otimes \Sc V,d_2) \arrow[two heads]{d}{\sim}\\
(\Bc,d_\Bc)\arrow[swap]{r}{\Id_{\Bc}} \arrow[dashed]{ur} & (\Bc,d_\Bc)\\
\end{tikzcd}.
\end{equation}
The dashed arrow in this diagram exists in view of the left lifting property of cofibrations with respect to trivial fibrations. The diagram

\begin{equation}\label{e:trick2}
\begin{tikzcd}[column sep=large]
&(\Ac,d_\Ac)\arrow[rightarrowtail]{dr} \arrow[rightarrowtail]{dl} \arrow[rightarrowtail]{d}&\\
(\Bc,d_\Bc)\arrow[dashed]{r} & (\Ac\otimes \Sc V,d_2)\arrow[two heads]{r}{\sim} & (\Bc,d_\Bc)\\
\end{tikzcd}
\end{equation}
now shows that $\Ac\to \Bc$ is a retract of $\Ac\to \Ac\otimes \Sc V$ in $\Ac\downarrow {\tt DG\Dc A}$, or, equivalently, that $\Bc$ is a retract of $\Ac\0\Sc V$ in ${\tt CMon(Mod}(\Ac))$ and so in ${\tt Mod}(\Ac)$.\medskip

Let now $f:P\to Q$ be a weak equivalence in $\tt Mod(\Ac)$. Since, as easily checked, $f\0_\Ac\Id_\Bc$ is a retract of $f\0_\Ac\Id_{\Ac\0\mathcal{S}V}$ in $\tt Mod(\Ac)$, it suffices to show that the latter morphism is a weak equivalence in $\tt Mod(\Ac)$. Indeed, in this case $f\0_\Ac\Id_\Bc$ is also a weak equivalence in $\tt Mod(\Ac)$, so a weak equivalence in $\tt DG\Dc M\,$; it follows that the $\tt  Mod(\Bc)$-morphism $f\0_\Ac\Id_\Bc$ is a weak equivalence in $\tt Mod(\Bc)$, what then completes the proof.\medskip

If we use the identification detailed in Lemma \ref{SimplTensA}, the morphism $$f\0_\Ac\Id_{\Ac\0\mathcal{S}V}:P\0_\Ac(\Ac\0\mathcal{S}V)\to Q\0_\Ac(\Ac\0\mathcal{S}V)$$ becomes $$f\0\Id_{\mathcal{S}V}:P\0\mathcal{S}V\to Q\0\mathcal{S}V$$ and it suffices to prove that $f\0\Id_{\Sc V}$ is a weak equivalence in $\tt Mod(\Ac)$, i.e., in $\tt DG\Dc M$.\medskip

It is known from \cite[Theorem 5 and Section 6.2]{DPP2} that $(\Ac\otimes \Sc V,d_2)$ is the colimit in ${\tt DG\Dc A}$ of a $\lambda$-sequence of injections
\begin{equation}\label{Colim}
\begin{split}(\Ac,d_\Ac)\hookrightarrow &(\Ac\otimes \Sc V_{<1},d_{2,<1})\hookrightarrow (\Ac\otimes \Sc V_{<2},d_{2,<2})\hookrightarrow \ldots\\
&\ldots(\Ac\otimes \Sc V_{<\omega},d_{2,<\omega})\hookrightarrow (\Ac\otimes \Sc V_{<\omega+1},d_{2,<\omega+1})\hookrightarrow \ldots\end{split}
\end{equation}
where $\Ac\otimes \Sc V_{<\zb}$ has the usual meaning (see above) and where $d_{2,<\zb}$ is the restriction of $d_2\,$. Since colimits in an undercategory are computed as colimits in the underlying category, the commutative $\tt Mod(\Ac)$-monoid $(\Ac\0\Sc V,d_2)$ is also the colimit of the preceding $\zl$-sequence in $\Ac\downarrow {\tt DG\Dc A}\simeq \tt CMon(Mod(\Ac))$. Moreover, as a coslice category of an accessible category is accessible, the categories $\tt CMon(Mod(\Ac))$ and $\tt Mod(\Ac)$ are both accessible. It follows that the forgetful functor $F_\Ac:{\tt CMon(Mod}(\Ac))\to {\tt Mod}(\Ac)$ commutes with filtered colimits as right adjoint $\Sc_\Ac\dashv F_\Ac$ between accessible categories. Hence, the sequence (\ref{Colim}) is a $\lambda$-filtration of $\Ac\0 \Sc V$ in ${\tt Mod}(\Ac)$. We can now argue as in the proof of Proposition \ref{P:flatness}: the sequence
\begin{equation*}
\begin{split}(P,d_P)\hookrightarrow &(P\0\Sc V_{<1},\zd_{P,<1})\hookrightarrow (P\otimes \Sc V_{<2},\zd_{P,<2})\hookrightarrow \ldots \\
&\ldots(P\otimes \Sc V_{<\omega},\zd_{P,<\omega})\hookrightarrow (P\otimes \Sc V_{<\omega+1},\zd_{P,<\omega+1})\hookrightarrow \ldots
\end{split}
\end{equation*}
is a $\zl$-filtration of $(P\0\Sc V,\zd_P)$ -- in $\tt Mod(\Ac)$, as well as in $\tt DG\Dc M$. Here $$\zd_P=\imath\circ(d_P\0\Id_\0+\Id_P\0\, d_2)\circ\imath^{-1}$$ is the differential $d_P\0\Id_\0+\Id_P\0\, d_2$ pushed forward from $P\0_\Ac(\Ac\0\Sc V)$ to $P\0\Sc V$. Let now $\zvf$ be, as in the proof of Proposition \ref{P:flatness}, the natural transformation between the $\tt DG\Dc M$-filtration functors $F_\zb(P\0\Sc V)=P\0\Sc V_{<\zb}$ and $F_\zb(Q\0\Sc V)=Q\0\Sc V_{<\zb}$, defined by $\zvf_\zb=f\0\Id_{\Sc V_{<\zb}}$:
\begin{center}
\begin{equation}\label{dia}
\begin{tikzcd}
(P,d_P) \ar[hook,r] \ar{d}{\zvf_0=f}&(P\otimes \Sc V_{<1},\delta_{P,<1})\ar{d}{\zvf_1=f\otimes\Id_{\Sc V_{<1}}}\ar[hook,r]&\ldots\ar[hook,r]&(P\otimes \Sc V_{<\zb},\delta_{P,<\zb})\ar[hook,r]\ar{d} {\zvf_\zb=f\otimes\Id_{\Sc V_{<\zb}}}&\ldots\\
(Q,d_Q) \ar[hook,r]&(Q\otimes \Sc V_{<1},\delta_{Q,<1})\ar[hook,r]&\ldots\ar[hook,r]&(Q\otimes \Sc V_{<\zb},\delta_{Q,<\zb})\ar[hook,r]&\ldots\\
\end{tikzcd}\quad .
\end{equation}
\end{center}
It follows again that the $\tt DG\Dc M$-morphism $f\0\Id_{\Sc V}$ is compatible with the $\tt DG\Dc M$-filtrations. To show that $f\otimes \Id_{\Sc V}$ is a weak equivalence in $\tt DG\Dc M$, it suffices to prove that the $\zvf_\zb=f\otimes \Id_{\Sc V_{<\zb}}$, $\zb<\zl$, are weak equivalences, see Lemma \ref{L:limitord}. This proof will be a transfinite induction on $\zb<\lambda$.\medskip

The induction starts, since $\zvf_0=f$ is a weak equivalence by assumption. We thus must show that $\zvf_\zb$, $\zb<\zl$, is a weak equivalence, assuming that the $\zvf_\za$, $\za<\zb$, are weak equivalences.\smallskip

The limit ordinal case $\zb\in\mathbf{O}_\ell$ is a direct consequence of Lemma \ref{L:limitord}.\smallskip

Let now $\zb\in\mathbf{O}_s$ be the successor of an ordinal $\zg\,$.\medskip

To simplify notation, we denote the differential graded $\Dc$-module $$F_\zg(P\0\Sc V)=(P\0\Sc V_{<\zg},\zd_{P,<\zg})\quad (\text{resp., } F_\zg(Q\0\Sc V)=(Q\0\Sc V_{<\zg},\zd_{Q,<\zg}))$$ by $(P',d_{P'})$ (resp., $(Q',d_{Q'})$) and we denote the morphism $\zvf_\zg=f\0\Id_{\Sc V_{<\zg}}$ by $f'$. The isomorphism
\begin{equation}\label{glupi}P\otimes\Sc V_{<\zb}\simeq P\0\Sc V_{<\zg}\otimes \Sc (\Dc\cdot 1_{n(\zg)})= P'\otimes \Sc S^{n(\gamma)}\end{equation} 
of graded $\Dc$-modules (it just replaces $\odot$ by $\0$ and vice versa, so that we will use it tacitly) allows to push the differential 
$$\zd_{P,<\zb}=\imath\circ(d_P\0\Id_\0+\Id_P\0\, d_2|_{\Ac\0\Sc V_{<\zb}})\circ\imath^{-1}$$
 of $P\otimes\Sc V_{<\zb}$ forward to a differential $\p_{P,<\zb}$ of $P'\otimes \Sc S^{n(\gamma)}$ and to thus obtain an isomorphic differential graded $\Dc$-module structure on $P'\otimes \Sc S^{n(\gamma)}$. The lowering property of $d_2$ induces a kind of lowering property for $\p_{P,<\zb}\,$:
\begin{equation}\label{sul}(\p_{P,<\zb}-d_{P'}\otimes \Id_{\Sc S^{n(\gamma)}})(P'\otimes \Sc^{k+1}S^{n(\gamma)})\subset P'\otimes\Sc^{k}S^{n(\gamma)}.\end{equation} Indeed, let $p\,\0 \odot_iv_{\za_i}\0 \odot_js_j$ be an element in $P'\otimes \Sc^{k+1} S^{n(\gamma)}$ (notation is self-explaining, in particular $\za_i<\zg$ and $s_j=D_j\cdot 1_{n(\zg)}$). We have
$$\p_{P,<\zb}(p\,\0 \odot_iv_{\za_i}\0 \odot_js_j)= \zd_{P,<\zb}(p\,\0 \odot_iv_{\za_i}\odot \odot_js_j)=$$ \be\label{preceding}\imath(d_P\0\Id_\0+\Id_P\0\, d_2|_{\Ac\0\Sc V_{<\zb}})(p\,\0(1_\Ac\0 (\odot_iv_{\za_i}\odot \odot_js_j)))\;.\ee
When noticing that $$1_\Ac\0(\odot_iv_{\za_i}\odot \odot_js_j)=(1_\Ac\0 \odot_i v_{\za_i})\lozenge \lozenge_j(1_\Ac\0 s_j)\;,$$ where $\lozenge$ is the multiplication in $\Ac\0\Sc V_{<\zb}$, and when remembering that $d_2$ is a derivation of $\lozenge$, we see that the expression in Equation (\ref{preceding}) reads 
$$\imath\left(d_Pp\0(1_\Ac\0(\odot_iv_{\za_i}\odot \odot_js_j))+\right.$$ \be\label{Losange}\left.(-1)^{|p|}p\0 d_2(1_\Ac\0\odot_iv_{\za_i})\odot\odot_js_j\right)\pm\ee $$\imath\left((-1)^{|p|}p\0(1_\Ac\0\odot_iv_{\za_i})\lozenge \sum_j\pm(1_\Ac\0 s_1)\lozenge\ldots\lozenge d_2(1_\Ac\0 s_j)\lozenge\ldots\lozenge(1_\Ac\0 s_{k+1})\right)\;.$$ The first term (two first rows) is equal to
$$\imath\left(d_Pp\0(1_\Ac\0\odot_iv_{\za_i})\0\odot_js_j+(-1)^{|p|}p\0 d_2(1_\Ac\0\odot_iv_{\za_i})\0\odot_js_j\right)=$$
$$\imath\left((d_P\0\Id_\0+\Id_P\0\, d_2|_{\Ac\0\Sc V_{<\zg}})(p\0(1_\Ac\0\odot_iv_{\za_i}))\right)\0\odot_js_j=$$ $$(\zd_{P,<\zg}\0\Id_{\Sc S^{n(\zg)}})(p\0\odot_iv_{\za_i}\0\odot_j s_j)=$$ $$(d_{P'}\0\Id_{\Sc S^{n(\zg)}})(p\0\odot_iv_{\za_i}\0\odot_j s_j)\;.$$ 
Since $d_2(1_\Ac\0 s_j)\in\Ac\0\Sc V_{<\zg}$, the remaining term in Equation (\ref{Losange}) is an element of $P'\0\Sc^{k}S^{n(\zg)}$, so that the claim (\ref{sul}) holds true.\medskip

The $\tt DG\Dc M$-isomorphism (\ref{glupi}) and the lowering property (\ref{sul}) are of course also valid for $Q\otimes \Sc V_{<\zb}\simeq Q'\0\Sc S^{n(\zg)}\,$. Recall now that it remains to prove that $$\zvf_\zb=f\0\Id_{\Sc V_{<\zb}}:(P\0\Sc V_{<\zb},\zd_{P,<\zb})\to (Q\0\Sc V_{<\zb},\zd_{Q,<\zb})$$ is a weak equivalence in $\tt DG\Dc M$, i.e., that $$f'\0\Id_{\Sc S^{n(\zg)}}:(P'\0\Sc S^{n(\zg)},\p_{P,<\zb})\to (Q'\0\Sc S^{n(\zg)},\p_{Q,<\zb})\;$$ is a weak equivalence. In view of the afore-detailed lowering property (\ref{sul}), it suffices to replicate the proof of the $\tt DG\Dc A$-case in Step 1 of the proof of Theorem \ref{T:leftproper}.

\section{Acknowledgments} The work of G. Di Brino was partially supported by the National Research Fund, Luxembourg, and the Marie Curie Actions Program of the European Commission (FP7-COFUND).

\section{Appendices}
\appendix

In the sequel, various (also online) sources have been used; notation is the same as in the main part of the text.

\section{Locally presentable categories}\label{LPC}

Recall that an infinite cardinal $\zk$ is {\em regular}, if no set of cardinality $\zk$ is the union of less than $\zk$ sets of cardinality less than $\zk$. For instance, if $\zk=\aleph_0=\zw$, no countable set is a finite union of finite sets, so that $\aleph_0=\zw$ is regular.\medskip

Let $(I,\le)$ be a {\em directed poset}, i.e., a partially ordered set in which every pair of elements has an upper bound, i.e., for any $i,j\in I$, there exists $k\in I$ such that $i\le k$ and $j\le k$. We view this poset as a category ${\tt I}$ whose morphisms $i\to j$ correspond to the inequalities $i\le j$. A diagram of type ${\tt I}$ in a category $\tt C$ is a direct system and its limit is a {\em direct limit} or {\em directed colimit}. More generally, for a regular cardinal $\zk$, a {\em $\zk$-directed poset} $(J,\le)$ is a poset in which every subset of cardinality less than $\zk$ has an upper bound. Then a colimit over a diagram of type $\tt J$ in a category $\tt C$ is called a {\em $\zk$-directed colimit}. For $\zk=\aleph_0$, we recover the preceding notion of directed colimit.\medskip

A finitely presented (left) module over a ring $R$ is an $R$-module that is generated by a finite number of its elements, which satisfy a finite number of relations. The categorical substitute for this idea is a category all of whose elements are directed colimits $\varinjlim_i c_i=\bigsqcup_i c_i/\sim$ of some generating objects $c_i$. This leads to the concept of locally $\zk$-presentable category: such a category is, roughly, a category that comes equipped with a (small) subset $S$ of $\zk$-small objects that generate all objects under $\zk$-directed colimits.\medskip

Remember first that the idea of smallness of an object $c\in\tt C$ is that the covariant Hom functor $\op{Hom}_{\tt C}(c,\bullet)$ commutes with a certain type of colimits. This actually means (see, for instance, \cite{DPP}) that any morphism $c\to \op{colim}_ic_i$ out of the small $c$ into a certain type of colimit $\op{colim}_ic_i$ factors through one of the maps $c_j\to \op{colim}_ic_i$. If $\zk$ is a regular cardinal, a {\em $\zk$-small, $\zk$-compact, or $\zk$-presentable} object $c\in\tt C$ is an object, such that $\op{Hom}_{\tt C}(c,\bullet)$ commutes with $\zk$-directed colimits. An object is called {\em small}, if it is $\zk$-small, for some regular $\zk.$\medskip

Combining the two last paragraphs, we get the

\begin{defi} For a regular cardinal $\zk$, a {\em locally $\zk$-presentable category} $\tt C$ is
\begin{itemize}

\item[1.] a locally small category

\item[2.] that has all small colimits

\item[3.] and admits a {\em set} $S\subset \op{Ob}({\tt C})$ of $\zk$-small objects, such that any object in $\tt C$ is the $\zk$-directed colimit of objects in $S$.

\end{itemize}

A category is termed a {\em locally presentable category}, if it is locally $\zk$-presentable, for some regular $\zk.$\end{defi}

\section{Operads, monads and algebras}\label{B:monads}

\subsection{Adjunctions}

Let us recall that the term {\em whiskering} refers to the composition of a natural transformation $\zn:F'\to F''$ between two functors $F',F'':{\tt C}\to {\tt D}$ and a functor $G:{\tt B}\to {\tt C}$ or a functor $H:{\tt D}\to {\tt E}$. The result $\zn G$ is a natural transformation $F'G\to F''G$ (where we omitted the symbol $\circ$). Its component $(\zn G)_b:F'(G b)\to F''(G b)$ is given by $(\zn G)_b=\zn_{G b}\,.$ Similarly $H\zn$ is the natural transformation $H F'\to H F''$ with component $(H\zn)_c=H(\zn_c)$. \medskip

Any adjunction is a unit-counit adjunction. Indeed, let $F:{\tt C}\rightleftarrows{\tt D}:G$ be an adjunction: \be\label{Adj}\zF_{c,d}:\te{Hom}_{\tt D}(Fc,d)\simeq \te{Hom}_{\tt C}(c,Gd)\;,\ee functorially in $c$ and $d$. For any $c$ and for $d=Fc$, we thus get a map $$\zh_c:=\zF_{c,Fc}(\Id_{Fc})\in\te{Hom}_{\tt C}(c,GFc)\;.$$ The resulting arrow $$\zh:\Id_{\tt C}\to GF\;$$ is a natural transformation. For any $c\in \tt C$, $d\in\tt D$ and $f\in\te{Hom}_{\tt D}(Fc,d)$, we now have a map $G(f)\in\te{Hom}_{\tt C}(GFc,Gd)$ and a map $$\zF_{c,d}(f)=G(f)\circ\zh_c\in\te{Hom}_{\tt C}(c,Gd)\,.$$ When using the inverse of $\zF$, we obtain a natural transformation $$\ze:FG\to \Id_{\tt D}\;$$ with components $$\ze_d:=\zF^{-1}_{Gd,d}(\Id_{Gd})\in\te{Hom}_{\tt D}(FGd,d)\;.$$ Moreover, we see that $$(\Id_{G})_d=\Id_{Gd}=\zF_{Gd,d}(\zF^{-1}_{Gd,d}(\Id_{Gd}))=G(\ze_d)\circ\zh_{Gd}=(G\ze\circ \zh G)_d\;,$$ so that the composition of natural transformations $$G\stackrel{\zh G}{\longrightarrow}GFG\stackrel{G\ze}{\longrightarrow}G$$ gives the identity transformation $\Id_G\,.$ The dual relation, stating that the composition $$F\stackrel{F\zh}{\longrightarrow}FGF\stackrel{\ze F}{\longrightarrow}F$$ is equal to $\Id_F$, is obtained analogously.

\begin{Defi} A {\em unit-counit adjunction} $\langle F,G,\zh,\ze\rangle$ between two categories $\tt C$ and $\tt D$ is a pair of functors $F:{\tt C}\to {\tt D}$ and $G:{\tt D}\to {\tt C}$, together with two natural transformations $\zh:\Id_{\tt C}\to GF$ and $\ze:FG\to \Id_{\tt D}$, called the {\em unit} and the {\em counit}, such that the compositions $$G\stackrel{\zh G}{\longrightarrow}GFG\stackrel{G\ze}{\longrightarrow}G$$ and $$F\stackrel{F\zh}{\longrightarrow}FGF\stackrel{\ze F}{\longrightarrow}F$$ are equal to $\Id_G$ and $\Id_F$, respectively. \end{Defi}

The concept of unit-counit adjunction is clearly related to the notion of equivalence of categories. Moreover, as sketched above, an adjunction $\langle F,G,\zF\rangle$ in the usual Hom-set sense (see (\ref{Adj})) {\em is} a unit-counit adjunction $\langle F,G,\zh,\ze\rangle$. The converse is true as well, i.e., a Hom-set adjunction is the same as a unit-counit adjunction.

\subsection{Classical definition of operads}

An {\bf operad} lifts (the fundamental aspects of) the operations of algebras of a given type (e.g., of associative algebras or Lie algebras), their possible symmetries, the compositions of operations, as well as the specific relations they satisfy (e.g., associativity or Jacobi identity), to a more universal and abstract level. The latter is best thought of by viewing a universal abstract $n$-ary operation as a {\bf tree $T_n$} with $n$ leaves (or inputs) and $1$ root (or output).\medskip

To be more explicit, an algebra of a given type is a vector space $V$ together with concrete {\bf generating operations}, i.e., linear maps $V^{\0 n}\to V$ (with possibly varying $n$), which satisfy certain {\bf defining relations} $r_j = 0$. We will assume that the relations are {\bf multilinear}, i.e., of the form \be\label{MLR}r_j = \zS_k \psi_k = 0\;,\ee where $\psi_k$ is a composite of generating operations. To any such type of algebras, we can associate an operad. This operad $P$ of made of a family of vector spaces $P(n)$, $n\in\N$, over a field $\mathbb{K}$ of characteristic 0. The elements of $P(n)$ are the mentioned abstract $n$-ary operations or trees $T_n$. If some symmetries must be encoded, the $n$-th symmetric group $\mathbb{S}_n$ must act appropriately (from the right) on $P(n)$, so that we deal with a family $P(n)$ of (right) $\mathbb{S}_n$-, or, better, $\mathbb{K}[\mathbb{S}_n]$-modules ($\mathbb{K}[\mathbb{S}_n]$ denotes the group algebra of $\mathbb{S}_n$) -- which we refer to as an $\mathbb{S}$-module. Eventually, in an {\bf operad}, the $\mathbb{S}$-module $P$ comes equipped with $\mathbb{K}$-linear composition maps $$\zg_{i_1\ldots i_k}: P(k)\0 P(i_1)\0\ldots\0 P(i_k)\to P(i_1 +\ldots + i_k)\;$$ of the abstract operations or trees. This composition is associative and has a unit $\mathbf{1}\in P(1)$ (further, it respects the possible $\mathbb{S}$-action).\medskip

An {\bf algebra over an operad} $P$ is given by a sequence of $\mathbb{K}$-linear maps \be\label{repr}\rho_n: P(n)\otimes_{\mathbb{S}_n} V^{\otimes n} \to V\ee that respect composition and send the operadic unit $\mathbf{1}$ to the identity $\Id_V$ (and respect the $\mathbb{S}$-action). These maps should be thought of as maps that assign to an abstract operation $T_n\in P(n)$, whose leaves are labelled by $n$ elements of $V$, a specific element of $V$. In other words, they assign a concrete operation to any abstract one, thus defining an algebraic structure on $V$, i.e., thus endowing $V$ with a so-called {\bf $P$-algebra} structure. {\it The $P$-algebras and their morphisms form a category, which turns out to be equivalent to the category of algebras of the initially considered type}.\medskip

Operads are of importance for several reasons. One of them is that they often highlight a common feature that underlies (apparently) different concepts. Likewise, when dealing with a result that should hold for all types of algebras, one can try to prove it once for all using an operadic approach. Furthermore, operads exhibit quite a number of common aspects with associative algebras. Hence, an algebraic or geometric situation that gives rise to an associativity property of some sort, might be advantageously dealt with via operads. These examples are however far from being exhaustive.

\subsection{Functorial definition of operads - monads - algebras}\label{OMA}

The category $\texttt{{End}(C)=[C,C]}$ of endofunctors of a category $\texttt{C}$ is strict monoidal. The monoidal product $\otimes$ is the composition $\circ$ of endofunctors and the monoidal identity $I$ is the identity functor $\Id_{\tt C}$. A {\bf monad} in $\texttt{C}$ is a monoid in $\texttt{{End}(C)}$, i.e., an endofunctor $T$ of $\tt C$ that, roughly, is equipped with an associative unital multiplication.\medskip

Let us be more precise. If $T:{\tt C}\to {\tt C}$ is an endofunctor, we set $T^2=T\circ T$ and $T^3=T\circ T^2=T^2\circ T\,$.

\begin{Defi}\label{monad} A {\bf monad } $T=\langle T, \mu, \zh \rangle$ in a category ${\tt C}$ consists of an endofunctor $T:{\tt C}\to {\tt C}$ and two natural transformations
$$\mu:T^2\to T,\hspace{5pt}\eta:\Id_{\tt C}\to T\;,$$
which make the following diagrams commutative (we write $\Id$ instead of $\Id_{\tt C}$ and instead of $\Id_{T}$)
\begin{center}
\begin{tikzcd}
T^2\circ T = T\circ T^2\arrow{r}{T\mu} \arrow{d}[swap]{\mu T}& T^2 \arrow{d}{\mu}\\
T^2 \arrow[swap]{r}{\mu} & T\\
\end{tikzcd},\hspace{20pt}
\begin{tikzcd}
\Id\circ T \arrow{r}{\zh T} \arrow{dr}[swap]{\Id}& T^2\arrow{d}{\mu} & \arrow{l}[swap]{T\eta} \arrow{dl}{\Id}T\circ \Id\\
&T\\
\end{tikzcd}.
\end{center}
\end{Defi}

We now give the functorial definition of an operad:

\begin{Defi} An {\bf operad} is monad in the category $\tt \mathbb{K}-Vect$ of $\mathbb{K}$-vector spaces.\end{Defi}

Below, we write $\tt Vect$ instead of $\tt \mathbb{K}-Vect$. In the classical definition, operads are defined as $\mathbb{S}$-modules $P$ equipped with an associative unital composition (of trees) (which is equivariant with respect to the $\mathbb{S}$-action). However, any $\mathbb{S}$-module $P$ gives rise to an endofunctor $\widetilde{P}: \texttt{Vect}\to \texttt{Vect}$, called {\bf Schur functor}, and defined on objects $V\in\tt Vect$ by
$$
\widetilde{P}(V) = \bigoplus_{n\in\N} P(n) \otimes_{\mathbb{S}_n} V^{\otimes n}\in\tt Vect\;,
$$ where the tensor product is over $\mathbb{K}[\mathbb{S}_n]\,.$ The map $\,\widetilde{-}\,$ from $\mathbb{S}$-modules to Schur functors is bijective and respects all operations. Hence, it is possible to identify $\mathbb{S}$-modules $P$ and Schur endofunctors $\widetilde{P}$. Now (functorially defined) operads $\langle \widetilde{P},\zm,\zh\rangle$ that are implemented by Schur functors, and, more generally, monads $\langle T,\zm,\zh\rangle$ (in Definition \ref{monad}) can be thought of as the realm of abstract operations $P$. The transformation $\zm$ `is' the composition of abstract operations, and the square diagram means that this composition is associative. When viewing $T$ as a Schur functor and similarly for $\Id=\Id_{\tt C}=\Id_{\tt Vect}$, we realize that $\Id:V\mapsto V$ corresponds to the $\mathbb{S}$-module $I(n)=0$ $(n\neq 1)$ and $I(1)=\mathbb{K}$. Hence, the transformation $\zh:\Id\to T$ must be identified with the $\mathbb{S}$-module morphism, which is just given by the $\mathbb{K}$-linear map $\zh_1:I(1)\to T(1)$ defined by $\zh_1(1)\in T(1)$. In view of the triangle diagrams, this means that the transformation $\zh$ `is' the unit abstract operation $\mathbf{1}\in T(1)$. {\it We thus understand that functorially defined operads induced by Schur functors are the same objects as classically defined operads}.

\begin{ex}\label{ExMonad} Any adjunction $\langle F,G,\zh,\ze\rangle$ between categories $\tt C$ and $\tt D$ defines a monad $T=GF:{\tt C}\to {\tt C}$, with composition $\zm=G\ze F:GFGF\to GF$ and unit $\zh:\Id_{\tt C}\to GF$.\end{ex}

Just as we defined algebras over operads, i.e., functorially, monads in $\tt Vect$, see Equation (\ref{repr}), we can define algebras over arbitrary monads.

\begin{Defi}If $T=\langle T,\mu,\zh\rangle$ is a monad in ${\tt C}$, a {\bf $T$-algebra} $\langle c, \zg\rangle$ consists of a $\tt C$-object $c$ and a $\tt C$-arrow $\zg:Tc\to c$, which make the following diagrams commutative:
\begin{center}
\begin{tikzcd}
T^2c \arrow{r}{T\zg} \arrow{d}[swap]{\mu_c}& Tc \arrow{d}{\zg}\\
Tc \arrow[swap]{r}{\zg} & c\\
\end{tikzcd}\hspace{20pt}
\begin{tikzcd}
c \arrow{r}{\eta_c} \arrow{dr}[swap]{\Id}& Tc\arrow{d}{\zg}\\
&c\\
\end{tikzcd}\quad .
\end{center}
\end{Defi}
\noindent In fact the object $c$ and the morphism $\zg:Tc\to c$ correspond to the vector space $V$ and the maps $$\zr_n:P(n)\0_{\mathbb{S}_n}V^{\0 n}\to V\;,$$ see Equation (\ref{repr}), hence, to the assignment of concrete operations to abstract ones. The square diagram encodes the information that `the composition of concrete maps' (upper and right parts of the diagram) and `the concrete map associated to abstract composition' (left and lower parts) coincide. The triangle diagram means that the abstract identity is sent to the concrete one.

\begin{ex}\label{Free} If $\langle\widetilde{P},\zm,\zh\rangle$ denotes an operad implemented by a Schur functor, the {\bf free $\widetilde{P}$-algebra} over $V$ -- in the categorical sense -- is the vector space $c=\widetilde{P}(V) = \bigoplus_{n\in \N} P(n)\otimes_{\mathbb{S}_n} V^{\otimes n}$, endowed with the $\mathbb{K}$-linear map $\gamma: \widetilde{P}(\widetilde{P}(V))\to \widetilde{P}(V)$ given by the component $\zm_V: (\widetilde{P}\circ \widetilde{P})(V)\to \widetilde{P}(V)$. In the case of the associative unital operad $u\Ac ss$, for instance, one actually finds the tensor algebra $\0(V)$ over $V$.\end{ex}

\begin{Defi}A {\bf $T$-algebra morphism} $f:\langle c',\zg'\rangle\to \langle c'',\zg''\rangle$ is a $\tt C$-arrow $f:c'\to c''\,$, which renders commutative the diagram
\begin{center}
\begin{tikzcd}
c'  \arrow{d}[swap]{f}&\arrow[swap]{l}{\zg'} Tc' \arrow{d}{Tf}\\
c''  & \arrow{l}{\zg''}Tc''\\
\end{tikzcd}\quad.
\end{center}\end{Defi}
\noindent For any underlying category $\tt C$, $T$-algebras and $T$-algebra morphisms form a category ${\tt C}^T$, which is often referred to as the {\bf Eilenberg-Moore category}. Further, as already mentioned, we have the

\begin{prop}\label{Equiv} Any type $\textup{\texttt{T}}$ of algebras (in the usual sense of the word), which satisfies the condition (\ref{MLR}), determines an operad $T$. The category ${\tt Vect}^T$ of $T$-algebras is equivalent to the category $\textup{\texttt{Alg}}$ of algebras of the considered type $\textup{\texttt{T}}$.\end{prop}

\begin{rem}{\em There are different ways to {\em construct the operad $T$ that encodes a given type $\tt T$ of algebras}. Often one considers the $\mathbb{S}$-module $\te{GO}$ spanned by the generating operations of type $\tt T$ and encrypts their symmetries into the $\mathbb{S}$-action. The defining relations span an $\mathbb{S}$-submodule $\te{DR}$ of $\te{GO}$, which in turn generates an operadic ideal $(\te{DR})$ of the free operad $\Fc (\te{GO})$ over $\te{GO}$. The operad $T$ is then given by \be\label{CM1} T=\Fc (\te{GO})/(\te{DR})\;.\ee On the other hand, according to Proposition \ref{Equiv} and Example \ref{Free}, the free algebra of type $\tt T$ over a vector space $V$, or the free $T$-algebra over $V$, is given by $T(V)$, where $T$ is the Schur functor \be\label{CM2} T=\widetilde{P}\ee associated to an $\mathbb{S}$-module $P$. This observation allows to determine this module and this functor, which is even an operad. For instance, in the case of associative unital algebras, we find $$\0(V)=\bigoplus_{k\in\N} P(k)\0_{\mathbb{K}[\mathbb{S}_k]}V^{\0 k}\;,$$ so that the associative unital operad is defined by $$u\Ac ss(k)=P(k)=\mathbb{K}[\mathbb{S}_k]\;,$$ for all $k\in\N$. It can be shown that the operads $T$ in (\ref{CM1}) and (\ref{CM2}) coincide.}\end{rem}

In this paragraph, we emphasize that the equivalence mentioned in Proposition \ref{Equiv} does not hold in more general situations. Let $F\dashv G$ be a Hom-set adjunction between two categories ${\tt C}$ and ${\tt D}$, or, equivalently a unit-counit adjunction $\langle F,G,\zh,\ze\rangle$: $$F:{\tt C}\rightleftarrows{\tt D}:G\;.$$ In view of Example \ref{ExMonad}, the latter defines a monad $\langle T,\zm,\zh\rangle=\langle GF,G\ze F,\zh\rangle$ over $\tt C$. Actually any monad is implemented by an (and even by many) adjunction(s). One of the inducing adjunctions of the monad $T$, beyond $\langle F,G,\zh,\ze\rangle$, is an adjunction $\langle F^T, G^T,\zh^T,\ze^T\rangle$ between the categories ${\tt C}$ and ${\tt C}^T$: $$F^T:{\tt C}\rightleftarrows{\tt C}^T:G^T\;.$$ It is defined by the free $T$-algebra functor $F^T(c)=\langle Tc,\zm_c\rangle$ (see Example \ref{Free}) and by the forgetful functor $G^T\langle c,\zg\rangle=c$. The unit and counit are $\zh^T=\zh$ and $\ze^T$ with components $\ze^T_{\langle c,\zg\rangle}=\zg$. Recall now that the original counit $\ze$ is a natural transformation $\ze:FG\to \Id_{\tt D}\,$. For any $\tt D$-object $d$, we thus get a $\tt C$-morphism $$G\ze_d:GFG(d)=T(G(d))\to G(d)\;.$$ This leads to a $T$-algebra $\langle G(d), G\ze_d\rangle$ and even to a functor $$K:{\tt D}\to {\tt C}^T\;,$$ called the {\bf comparison functor}. It is the unique functor between these categories, such that \be\label{Comparison2 }G^TK=G\quad\te{and}\quad KF=F^T\;.\ee  In other words, the following diagram commutes: \begin{center}
\begin{tikzcd}
{\tt D} \arrow{r}{K} \arrow{d}{G}& {\tt C}^T\arrow{d}{G^T}\\
{\tt C}\arrow{u}{F} \arrow{r}{=} & {\tt C}\arrow{u}{F^T}
\end{tikzcd}\quad.\end{center}\medskip
Let now ${\tt C}={\tt Vect}$ and ${\tt D}={\tt Alg}$ be the categories of $\mathbb{K}$-vector spaces and of associative unital $\mathbb{K}$-algebras, respectively. The free associative unital algebra functor $F=\0:{\tt C}\to {\tt D}$ is left adjoint to the forgetful functor $G=\te{For}:{\tt D}\to {\tt C}$, i.e., $$\te{Hom}_{\tt Alg}(\0 V,A)\simeq \te{Hom}_{\tt Vect}(V,\te{For}A)\;.$$ If we start from this adjunction $F\dashv G$, the comparison functor is a functor $$K:{\tt Alg}\to {\tt Vect}^T\;.$$ As mentioned above, this functor induces an equivalence of categories. When identifying isomorphic objects in the source and the target categories, we get isomorphic categories and when identifying the latter, $K$ becomes identity and $F=F^T, G=G^T$: the free object functors and the forgetful functors in $\tt Alg$ and ${\tt Vect}^T$ coincide. However, in general situations, the comparison functor yields an equivalence, if the right adjoint $G$ is nice enough, i.e., if it is {\bf monadic}. In the present paper, we use the so-called {\bf crude monadicity} theorem:

\begin{thm}\label{MonaTheo} A functor $G:{\tt D}\to {\tt C}$ is {\bf monadic}, if
\begin{enumerate}\item $G$ has a left adjoint,
\item $G$ reflects isomorphisms,
\item $\tt D$ has and $G$ preserves coequalizers of reflexive pairs.
\end{enumerate}
\end{thm}
\noindent Let us recall that a functor $G:{\tt D}\to {\tt C}$ {\bf reflects isomorphisms}, if a ${\tt D}$-morphism $m: d'\to d''$ is a $\tt D$-isomorphism whenever $G(m):G(d')\to G(d'')$ is a $\tt C$-isomorphism. Furthermore, a pair $f,g:d'\rightrightarrows d''$ is said to be a {\bf reflexive pair}, if $f$ and $g$ have a common section, i.e., if there is a morphism $h:d''\to d'$, such that $f\circ h=g\circ h= \Id_{d''}\,.$

\section{Internal Hom in modules over a commutative monoid}\label{IntHomMod}

Let $({\tt C}, \otimes,\te{\em I},\te{\underline{Hom}})$ be a closed symmetric monoidal category with all small limits and colimits.\medskip

In the main part of the present text, we reminded that the category $\tt Mod_{\tt C}(\Ac)$ of modules in $\tt C$ over a commutative monoid $\Ac$ in $\tt C$ is also a closed symmetric monoidal category with all small limits and colimits. However, we did not define the internal $\te{\underline{Hom}}_\Ac$ of this category of modules, at least not in the considered abstract setting. Remember first that $\tt C$ and $\tt Mod_{\tt C}(\Ac)$ are endowed with bifunctors $\te{ Hom}$ and $\te{ Hom}_\Ac$, and that $\tt C$ is in addition equipped with an internal $\te{\underline{Hom}}$. Before considering the internal $\underline{\te{Hom}}_\Ac$, recall still that a closed monoidal category $\tt C$ can be {\it equivalently defined} as a monoidal category together with, for any two objects $C'$ and $C''$, a $\tt C$-object $\te{\underline{Hom}}(C',C'')$ and a $\tt C$-morphism $$\te{ev}_{C',C''}:\te{\underline{Hom}}(C',C'')\otimes C'\to C''\;,$$ which are universal in the sense that, for every $\tt C$-object $X$ and $\tt C$-morphism $f:X\otimes C'\to C''$, there exists a unique $\tt C$-morphism $h:X\to\te{\underline{Hom}}(C',C'')$ such that $f=\te{ev}_{C',C''}\circ(h\otimes\te{id}_{C'})$. Indeed, if we start for instance from the usual definition, the existence of $\te{ev}_{C',C''}$ comes from $$\te{Hom}(X,\te{\underline{Hom}}(C',C''))\simeq \te{Hom}(X\otimes C',C'')\;,$$ when choosing $X=\te{\underline{Hom}}(C',C'')$. Moreover, if $h:X\to \te{\underline{Hom}}(C',C''),$ we get a $\tt C$-map $$f:=\te{ev}_{C',C''}\circ(h\otimes \te{id}_{C'}):X\otimes C'\to C''\;.$$ Conversely, if $f\in\te{Hom}(X\otimes C',C'')$, there exists a unique $h\in\te{Hom}(X,\te{\underline{Hom}}(C',C''))$, such that $\te{ev}_{C',C''}\circ(h\otimes \te{id}_{C'})=f.$ Now, if $M',M''\in\tt Mod_{\tt C}(\Ac)$, the $\tt Mod_{\tt C}(\Ac)$-object $\te{\underline{Hom}}_\Ac(M',M'')$ should be the kernel of the $`\,$map $$\te{\underline{Hom}}(M',M'')\ni f\mapsto f\circ\zm_{M'}-\zm_{M''}\circ (\te{id}_{\Ac}\otimes f)\in\te{\underline{Hom}}(\Ac\otimes M',M'')\,\text{'}\,.$$ To put this idea right, we consider the isomorphism $${\te{Hom}}(\te{\underline{Hom}}(M',M''),\te{\underline{Hom}}(\Ac\otimes M',M''))\simeq\te{Hom}(\te{\underline{Hom}}(M',M'')\otimes \Ac\otimes M',M'')\;,$$ and define the preceding kernel as the equalizer of the pair of parallel $\tt C$-arrows
$$\te{ev}_{M',M''}\circ(\te{id}\otimes \zm_{M'}),\zm_{M''}\circ(\te{id}_\Ac\otimes\te{ev}_{M',M''})\circ(\te{com}\otimes\te{id}_{M'})\in\te{Hom}(\te{\underline{Hom}}(M',M'')\otimes \Ac\otimes M',M'')\;.$$ This $\tt C$-object inherits an $\Ac$-module structure.

\section{Homotopy categories, derived categories, and derived functors}\label{HDCatDFun}

The interesting category is not the category $\tt Top$ of topological spaces, but the category of classes of topological spaces that have the same homotopy groups, i.e., roughly the same shape, although the are not necessarily homeomorphic. Therefore, we wish to view weak homotopy equivalences -- continuous maps that induce isomorphisms between homotopy groups -- as isomorphisms, i.e., we aim at introducing a new category in which weak homotopy equivalences become invertible. This is reminiscent of the localization of a commutative ring $R$ by a multiplicative subset $S\subset R$ -- which consists in the `best possible' construction of a commutative ring $R[S^{-1}]$ and a ring morphism $\ell: R \to R[S^{-1}]$, such that the images of $S$ become invertible elements in $R[S^{-1}]$. Similarly, it is possible to `localize' $\tt Top$ by the class $W$ of its weak homotopy equivalences, so that the elements of $W$ become invertible in the `localized category' ${\tt Top}[W^{-1}]$.\medskip

The exist various constructions of localized categories, of homotopy and derived categories, as well as of derived functors. It is actually difficult to find a structured approach that insists on differences between the most important of these notions and studies some relations between them carefully. We thought that a paper on a Homotopical Algebraic Context should contain a concise account on all this.

\subsection{Localization of a category with respect to a class of morphisms}

The {\bf localization of a category $\tt C$ by any collection of $\tt C$-morphisms $W$} is a universal pair, made of a category ${\tt C}[W^{-1}]$ and a functor $L:{\tt C} \to {\tt C}[W^{-1}]$, which sends the morphisms in $W$ to isomorphisms in ${\tt C}[W^{-1}]$. In fact one asks that the factorization resulting from the universal property be valid only up to natural isomorphism and that the functor $\bullet\circ L$ be fully faithful.\medskip

Specific examples of localization are, for instance, the {\bf localization of a model category $\tt C$ at its class $W$ of weak equivalences}, as well as, a bit more generally (recall that, in a model category, any isomorphism is a trivial cofibration and a trivial fibration (hence a weak equivalence), in view of the lifting property characterization), the {\bf localization at its weak equivalences $W$ of a category $\tt C$ {\it with} weak equivalences}, i.e., of a category $\tt C$ that comes equipped with a subcategory $W$, which contains all the isomorphisms of $\tt C$ and satisfies the 2 out of 3 axiom.\medskip

There exists a general construction for $({\tt C}[W^{-1}],L)$ (the objects of ${\tt C}[W^{-1}]$ are the same as those of $\tt C$ and $L$ is the identity on objects), but its result is not sufficiently handy. In the sequel, we often add conditions on $W$ or/and on $\tt C$. These then allow to avoid set-theoretical problems, to use less abstract constructions, and to give suitable detailed descriptions. \medskip

We describe now the localization of a category $\tt C$, in the case $W$ has `good' properties, i.e., is a {\bf multiplicative or localization system}, i.e., is a class of morphisms that satisfies four properties MS1 -- MS4 that we will not recall here \cite{KS}. We set $\op{Ob}({\tt C}[W^{-1}])=\op{Ob}({\tt C})$ and decide that, for any objects $X,Y$, a morphism $f\in \h_{{\tt C}[W^{-1}]}(X,Y)$ is an equivalence class $[(f',Y',t)]$ of triplets $(f',Y',t)$, $t\in W$, or, more precisely, of diagrams $${X\stackrel{f'}{\longrightarrow}Y'\stackrel{t}{\longleftarrow}Y,\quad t\in W}\;,$$ for the {relation} $(f', Y',t)\sim(f'', Y'',t')$ if and only if there exists $(f''',Y''',t'')$, as well as a commutative diagram
\begin{equation}
\xymatrix{& Y'\ar@{.>}[d] & \\
X \ar[ur]^{f'}\ar[dr]_{f''} \ar@{.>}[r]^{f'''} &Y''' & \ar@{.>}[l]_{t''}\ar[ul]_{t} \ar[dl]^{t'}Y\quad .\\
&Y''\ar@{.>}[u] &}
\end{equation}
If one defines $L:{\tt C}\to {\tt C}[W^{-1}]$ to be the functor given by $L(X)=X$ and, for $f:X\to Y$, by $L(f)=[(f,Y,\op{id}_Y)]$ (in view of MS1, we have $\op{id}_Y\in W$), then the pair $({\tt C}[W^{-1}],L)$ is the localization of $\tt C$ by $W$.\medskip

Let us still mention that, for an Abelian category $\tt A$, the category $\tt Ch(A)$ of chain complexes in $\tt A$ together with its class $W$ of quasi-isomorphisms, is a category with weak equivalences, but that $W$ is not a multiplicative system.

\subsection{Homotopy category of a model category}

Above we stressed the importance of the localization ${\tt Top}[W^{-1}]$ of $\tt Top$ at the class $W$ of its weak homotopy equivalences, i.e., at the class of its weak equivalences in its standard Quillen model structure.\medskip

More generally, if $\tt C$ is a model category or (only) a category with weak equivalences, the interesting category is the localization ${\tt C}[W^{-1}]$ of $\tt C$ at the class $W$ of weak equivalences. It is referred to as the {\bf homotopy category of the considered model category or category with weak equivalences}: \be\label{HoCatModCat}{\tt Ho}({\tt C})={\tt C}[W^{-1}]\;.\ee Actually the model category $\tt C$ models the more fundamental category $\tt Ho(C)$, which has the same objects as $\tt C$, but different morphisms. Explicit results on the homotopy category of a model category and on its morphisms can be found in \cite{Ho99}.

\subsection{Localization of additive and triangulated categories}

\newcommand{\stf}{\stackrel{f}{\longrightarrow}}
\newcommand{\stg}{\stackrel{g}{\longrightarrow}}
\newcommand{\sth}{\stackrel{h}{\longrightarrow}}
\newcommand{\stza}{\stackrel{\alpha}{\longrightarrow}}
\newcommand{\stzb}{\stackrel{\beta}{\longrightarrow}}

If $\tt C$ is an {\bf additive category} and $W$ a {\bf multiplicative system}, then the localized category ${\tt C}[W^{-1}]$ and the localization functor $L$ become additive as well, and they are universal among such additive pairs that send morphisms in $W$ to isomorphisms.\medskip

Localization of triangulated categories -- specific additive categories -- is still richer and it can be defined with respect to a class of objects!\medskip

Roughly, a {\bf triangulated category} or $\zD$-category is an additive category endowed with an invertible endofunctor $T$, called translation functor, and with a distinguished class of triangles (d.t.-s), i.e., of `triangular' sequences of morphisms $${X\stf Y\stg Z\sth TX}\;,$$ subject to six axioms TR0 -- TR5 that we will not describe \cite{KS}. The prototypical example of a $\zD$-category is the additive category $\mathbb{K}(\tt A)$ of chain complexes in an additive category $\tt A$ together with chain maps up to chain homotopies. The invertible endofunctor is the shift functor $[1]$ (with inverse $[-1]$) and the d.t.-s are the (triangles that are isomorphic to a) mapping cone triangles (triangle) $${X\stf Y\stg \op{Mc}(f)\sth X[1]}\;,$$where $f$ is any morphism and $g,h$ are the canonical ones (and where a morphism of triangles is made of four morphisms such that the resulting diagram commutes). The mentioned axioms TR0 - TR5 are actually the abstractions of the basic properties of the mapping cone triangles, which therefore satisfy of course these axioms.\medskip

It is possible to localize a $\zD$-category ${\tt D}$ with respect to a family of {\em objects}. More precisely, {\bf we localize a triangulated category with respect to a null system}, i.e., any subset $N$ of objects of ${\tt D}$, which satisfies three axioms NS1 -- NS3 that we do not describe \cite{KS}. To such a null system, we can associate a (right) multiplicative system $W$ made of those morphisms $f:X\to Y$ that can be extended to a d.t. with $Z\in N$. The idea is here that each morphism $f$ can be extended by a d.t. (axiom TR2) and that instead of looking at $f\in W$ we can look at $Z\in N$. The localization ${\tt D}/N$ of $\tt D$ with respect to $N$ is then the localization ${\tt D}[W^{-1}]$ of $\tt D$ with respect to the associated multiplicative $W$. The localization of a $\zD$-category with respect to a null system it is itself a $\zD$-category.

\subsection{Derived category of an Abelian category}

Let $\tt A$ be an additive category and denote by $\tt Ch(A)$ the category of chain complexes in $\tt A$ together with chain maps. The category $\mathbb{K}(\tt A)$ of chain complexes in $\tt A$ together with chain maps up to chain homotopies, is usually called the homotopy category of the additive category $\tt A$. In the case $\tt A$ is Abelian, this implies that those {\it chain maps, which are quasi-isomorphisms in $\tt Ch(A)$ for the good reason that they have an inverse up to chain homotopy, become isomorphisms in $\mathbb{K}(\tt A)$}. Note that the denomination homotopy category of $\tt A$ is not unambiguous, since the homotopy category ${\tt Ho}({\tt A})={\tt Ho(Ch(A))}$ of $\tt A$ or $\tt Ch(A)$ -- the weak equivalences $W$ of $\tt Ch(A)$ are {\em all the quasi-isomorphisms} -- has been defined above as the localization ${\tt Ch(A)}[W^{-1}]$.\medskip

The category $\mathbb{K}(A)$ is in some sense intermediate between the category $\tt Ch(A)$ and the {\bf derived category $\mathbb{D}(\tt A)$ of the Abelian category $\tt A$}. The latter is the category of chain complexes, but {\it all chain maps, which are quasi-isomorphisms in $\mathbb{K}(\tt A)$, become isomorphisms in $\mathbb{D}(\tt A)$.}\medskip

We will first construct this derived category $\mathbb{D}({\tt A})$ and then explain that, as suggested by this paragraph, it is actually equivalent to the homotopy category ${\tt Ho(Ch(A))}={\tt Ch(A)}[W^{-1}]$.\medskip

In fact the derived category $\mathbb{D}(\tt A)$ of an Abelian category $\tt A$ is the localization of the $\zD$-category $\mathbb{K}({\tt A})$ with respect to the null system
$$N({\tt A})=\{X\in\op{Ob}(\mathbb{K}({\tt A})): \Hc(X)\simeq 0\}\;,$$where $\Hc$
is the homology functor. The associated multiplicative system $\tilde W$ is then the class of morphisms $f:X\to
Y$ of $\mathbb{K}({\tt A})$, i.e., morphisms $f:X\to
Y$ of ${\tt Ch}({\tt A})$ considered up to chain homotopies,
such that $\Hc(\op{Mc}(f))\simeq 0$, or, still, such that $f$ is a quasi-isomorphism (this condition is of course independent of the chosen representative $f$). It is thus clear that in the localization \be\label{DerCatAbCat}\mathbb{D}({\tt A})=\mathbb{K}({\tt A})/N({\tt A})=\mathbb{K}({\tt A})[\tilde W^{-1}]\;,\ee the quasi-isomorphisms of ${\tt Ch}({\tt A})$ considered up to chain homotopies become isomorphisms. The same construction goes through for $\mathbb{D}^+({\tt A}),\mathbb{D}^-({\tt A}),\mathbb{D}^b({\tt A})$, whose objects are non-negatively graded, non-positively graded, and bounded chain complexes in $\tt A$, respectively.

\subsection{Relations}

The localization ${\tt Ho}({\tt Ch}({\tt A}))={\tt Ch}({\tt A})[W^{-1}]$ of the category ${\tt Ch}({\tt A})$ at its class $W$ of weak equivalences or quasi-isomorphisms is traditionally called the derived category of ${\tt A}$ (see nLab: `homotopy category'). However, in view of the preceding subsection, the derived category of $\tt A$ is also the localization $\mathbb{D}({\tt A})=\mathbb{K}({\tt A})[\tilde W^{-1}]$ of the $\zD$-category $\mathbb{K}({\tt A})$ at the multiplicative system $\tilde W$. The canonical composite functor $\Ic:{\tt Ch(A)}\to \mathbb{K}({\tt A})\to \mathbb{K}({\tt A})[\tilde W^{-1}]$ sends an element in $W$, i.e., a quasi-isomorphism, to an element in $\tilde W$, i.e., a quasi-isomorphism in $\tt Ch(A)$ up to chain homotopies, and finally to an isomorphism. The functor $\Ic$ thus factors through the homotopy category ${\tt Ch(A)}[W^{-1}]$, i.e., we get a functor $\tilde\Ic:{\tt Ch(A)}[W^{-1}]\to\mathbb{K}({\tt A})[\tilde W^{-1}]$. It can be proven \cite{MilDerCat} that $\tilde\Ic$ implements an equivalence of categories: $${\tt Ho(Ch(A))}={\tt Ch(A)}[W^{-1}]\simeq\mathbb{K}({\tt A})[\tilde W^{-1}]=\mathbb{D}({\tt A})\;.$$ In the case ${\tt A}={\tt Mod}(R)$, we get \be\label{DerHoDGRM}{\tt Ho}({{\tt DG}\text{\it\small R}\,{\tt M}})\simeq\mathbb{D}^+({\tt Mod}(R))\;.\ee

So far we considered the category ${{\tt DG}\text{\it\small R}\,{\tt M}}={{\tt DG}_+\text{\it\small R}\,{\tt M}}$ of differential non-negatively graded $R$-modules endowed with its standard weak equivalences or even with its standard {\bf projective model structure}. There exists a dual situation, i.e., an {\bf injective model structure} on the category ${{\tt DG}_-\text{\it\small R}\,{\tt M}}$ of differential non-positively graded $R$-modules. Its week equivalences are the quasi-isomorphisms (just as for the projective model structure), the cofibrations are the chain maps that are injective in each (strictly) negative degree (for the projective structure, they are the chain maps that are degree-wise injective with degree-wise $R$-projective cokernel), and the fibrations are the chain maps that are degree-wise surjective and have degree-wise an $R$-injective kernel (for the projective structure, they are the chain maps that are surjective in each (strictly) positive degree). Of course, \be\label{DerHoBoth}{\tt Ho}({{\tt DG}_+\text{\it\small R}\,{\tt M}}\,|_{\op{proj}})\simeq\mathbb{D}^+({\tt Mod}(R))\quad\text{and}\quad{\tt Ho}({{\tt DG}_-\text{\it\small R}\,{\tt M}}\,|_{\op{inj}})\simeq\mathbb{D}^-({\tt Mod}(R))\;.\ee

\subsection{Derivation of a functor whose source category is equipped with a class of morphisms} One defines left and right derived functors for a functor $F:\tt C\to D$, whose source $\tt C$ is endowed with a class $W$ of distinguished morphisms, even when $F$ does not transform morphisms in $W$ into isomorphisms. A {\bf right derived functor} (resp., a {\bf left derived functor}) of $F$ is then a functor $\mathbf{R}F:{\tt C}[W^{-1}]\to\tt D$ (resp., a functor $\mathbf{L}F:{\tt C}[W^{-1}]\to\tt D$) and a natural transformation $\eta:F \Rightarrow \mathbf{R}F\circ L_{\tt C}$ (resp., a natural transformation $\epsilon:\mathbf{L}F\circ L_{\tt C}\Rightarrow F$), such that the pair $(\mathbf{R}F,\eta)$ (resp., the pair $(\mathbf{L}F,\epsilon)$) be universal.\medskip

In the case $F$ does transform elements of $W$ into isomorphisms, it factors through the localized category, i.e., there is a functor $\mathbf{F}: {\tt C}[W^{-1}]\to\tt D$ such that $F\simeq \mathbf{F}\circ L_{\tt C}$. If we denote the latter natural isomorphism by $\eta:F\to\mathbf{F}\circ L_{\tt C}$ (resp., by $\epsilon: \mathbf{F}\circ L_{\tt C}\to F$), the pair $(\mathbf{F},\eta)$ (resp., the pair $(\mathbf{F},\epsilon)$) is the right derived functor $\mathbf{R}F$ (resp., the left derived functor $\mathbf{L}F$) of $F$. The factorization $\mathbf{F}$ can thus be interpreted as right {\it and} as left derived functor of $F$.

\subsection{Derived functors between categories with weak equivalences}

Let $$G:{\tt C}\to{\tt D}$$ be a {\bf functor of categories with weak equivalences}, i.e., a functor between categories with weak equivalences $W_{\tt C}$ and $W_{\tt D}$, such that $$G(W_{\tt C})\subset W_{\tt D}\;.$$ Since the localization functor $L_{\tt D}:{\tt D}\to {\tt D}[W_{\tt D}^{-1}]=\tt Ho(D)$ satisfies $L_{\tt D}(W_{\tt D})\subset\op{Isom}({\tt Ho(D)})$, the image of $W_{\tt C}$ by the composite functor $L_{\tt D}G:{\tt C}\to \tt Ho(D)$ is included in $(L_{\tt D}G)(W_{\tt C})\subset\op{Isom}({\tt Ho(D)})$, so that this functor factors through ${\tt C}[W_{\tt C}^{-1}]=\tt Ho(C)$: $$\mathbf{L_{\tt D}G}:\tt Ho(C)\to Ho(D)\;.$$ The functor $\mathbf{L_{\tt D}G}$ between homotopy categories is called the {\bf derived functor} of $G$.\medskip

In the case $\tt C$ and $\tt D$ are model categories -- specific categories with weak equivalences -- , the following holds. Let $$G:{\tt C}\rightleftarrows{\tt D}:H$$ be a {\bf Quillen adjunction}, i.e., an adjunction where $G$ respects Cof and $H$ respects Fib (or, equivalently, $G$ respects TrivCof and $H$ respects TrivFib, or $G$ respects Cof and TrivCof, or, still, $H$ respects Fib and TrivFib). Brown's lemma states roughly that, if a functor takes trivial cofibrations between cofibrant objects (resp., trivial fibrations between fibrant objects) to weak equivalences (weq-s), it takes all {weq-s} between cofibrant objects (resp., fibrant objects) to weq-s. It follows that the restriction $G:{\tt C}_c\to{\tt D}$ of $G$ to the full subcategory ${\tt C}_c$ of cofibrant objects (resp., that the restriction $H:{\tt D}_f\to{\tt C}$ of $H$ to the full subcategory ${\tt D}_f$ of fibrant objects) sends weq-s to weq-s. Since the cofibrant replacement functor $Q_{\tt C}$ (resp., the fibrant replacement functor $\mathcal{R}_{\tt D}$) respects weq-s, and since the functor $L_{\tt C}:{\tt C}\to {\tt C}[W^{-1}]={\tt Ho}({\tt C})$ (resp., $L_{\tt D}:{\tt D}\to {\tt D}[W^{-1}]={\tt Ho}({\tt D})$) sends weq-s to isomorphisms, the functor $L_{\tt D}GQ_{\tt C}:{\tt C}\to {\tt Ho}({\tt D})$ (resp., $L_{\tt C}H\mathcal{R}_{\tt D}:{\tt D}\to {\tt Ho}({\tt C})$) sends weq-s to isomorphisms. It thus factors through $\tt Ho(C)$ (resp., $\tt Ho(D)$), so that we get functors \be\label{DerFunModCat}\mathbf{L}(L_{\tt D}GQ_{\tt C}):{\tt Ho}({\tt C})\rightleftarrows{\tt Ho}({\tt D}):\mathbf{R}(L_{\tt C}H\mathcal{R}_{\tt D})\;\ee that we call {\bf left derived functor} of $G$ and {\bf right derived functor} of $H$, respectively. They are usually denoted simply by $\mathbf{L}G$ and $\mathbf{R}H$.

\subsection{Localization of a triangulated functor} We can {\bf localize with respect to a null system $N$ of $\tt D$ a triangulated functor} $F:{\tt D}\to {\tt D'}$ between triangulated categories, i.e., an additive functor that commutes with the translations and preserves the d.t.-s. The localization $F_N:{\tt D}/N\to{\tt D'}$ of a triangulated functor will also be triangulated. More precisely, to localize $F:{\tt D}\to{\tt D'}$ with respect to $N$, it suffices to localize its restriction $F:{\tt I}\to{\tt D'}$ to a full triangulated subcategory $\tt I$ of $\tt D$ with respect to the induced null system $N\cap \tt I$ of $\tt I$. However, the subcategory $\tt I$ must satisfy two conditions. First, any object of $\tt D$ must be connected to an object of ${\tt I}$ by a morphism in $W$ -- the multiplicative system associated to $N$ (this condition implies that the category ${\tt I}/N\cap\tt I$ is equivalent to the category ${\tt D}/N$: we denote the quasi-inverse functors by $i:{\tt I}/N\cap {\tt I}\to {\tt D}/N$ and $i^{-1}$). The second condition is that $F$ has to send any element of $N\cap \tt I$ to 0 (this condition implies that $F:\tt I\to\tt D'$ admits a localization $F_{N\cap {\tt I}}:{\tt I}/N\cap \tt I\to D'$). If these two conditions are satisfied, we say that the {\bf full triangulated subcategory $\tt I$ of $\tt D$} is {\bf $F$-injective} with respect to $N$. In this case, the right localization $F_N:{\tt D}/N\to\tt D'$ with respect to $N$ exists and is given by $$F_N=F_{N\cap \tt I}\circ i^{-1}:{\tt D}/N\to \tt D'\;.$$ The left localization of $F$ with respect to $N$ is defined dually via a {\bf full triangulated subcategory $\tt P$ of $\tt D$} that is {\bf $F$-projective} with respect to $N$.

\subsection{Derived functors between derived categories of Abelian categories}

Let $$G:{\tt A}\rightleftarrows{\tt B}:H$$ be {\bf right and left exact covariant functors of Abelian categories}.\medskip

We focus on the right localization of $H$; the left localization of $G$ is obtained dually. Due to its additivity, $H$ can be extended to any additive category. In particular, we have canonical functors $${\tt Ch}^-H:{\tt Ch^-(B)}\to {\tt Ch^-(A)}\quad\text{and}\quad\mathbb{K}^-H: \mathbb{K}^-({\tt B})\to \mathbb{K}^-({\tt A})\;,$$ which are obtained by just applying $H$ to chain complexes, maps, and homotopies. Since ${\tt Ch^-}H$ does in general not respect quasi-isomorphisms, it cannot be derived as functor between categories with weak equivalences. However, when post-composing $\mathbb{K}^-H$ with the localization functor $L_{\tt A}:\mathbb{K}^-({\tt A})\to \mathbb{D}^-({\tt A}),$ we get a triangulated functor $$\mathcal{K}^-H: \mathbb{K}^-({\tt B})\to \mathbb{D}^-({\tt A})\;$$ between $\zD$-categories. It can therefore be localized with respect to the null system $N^-({\tt B})$ of $\mathbb{K}^-({\tt B})$ made of all acyclic complexes. As mentioned before, it suffices to localize -- with respect to the induced null system $N^-({\tt B})\cap\mathbf{I}$ -- the restriction $\mathcal{K}^-H$ of $\mathcal{K}^-H$ to a $\mathcal{K}^-H$-injective full triangulated subcategory $\mathbf{I}$ of $\mathbb{K}^-({\tt B})$.\medskip

It is actually possible to deduce the injective full triangulated subcategory $\mathbf{I}$ of the source of the functor $\mathcal{K}^-H$, from an injective full additive subcategory $\tt I$ of the source of the functor $H$. A full additive subcategory $\tt I$ of $\tt B$ is called $H$-injective, if any object in $\tt B$ is related by a monomorphism to an object in $\tt I$, if the {\small RHS} object of a {\small SES} with {\small LHS} and central objects in $\tt I$, is in $\tt I$, and if $H$ transforms any {\small SES} with {\small LHS} object in $\tt I$ into a {\small SES}. It can quite easily been checked that, if $\tt I$ is an {\bf $H$-injective full additive subcategory} of $\tt B$, then $\mathbb{K}^-({\tt I})$ is a $\mathcal{K}^-H$-injective full triangulated subcategory of $\mathbb{K}^-({\tt B})$.\medskip

The localization of $\mathcal{K}^-H$ with respect to $N^-({\tt B})$ is called {\bf right derived functor} $$\mathbb{R}^-H:\mathbb{D}^-({\tt B})\to\mathbb{D}^-({\tt A})$$ of $H$. By definition, we have \be\label{DefDerFunAb}\mathbb{R}^-H=\mathcal{K}^-H_{\,N^-({\tt B})}=\mathcal{K}^-H_{\,N^-({\tt B})\,\cap\,\mathbb{K}^-({\tt I})}\circ i^{-1}\;,\ee where notation is the same as in the preceding subsection. Since the homology functor $\Hc_k:{\tt Ch}^-({\tt A})\to\tt A$, $k\le 0$, factors through the derived category, thus inducing a homology functor $\Hc_k:\mathbb{D}^-({\tt A})\to \tt A$, we obtain the {\bf $k$-th right derived functor} $$\mathbb{R}_k H:=\Hc_k\circ\mathbb{R}^-H:\mathbb{D}^-({\tt B})\to{\tt A}\;.$$ For any $Y\in \mathbb{D}^-({\tt B})$, i.e., $Y\in\mathbb{K}^-(\tt B)$, there exists $I\in \mathbb{K}^-(\tt I)$ and a quasi-isomorphism $Y\to I$ in $\mathbb{K}^-({\tt B})$, so that $Y\simeq I$ in $\mathbb{D}^-(\tt B)$ and $\mathbb{R}^-H\,(Y)\simeq \mathbb{R}^-H\,(I)$. It follows from the definition (\ref{DefDerFunAb}) that \be\label{DefDerFunAb2}\mathbb{R}^-H\,(Y)=\mathcal{K}^-H\,(I)=H(I)\;.\ee

{\bf If $\tt B$ has enough injectives}, the full subcategory $\tt I_{B}$ of all injective objects of $\tt B$ is a full additive subcategory that is injective for any left exact covariant functor \cite{KS}, in particular for $H$. Hence, $\mathbb{K}^-({\tt I_B})$ is a $\mathcal{K}^-H$-injective full triangulated subcategory of $\mathbb{K}^-(\tt B)$, so that, for any $Y\in\mathbb{D}^-({\tt B})$, we have, in view of (\ref{DefDerFunAb2}), \be\label{DefDerFunAb3}\mathbb{R}^-H\,(Y)=H(I_Y)\;,\ee where $I_Y\in\mathbb{K}^-({\tt I_B})$ is a non-positively graded chain complex made of injective objects of $\tt B$ that is quasi-isomorphic in $\mathbb{K}^-({\tt B})$ to $Y$, i.e., {\bf where $I_Y$ is an invective resolution of the chain complex $Y$ of $\tt B$}.

Dually, {\bf if $\tt A$ has enough projectives}, the full subcategory $\tt P_{A}$ of all projective objects of $\tt A$ is a full additive subcategory that is projective for any right exact covariant functor \cite{KS}, in particular for $G$. Hence, $\mathbb{K}^+({\tt P_A})$ is a $\mathcal{K}^+G$-projective full triangulated subcategory of $\mathbb{K}^+(\tt A)$, so that, for any $X\in\mathbb{D}^+({\tt A})$, we have $$\mathbb{L}^+G:\mathbb{D}^+({\tt A})\to\mathbb{D}^+({\tt B})$$ and \be\label{DefDerFunAb4}\mathbb{L}^+G\,(X)=G(P_X)\;,\ee where $P_X\in\mathbb{K}^+({\tt P_A})$ is a non-negatively graded chain complex made of projective objects of $\tt A$ that is quasi-isomorphic in $\mathbb{K}^+({\tt A})$ to $X$, i.e., {\bf where $P_X$ is a projective resolution of the chain complex $X$ of $\tt A$}.

\subsection{Relations}\hspace{10cm}

\vspace{1.5mm}

\noindent{\bf a. Relation to classical derived functors.} For any {\bf left exact covariant functor $H:\tt B\to \tt A$ between Abelian categories} (the case of a {right exact covariant functor between Abelian categories} is dual), {\bf whose source $\tt B$ has enough injectives} (dually, {whose source has enough projectives}), the {\bf classical right derived functors} $$R_kH:\tt B\to \tt A\;,$$ $k\le 0$, are defined on $Y\in\tt B$ using an injective resolution of the object $Y$, i.e., an exact sequence $$0\to Y\to I_0\to I_{-1}\to \ldots$$ made of injective objects $I_k$ (or, equivalently, an injective resolution $$I_\bullet: 0\to I_0\to I_{-1}\to \ldots\,$$ of $Y\in\tt B$ viewed as a non-positively graded chain complex concentrated in degree 0 and with zero differential). More precisely, one considers the complex $H(I_\bullet)$ and computes its homology $\mathcal H$ at any spot $k$: $$R_k H(Y)=\Hc_k(H(I_\bullet))\in\tt A\;.$$

When applying the results of the preceding subsection to the present situation, we (also) view $Y\in\tt B$ as a complex $Y_\bullet\in \mathbb{K}^-(\tt B)$ and note that the complex $I_\bullet\in\mathbb{K}^-(\tt I_{\tt B})$ is quasi-isomorphic to $Y_\bullet$. Hence, we get $$\mathbb{R}^-H(Y)=H(I_\bullet)\in\mathbb{D}^-({\tt A})\quad\text{and}\quad\mathbb{R}_kH(Y)=\Hc_k(\mathbb{R}^-H(Y))=\Hc_k(H(I_\bullet))=R_kH(Y)\;.$$ We thus recover the classical situation as a particular case of the derived functors constructed via localization of triangulated functors. The classical approach is interesting inter alia for practical computations.\medskip

\noindent{\bf b. Relation to model categorical derived functors.} Let now $R,S$ be unital rings and let $$G:{\tt A}:={\tt Mod}(R)\rightleftarrows{\tt B}:={\tt Mod}(S):H\;$$ be an adjunction of covariant additive functors. The latter can be extended to an adjunction $$G:{\tt C}:={\tt DG}_+\text{\it\small R}\,{\tt M}\,|_{\op{proj}}\rightleftarrows{\tt D}:={\tt DG}_+\text{\it\small S}\,{\tt M}\,|_{\op{prop}}:H\;,$$ where the fact that the canonically constructed adjoint of a chain map $G(X)\to Y$ (resp., $X\to H(Y)$) respects the differentials, comes from the naturality of the original adjunction. As {\em any} left (resp., right) adjoint, the functor $G$ (resp., $H$) is right (resp., left) exact and respects colimits (resp., limits), in particular cokernels (resp., kernels). Since categories of modules are Abelian and have enough projectives (resp., injectives), the result (\ref{DefDerFunAb4}) (resp., (\ref{DefDerFunAb3})) holds.

If we now endow the source and the target, as indicated, with the projective model structure, the extended adjunction is not necessarily a Quillen adjunction. Let us assume that $G$ (resp., $H$) respects injections (resp., surjections) (these conditions are automatically satisfied if the injections (resp., surjections) are split). Then $H$ respects fibrations and $G$ transforms cofibrations, i.e., degree-wise injective chain maps $\varphi_\bullet$ with degree-wise projective cokernels $\op{coker}(\varphi_i)$ into degree-wise injective chain maps $G(\varphi_\bullet)$ with cokernels $\op{coker}(G(\varphi_i))=G(\op{coker}(\varphi_i))$. The latter are projective. Indeed, a projective object $P$ can be defined as an object such that the covariant Hom-functor $\h(P,\bullet)$ transforms surjections into injections. The adjunction isomorphism shows that any surjection $s:N'\to N''$ is transformed into $$\h(G(\op{coker}(\varphi_i)),s)\simeq\h(\op{coker}(\varphi_i),H(s))\;,$$ where the {\small RHS} is an injection, since $H$ respects surjections. Hence, $G$ respects cofibrations, the extended adjunction is Quillen, and Equation (\ref{DerFunModCat}) is valid as well.

Note now that the sources $\mathbb{D}^+({\tt A})$ and $\op{Ho}({\tt C})$ (resp., the targets) of $\mathbb{L}^+G$ and $\mathbf{L}(L_{\tt D}GQ_{\tt C})$ coincide in view of (\ref{DerHoBoth}), and that their objects are just those of $\tt C$. Moreover, we already mentioned that an object in ${\tt C}:={\tt DG}_+\text{\it\small R}\,{\tt M}\,|_{\op{proj}}$ is cofibrant if and only if its terms are $R$-projective. It is thus clear that $Q_{\tt C}X$, $X\in{\tt C}$, is a projective resolution of $X$ and that $\mathbb{L}^+G(X)=G(Q_{\tt C}X).$ On the other hand, the localization functor $L_{\tt D}$ is the identity on objects and the localization $\mathbf{L}$, being the factorization through the homotopy category $\op{Ho}({\tt C})$, the localized and non-localized functors coincide on objects. Eventually, we have $$\mathbf{L}(L_{\tt D}GQ_{\tt C})(X)=G(Q_{\tt C}X)=\mathbb{L}^+G(X)\;.$$ When extending the original adjunction to an adjunction $$G:{\tt DG}_-\text{\it\small R}\,{\tt M}\,|_{\op{inj}}\rightleftarrows{\tt DG}_-\text{\it\small S}\,{\tt M}\,|_{\op{inj}}:H\;$$ and assuming again that $G$ (resp., $H$) respects injection (resp., surjections), we see similarly that the model categorical derived functor $\mathbf{R}(L_{\tt C}H\mathcal{R}_{\tt D})$ and the triangulated derived functor $\mathbb{R}^-H$ coincide.

\section{Universes}\label{AppendixUniverses}

It is well-known that the set of all sets is not a set but a proper class. In the following, we consider a pyramid of types of set. Start with some type of set on top of the pyramid and call it the 0-sets. Then the set of all 0-sets is not a 0-set, but a set of a next, more general, type, say a 1-set. Similarly, the set of all 1-sets is not a 1-set but a 2-set, and so on. Finally, the union of all types of set is the proper class of all sets.\medskip

The adequate formalization of the idea of set of all sets of a certain type is the notion of Grothendieck universe ($\ast$). A Grothendieck universe a (very large) set $U$, whose elements are sets and which is closed under all standard set-theoretical operations. More precisely \cite{SGA},

\begin{defi} A {\bf universe} is a set $U$ that satisfies the axioms: \begin{enumerate}\item if $x\in U$ and $y\in x$, then $y\in U$, \item if $x,y\in U$, then the set $\{x,y\}$ is an element of $U$, \item if $x\in U$, then the set $\mathcal{P}(x)$ of all subsets of $x$ is an element of $U$, \item if $I\in U$ and $x_i\in U$, for all $i\in I$, then $\bigcup_{i\in I}x_i\in U$, \item $\N\in U$.\end{enumerate}\end{defi}

The preceding axioms allow to prove many additional closure properties, but it is not impossible to leave a universe. The elements of $U$ are termed {\bf $U$-sets}. In particular, $U$ is the set of all $U$-sets (see ($\ast$)). As suggested above $U\notin U$, but there exists a pyramid of universes $U \in V \in W\in \ldots \,$, so that any element of $U$ is also an element of $V$ and of $W$, and so on. It is therefore natural to think about the union of all Grothendieck universes as the proper class of all sets. Moreover, this interpretation implies that any set belongs to some universe (Grothendieck's axiom).\medskip

We continue with a number of basic definitions.\medskip

A set $S$ is {\bf $U$-small}, if $S$ is isomorphic to a $U$-set (not all authors distinguish between $U$-set and $U$-small set). The category {\tt $U$-Set} is the category with objects all the $U$-sets and with morphisms all the maps between two $U$-sets. Both, the collection $\op{Ob}(U{\tt -Set})$ of objects and the collection $\op{Mor}(U{\tt -Set})$ of morphisms are sets, although no $U$-sets, but we can speak about the category {\tt $U$-Set} without having to pass to proper classes.\medskip

Moreover, a {\bf $U$-category} {\tt C}, or, better, a {\bf locally $U$-small category} {\tt C}, is a category such that, for any $c',c''\in {\tt C}$, the set $\h_{\tt C}(c',c'')$ is $U$-small. In \cite{SGA}, a category {\tt C} is viewed as the set $\op{Mor}({\tt C})$ of its arrows (containing the subset of identity arrows, i.e., the subset $\op{Ob}({\tt C})$ of objects). Hence, ${\tt C} \in U$ and $\tt C$ is $U$-small can be given the usual meanings. More precisely, if ${\tt C}\simeq \op{Mor}({\tt C})\in U$, then $\op{Ob}({\tt C})\in \mathcal{P}({\tt C})\in U$: for ${\tt C}\in U$, we have $\op{Ob}({\tt C})\in U$ and $\op{Mor}({\tt C})\in U$, i.e., objects and morphisms are $U$-sets. Similarly, if a category ${\tt C}\simeq \op{Mor}({\tt C})$ is $U$-small, it is easily seen that $\op{Ob}({\tt C})$ and $\op{Mor}({\tt C})$ are $U$-small sets. Let us stress that:

\begin{rem}{Contrarily to a $U$-set $S$, which is just a set $S\in U$, a $U$-category $\tt C$ is not a category ${\tt C}\in U$: A $U$-category is a locally $U$-small category in the above sense, whereas a category ${\tt C}\in U$ is a category such that $\op{Ob}({\tt C}),\op{Mor}({\tt C})\in U$. Note that, in view of what has been said above, any category $\tt C$ belongs to $U$, is $U$-small, and is locally $U$-small, for some universe $U$.
}\end{rem}

The necessity to change from a universe $V$ to a larger universe $W\ni V$ appears in particular when speaking about generalized spaces. If $\tt C$ denotes some category of spaces, its Yoneda dual category $${\tt C}\text{\Large\v{}}:={\tt Fun}({\tt C}^{\op{op}},{\tt Set})\;,$$ i.e., the category of contravariant $\tt Set$-valued functors defined on $\tt C$, or, still, the category of presheaves defined on $\tt C$, may be viewed as a category of generalized spaces. In our work, the category $${\tt S\,C}\text{\Large\v{}}_{\!\!V}:={\tt Fun}({\tt C}^{\op{op}},V{\tt -SSet})$$ of simplicial presheaves on $\tt C$ with respect to $V$ will play an important role. We start recalling some fundamental results \cite{SGA}:\medskip

\begin{prop}\label{SmallFunUniverse}Consider a universe $V$, two categories $\tt C, \tt D$, as well as the category ${\tt Fun}({\tt C},{\tt D})$ of functors from $\tt C$ to $\tt D$.
\begin{enumerate}
\item If ${\tt C}, {\tt D}\in V$ $(\,$resp., are $V$-small$\;)$, the category ${\tt Fun}({\tt C},{\tt D})$ is an element of $V$ $(\,$resp., is $V$-small$\;)$.
\item If $\tt C$ is $V$-small and $\tt D$ is a $V$-category, the category ${\tt Fun}({\tt C},{\tt D})$ is a $V$-category.
\item If $\tt C$ is $V$-small, the category ${\tt C}\text{\!\Large\v{}}_{\! V}$ is a $V$-category.
\item If $\tt C$ is a $V$-category, the category ${\tt C}\text{\!\Large\v{}}_{\! V}$ is not necessarily a $V$-category.\end{enumerate}
\end{prop}

\begin{rem} Usually people do not specify the universe in which they work, assuming implicitly that their constructions and results hold in {\em any universe} $V$. However, sometimes set-theoretical size issues force them to pass to a higher universe $W \ni V$. In this case, their theory is (considered as) valid in {\em any universes} $V \in W$. One says that the theory has been {\em universally quantified} over 1,2, or several universes, and one speaks about the {\bf universal polymorphism approach}. If the passage to higher universes is also implicit, one speaks about {\bf typical ambiguity}. However, this ambiguity, although often used and even sometimes recommended, can be dangerous \cite [Remarks 1.3.2 and 2.5.12]{SchAlgTop}.\end{rem}

In our paper, we start with the category ${\tt C} = {\tt DG\Dc M}$, which is locally $U$-small for some universe $U$ (it is clear that the categories ${\tt DG\Dc A}$ and ${\tt Mod}(\Ac)$ are also locally $U$-small). However, $\op{Ob}({\tt DG\Dc M})$ and $\op{Mor}({\tt DG\Dc M})$ can be sets that belong only to a higher universe $V \ni U$, so that ${\tt DG\Dc M}$ is then $V$-small (and the same holds for ${\tt DG\Dc A}$ and ${\tt Mod}(\Ac)$). Since ${\tt DG\Dc M}$ is $V$-small, the category $${\tt DG\Dc M}\text{\Large\v{}}_{\! V}={\tt Fun}({\tt DG\Dc M}^{\op{op}},V{\tt -Set})$$ is locally $V$-small (\ref{SmallFunUniverse}) and thus it is $W$-small for some higher universe $W \ni V$. When considering the $V$-small category ${\tt C} = {\tt DG\Dc A}^{\op{op}}$, we conclude that $${\tt S}\,{\tt DG\Dc A}^{\op{op}}\text{\Large\v{}}_{\! V}={\tt Fun}({\tt DG\Dc A},V{\tt -SSet})$$ is locally $V$-small and $W$-small \cite[Appendix A.1]{TV05}.\medskip

The preceding paragraph explains the idea behind the introduction of the three universes $U\in V\in W$ in \cite{TV08}. In the present paper, we work implicitly in an arbitrary universe $U$ that we need a priori not mention. However, since typical ambiguity can lead to problems, we mention explicitly the change of universe each time it is required. In fact, this is not necessary until we pass to simplicial presheaves.

\vfill
{\small GD, \emph{email:} {\sf gennaro.dibrino@gmail.com}; DP, \emph{email:} {\sf damjan.pistalo@uni.lu}; NP, \emph{email:} {\sf norbert.poncin@uni.lu}.}

\begin{thebibliography}{99}
\bibitem[AR94]{AR} {Ad{\'a}mek J and Rosick{\'y} J}, \emph{Locally presentable and accessible categories}, {London Mathematical Society Lecture Note Series}, {v. 189}, {CUP, Cambridge}, {1994}, {pp. xiv+316}.
\bibitem[SGA4-I]{SGA} Artin M, Grothendieck A, and Verdier J-L, \emph{Th\'eorie des topos et cohomologie \'etale des sch\'emas, tome 1, th\'eorie des topos}, \href{http://www.cmls.polytechnique.fr/perso/laszlo/sga4/SGA4-1/sga41.pdf}{http://www.cmls.polytechnique.fr/perso/laszlo/sga4/SGA4-1/sga41.pdf}.
\bibitem[BW85]{BW}{Barr M and Wells C}, \emph{Toposes, triples and theories}, {Grundlehren der Mathematischen Wissenschaften}, {v. 278}, {Springer-Verlag, New York}, {1985}, {pp. xiii+345}.
\bibitem[BD04]{BD04} {Beilinson A and Drinfeld V}, \emph{Chiral algebras}, {American Mathematical Society Colloquium Publications}, {v. 51}, {American Mathematical Society}, {Providence, RI}, {2004}, {pp. vi+375}.
\bibitem[CG15]{CG1} Costello K and Gwilliam O, \emph{Factorization Algebras in Quantum Field Theory. Volume 1}, to be published by Cambridge University Press, \href{http://math.northwestern.edu/~costello/factorization.pdf}{http://math.northwestern.edu/$\sim$costello/factorization.pdf}.
\bibitem[DPP15a]{DPP} Di Brino G, Pi\v{s}talo D, Poncin N, \emph{Model structure on differential graded commutative algebras over the ring of differential operators}, \href{http://arxiv.org/abs/1505.07720}{\texttt{arXiv:1505.07720 [math.AT]}} (2015).
\bibitem[DPP15b]{DPP2} Di Brino G, Pi\v{s}talo D, Poncin N, \emph{Model categorical Koszul-Tate resolution for algebras over differential operators}, \href{http://arxiv.org/abs/1505.07964}{\texttt{arXiv:1505.07964 [math.AT]}} (2015).
\bibitem[Dug01]{Du01} Dugger D, \emph{Combinatorial Model Categories Have Presentations}, {Adv. Math.}, {v. 164}, {2001}, {n. 1}, {pp. 177--201}.
\bibitem[DS95]{DS} {Dwyer W G and Spali{\'n}ski J}, \emph{Homotopy theories and model categories}, in {Handbook of algebraic topology}, {p. 73--126}, {North-Holland, Amsterdam}, {1995}.
\bibitem[FKS14]{FK} {Felder G and Kazhdan D}, \emph{The classical master equation}, {Perspectives in representation theory}, {Contemp. Math.}, {v. 610}, {pp. 79--137}, {With an appendix by Tomer M. Schlank}, {Amer. Math. Soc., Providence, RI}, {2014}.
\bibitem[Gil06]{Gil06} {Gillespie J}, \emph{The flat model structure on complexes of sheaves}, {Trans. Amer. Math. Soc.}, {v. 358}, {2006}, {n. 7}, {pp. 2855--2874}.
\bibitem[Gro61]{EGAII} {Grothendieck A}, \emph{\'{E}l\'ements de g\'eom\'etrie alg\'ebrique. {II}. \'{E}tude globale \'el\'ementaire de quelques classes de morphismes}, {Inst. Hautes \'Etudes Sci. Publ. Math.}, {\bf n. 8}, {1961}, {pp. 222}
\bibitem[Har77]{Ha}{Hartshorne R}, \emph{Algebraic geometry}, {Graduate Texts in Mathematics, No. 52}, {Springer-Verlag, New York-Heidelberg}, {1977}, {pp. xvi+496}.
\bibitem[Hir03]{Hir} {Hirschhorn P S}, \emph{Model categories and their localizations}, {Mathematical Surveys and Monographs}, {vol. 99}, {American Mathematical Society, Providence, RI}, {2003}, {pp. xvi+457}.
\bibitem[Hir15]{HirOU} {Hirschhorn P S}, \emph{Overcategories and undercategories of model categories}, \href{http://arxiv.org/abs/1507.01624}{\texttt{arXiv:1507.01624 [math.AT]}} (2015).
\bibitem[HTT08]{HTT} Hotta R, Takeuchi K, and Tanisaki T, {\em $\Dc$-Modules, Perverse Sheaves, and Representation Theory}, Progress in Mathematics, {\bf 236}, Birkh\"auser (2008).
\bibitem[Hov99]{Ho99}{Hovey M}, {\it Model categories}, {Mathematical Surveys and Monographs}, {v. 63}, {American Mathematical Society}, {Providence, RI}, {1999}, {pp. xii+209}.
\bibitem[Joy84]{Joy} Joyal A, {\it Letter to A. Grothendieck} (1984).
\bibitem[KS90]{KS} Kashiwara M and Schapira P, {\em Sheaves on Manifolds}, Grundlehren der mathematischen Wissenschaften, {\bf 292}, Springer-Verlag (1990).
\bibitem[Lan02]{Lang} {Lang S}, \emph{Algebra}, {Graduate Texts in Mathematics}, {v. 211}, {third Ed.}, {Springer-Verlag, New York}, {2002}, {pp. xvi+914}.
\bibitem[LH09]{LH} {Lipman J, Hashimoto M}, \emph{Foundations of Grothendieck Duality for Diagrams of schemes}, {Springer}, {2009}.
\bibitem[Mac98]{MacL} Mac Lane S, \emph{Categories for the working mathematician}, Graduate Texts in Mathematics 5, Second edition, Springer-Verlag, New York 1998, {pp. xii+314}.
\bibitem[Mat89]{Matsu} {Matsumura H}, \emph{Commutative ring theory}, {Cambridge Studies in Advanced Mathematics}, {v. 8}, {Second Ed.}, {Translated from the Japanese by M. Reid}, {Cambridge University Press, Cambridge}, {1989}, {pp. xiv+320}.
\bibitem[Mil86]{MilDMod} {Mili\u{c}i\'{c} D}, {\em Lectures on Algebraic Theory of $\Dc$-Modules}, Department of Mathematics, University of Utah, Salt Lake City, \href{https://www.math.utah.edu/~milicic/Eprints/dmodules.pdf}{\tt{https://www.math.utah.edu/~milicic/Eprints/dmodules.pdf}}.
\bibitem[MilDC]{MilDerCat} {Mili\u{c}i\'{c} D}, {\em Lectures on Derived Categories}, Department of Mathematics, University of Utah, Salt Lake City,
    \href{https://www.math.utah.edu/~milicic/Eprints/dercat.pdf}{\tt{https://www.math.utah.edu/~milicic/Eprints/dercat.pdf}}.
\bibitem[Pau11]{P11} {Paugam F}, \emph{Histories and observables in covariant field theory}, {J. Geom. Phys.} {61} (2011), {n. 9}, {pp. 1675--1702}.
\bibitem[Pau12]{Paugam1} Paugam F, \emph{Towards the mathematics of quantum Field theory}, Jussieu, France, 2012.
\bibitem[PP15]{PP} Pi\v{s}talo D and Poncin N, \emph{On four Koszul-Tate resolutions}, \href{ http://hdl.handle.net/10993/22858}{\texttt{ORBilu:hdl.handle.net/10993/22858}} (2015).
\bibitem[PP17]{PP1} Pi\v{s}talo D and Poncin N, \emph{Homotopical Algebraic Geometry over Differential Operators and Applications}, to appear.
\bibitem[Por08]{Porst} Porst H-E, {\em On Categories of Monoids, Comonoids, and Bimonoids}, Quaestiones Mathematicae 31 (2008), n. 2, pp. 127--139.
\bibitem[Ros07]{Rosicky} Rosick\'y J, {\em Accessible categories and homotopy theory}, {Contemporary Categorical Methods in Algebra and Topology} (2007).
\bibitem[Sch11]{SchAlgTop} Schapira P, {\em Algebra and Topology}, lecture notes,\\\href{https://webusers.imj-prg.fr/$\sim$pierre.schapira/lectnotes/AlTo.pdf}{\tt https://webusers.imj-prg.fr/~pierre.schapira/lectnotes/AlTo.pdf}.
\bibitem[Sch12]{Scha} Schapira P, {\em D-modules}, lecture notes,\\ \href{ http://www.math.jussieu.fr/$\sim$schapira/lectnotes/Dmod.pdf}{\tt{http://www.math.jussieu.fr/$\sim$schapira/lectnotes/Dmod.pdf}}.
\bibitem[Sch94]{Schn} Schneiders J-P, {\em An introduction to $\Dc$-modules}, Bulletin de la Soci\'et\'e Royale des Sciences de Li\`ege (1994).
\bibitem[SS98]{SS98} Schwede S and Shipley B E, \emph{Algebras and modules in monoidal model categories}, Proc. London Math. Soc. 80 (2000), n. 2, pp. 491--511.
\bibitem[TV05]{TV05}{To{\"e}n B and Vezzosi G}, \emph{Homotopical algebraic geometry. {I}. {T}opos theory}, {Adv. Math.}, {v. 193} (2005), {n. 2}, {pp. 257-372}.
\bibitem[TV08]{TV08} {To{\"e}n B and Vezzosi G}, \emph{Homotopical algebraic geometry. {II}. {G}eometric stacks and applications}, {Mem. Amer. Math. Soc.}, {v. 193} (2008), {n. 902}, {pp. x+224}.
\bibitem[Uen99]{Ueno} Ueno K, {\it Algebraic Geometry 1. From Algebraic Varieties to Schemes}, Transl. Math. Monographs, {\bf 185}, AMS, Providence, RI, 1999.
\bibitem[Vin01]{Vino} Vinogradov A M, {\it Cohomological analysis of Partial Differential Equations and Secondary Calculus}, Transl. Math. Monographs, {\bf 204}, AMS, Providence, RI, 2001. 
\bibitem[Web12]{Webb} Webb P, \emph{Triangulated and derived categories}, lecture notes,\\ \href{https://www.math.umn.edu/$\sim$webb/oldteaching/Year12-13/TriangDerivedCatNotes.pdf}{\tt {https://www.math.umn.edu/$\sim$webb/oldteaching/Year12-13/TriangDerivedCatNotes.pdf}}.
\bibitem[Wei94]{Wei} Weibel C A, \emph{An introduction to homological algebra}, {Cambridge Studies in Advanced Mathematics}, {vol. 38}, {Cambridge University Press, Cambridge}, {1994}, {pp. xiv+450}.
\end{thebibliography}
\end{document}